\documentclass[11pt,reqno]{amsart}

\usepackage{amsmath,amssymb, url,mathtools,color}
\usepackage{amsmath,amssymb,amsthm,mathtools}
\usepackage{tikz}
\usepackage{multirow}
\usepackage{enumitem}
\usepackage{float}
\usepackage{tikz-cd}
\usepackage{longtable}

\usepackage{amsmath,amssymb,amsthm,mathtools}
\usepackage[margin=3cm]{geometry}

\newtheorem{theorem}{Theorem}[section]

\newtheorem{lemma}[theorem]{Lemma}
\newtheorem{proposition}[theorem]{Proposition}
\newtheorem{corollary}[theorem]{Corollary}
\newtheorem{algorithm}[theorem]{Algorithm}

\theoremstyle{definition}
\newtheorem{definition}[theorem]{Definition}
\newtheorem{example}[theorem]{Example}

\newtheorem{remark}[theorem]{Remark}

\newcommand{\Z}{\mathbb{Z}}

\newcommand{\Q}{\mathbb{Q}}

\newcommand{\ff}{f}

\newcommand{\spec}{\operatorname{\textnormal{Spec}}}

\newcommand{\genlegendre}[4]{%
  \genfrac{(}{)}{}{#1}{#3}{#4}%
  \if\relax\detokenize{#2}\relax\else_{\!#2}\fi
}
\newcommand{\legendre}[3][]{\genlegendre{}{#1}{#2}{#3}}
\newcommand{\dlegendre}[3][]{\genlegendre{0}{#1}{#2}{#3}}

\date{\today}
\author{Frederick Saia}
\title{CM points on Shimura curves via QM-equivariant isogeny volcanoes}
\begin{document}

\begin{abstract}
We study CM points on the Shimura curves $X_0^D(N)_{/\mathbb{Q}}$ and $X_1^D(N)_{/\mathbb{Q}}$, parametrizing abelian surfaces with quaternionic multiplication and extra level structure. A description of the locus of points with CM by a specified order is obtained for general level, via an isogeny-volcano approach in analogy to work of Clark and Clark--Saia in the $D=1$ case of modular curves. This allows for a count of all points with CM by a specified order on such a curve, and a determination of all primitive residue fields and primitive degrees of such points on $X_0^D(N)_{/\mathbb{Q}}$. We leverage computations of least degrees towards the existence of sporadic CM points on $X_0^D(N)_{/\mathbb{Q}}$.
\end{abstract}

\maketitle
\setcounter{tocdepth}{1}
\tableofcontents


\section{Introduction}

The restriction of the study of torsion of elliptic curves over number fields to the case of complex multiplication (CM) has seen considerable recent progress. In particular, work of Clark and Clark--Saia \cite{Cl22,CS22}, continuing a program of research in this area from the perspective of CM points on modular curves (see, e.g., \cite{CCS13, BC20, CGPS22}), approaches the study of the CM locus on the modular curves $X_0(M,N)_{/\mathbb{Q}}$ and $X_1(M,N)_{/\mathbb{Q}}$ via a study of CM components of isogeny graphs of elliptic curves over $\overline{\mathbb{Q}}$. For $K$ an imaginary quadratic field and $\Delta = f^2\Delta_K$ the discriminant of the order $\mathfrak{o}(f)$ of conductor $f$ in $K$, let $j_{\Delta} \in X(1)_{/\mathbb{Q}}$ denote the closed point corresponding to elliptic curves with CM by the order of discriminant $\Delta$.  The work of \cite{Cl22, CS22} results, for instance, in a description of all points in the fiber of the natural map $X_0(M,N)_{/\mathbb{Q}} \rightarrow X(1)_{/\mathbb{Q}}$ over $j_{\Delta}$. This description provides the list of residue fields of $\Delta$-CM points on the first curve, along with a count of closed points in this fiber with each specified residue field. 

In this paper, we study the Shimura curves $X_0^D(N)_{/\mathbb{Q}}$ and $X_1^D(N)_{/\mathbb{Q}}$ parametrizing abelian surfaces with quaternionic multiplication (QM) by the indefinite quaternion algebra $B$ over $\mathbb{Q}$ of discriminant $D$, along with certain specified level structure. Our main result allows for a similar description of the CM loci on these curves. 

In particular, we show that if $x \in X_0^D(N)_{/\mathbb{Q}}$ has CM by the order $\mathfrak{o}(\ff) $ of conductor $\ff$ in the imaginary quadratic field $K$, then the residue field $\mathbb{Q}(x)$ is either a ring class field $K(\ff')$ for some $\ff'$ with $\ff \mid \ff' \mid Nf$, or is isomorphic to an index $2$ subfield of such a field $K(\ff')$. The ramification index of $x$ with respect to the natural map from $X_0^D(N)$ to $X_0^D(1)$ is always $1$ when the CM order has discriminant $\ff^2\Delta_K = \Delta < -4$. In general, this index is at most $3$. The work of this paper culminates in a determination of the residue fields and ramification indices of all CM points on $X_0^D(N)$, and putting together the casework based on the quaternion discriminant, level and CM order gives a result of the following form.

\begin{theorem}\label{algorithm_thm}
There exists an algorithm which, given as input
\begin{itemize}
\item an indefinite quaternion discriminant $D$ over $\mathbb{Q}$, 
\item a positive integer $N$ coprime to $D$ and 
\item an imaginary quadratic discriminant $\Delta = f^2\Delta_K$, 
\end{itemize}
returns as output the complete list of tuples $(\texttt{is{\_}fixed},f',e,c)$, consisting of 
\begin{itemize}
\item a boolean $\texttt{is{\_}fixed}$, 
\item a positive integer $f'$ (necessarily with $f \mid f'$), 
\item an integer $e \in \{1,2,3\}$ and
\item a positive integer $c$
\end{itemize}
such that there exist exactly $c$ closed $\mathfrak{o}$-CM points $x$ on $X_0^D(N)_{/\mathbb{Q}}$ with the following properties:
\begin{itemize} 
\item the residue field of $x$ over $K$ is $K(x) \cong K(f')$, the ring class field of conductor $f'$ associated to $K$,
\item $\mathbb{Q}(x) \cong K(f')$ if $\texttt{is{\_}fixed}$ is False,
\item $[K(f') : \mathbb{Q}(x)] = 2$ if $\texttt{is{\_}fixed}$ is True and
\item $x$ has ramification index $e$ with respect to the natural map to $X_0^D(1)_{/\mathbb{Q}}$. 
\end{itemize}
\end{theorem}

This algorithm, outlined in Algorithm \ref{algorithm}, has been implemented, and is publicly available at \cite{Rep} along with \cite{Magma} code for all other computations described in this paper. 

The outline towards developing this algorithm is as follows: in \S \ref{background_section} we provide relevant background and prior results on CM points on the Shimura curves of interest. This includes results on concrete decompositions of QM abelian surfaces with CM as products of CM elliptic curves. The main result here is Theorem \ref{fake-ell-curves-ell-curves-correspondence}. In \S \ref{isogenies_section} and \S \ref{volcanoes_section}, we then consider QM-equivariant isogenies and the QM-equivariant $\ell$-isogeny graph $\mathcal{G}^D_{\ell}$. We prove in Theorem \ref{volcano_thm} that a CM component of this graph for a prime $\ell$ and quaternion discriminant $D$ has the structure of an $\ell$-volcano for CM discriminant $\Delta < -4$. We handle the slight deviation from the structure of an $\ell$-volcano in the $\Delta \in \{-3,-4\}$ case in Proposition \ref{fom-prime-power}.

We study the action of $\text{Gal}\left(\overline{\mathbb{Q}}/\mathbb{Q}\right)$ on such components in \S \ref{Galois_action_section}, allowing for an enumeration of closed point equivalence classes of paths in these graphs and hence a description the CM locus on a prime-power level Shimura curve $X_0^D(\ell^a)_{/\mathbb{Q}}$ as provided in \S \ref{CM_points_prime_power_section}. The algebraic results of \S \ref{algebraic-results-section} then feed into a description of the CM locus on $X_0^D(N)_{/\mathbb{Q}}$ for general level $N$ coprime to $D$ provided in \S \ref{CM_points_level_N}, which provides the algorithm mentioned in Theorem \ref{algorithm_thm}. 

The ability to transition to information about the $\mathfrak{o}$-CM locus on $X_1^D(N)_{/\mathbb{Q}}$ is explained in \S \ref{X_1_section}, in which we prove the following result. While this does not determine the list of residue fields of CM points on $X_1^D(N)$ in the vein of Theorem \ref{algorithm_thm}, it allows us to count all CM points on $X_1^D(N)$ of specified degree and list their corresponding CM orders. Otherwise put, this is enough data to determine, for a fixed discriminant $D$ and degree $d$, all levels $N$ such that there exists a QM abelian surface $(A,\iota)$ and a torsion point $P \in A(\overline{\Q})$ of order $N$ such that the induced point $[A,\iota,P] \in X_1^D(N)$ has residue field of degree $d$. 

\begin{theorem}\label{inertness_thm_general}
Suppose that $x \in X_0^D(N)_{/\mathbb{Q}}$ is a point with CM by the imaginary quadratic order of discriminant $\Delta$. Let $\pi_1: X_1^D(N)_{/\mathbb{Q}} \rightarrow X_0^D(N)_{/\mathbb{Q}}$ and $\pi_0 : X_0^D(N)_{/\mathbb{Q}}$ denote the natural morphisms. The following hold:
\begin{enumerate}
\item The scheme-theoretic fiber of $\pi_1$ over $x$ consists of a single closed point. 
\item If any of the following hold:
\begin{itemize} 
\item $\Delta < -4$,
\item  $x$ is ramified with respect to $\pi_0$ or
\item  $N \leq 3$, 
\end{itemize}
then $\pi_1$ is unramified over $x$. 
\item If $N \geq 4$ and $x$ is unramified with respect to $\pi_0$, then in the $\Delta \in \{-3,-4\}$ case we have 
\[ e_{\pi_1}(x) = \begin{cases} 2 \quad &\textnormal{ if } \Delta = -4, \\ 3 \quad &\textnormal{ if } \Delta = -3, \end{cases} \quad \textnormal{ and } \quad f_{\pi_1}(x) = \begin{cases} \phi(N)/4 \quad &\textnormal{ if } \Delta = -4, \\ \phi(N)/6 \quad &\textnormal{ if } \Delta = -3, \end{cases} \]
for the ramification index and residue degree of $x$, respectively, with respect to $\pi_1$.
\end{enumerate}
\end{theorem}

We define a \textbf{primitive residue field} (respectively, a \textbf{primitive degree}) \textbf{of an $\mathfrak{o}$-CM point on $X_0^D(N)_{/\mathbb{Q}}$} to be one that does not properly contain (respectively, does not properly divide) that of another $\mathfrak{o}$-CM point on the same curve. Our work allows for a determination of all primitive residue fields and primitive degrees of $\mathfrak{o}$-CM points on $X_0^D(N)_{/\mathbb{Q}}$, as discussed in \S \ref{prim_res_flds_N}. An abridged version of our main result on primitive residue fields and degrees is as follows, with Theorem \ref{prim_res_flds_thm} providing the complete result:

\begin{theorem}
Suppose that $K$ splits $B$, let $\ff$ be a positive integer, and let $N$ be a positive integer relatively prime to $D$ with prime-power factorization $N = \ell_1^{a_1}\cdots \ell_r^{a_r}$. One of the following occurs:

\begin{enumerate}
\item There is a unique primitive residue field $L$ of $\mathfrak{o}(\ff) $-CM points on $X^D_0(N)_{/\mathbb{Q}}$, with $L$ an index $2$, totally complex subfield of a ring class field $K(Hf)$ for some $H \mid N$. 
\item There are exactly $2$ primitive residue fields of such points, with one of the same form as $L$ in part (1) and the other being a ring class field of the form $K(Cf)$ with $C < H$ and $C \mid N$. 
\end{enumerate}
\end{theorem}

Knowledge of all primitive degrees provides the ability to compute the \emph{least} degree $d_{\mathfrak{o},\text{CM}}(X_0^D(N))$ of an $\mathfrak{o}$-CM point on $X_0^D(N)_{/\mathbb{Q}}$ for any imaginary quadratic order $\mathfrak{o}$. In \S \ref{least-degrees-section}, we discuss minimizing over orders $\mathfrak{o}$ to compute the least degree $d_{\text{CM}}(X_0^D(N))$ of a CM point on $X_0^D(N)_{/\mathbb{Q}}$, and Proposition \ref{least_deg_X1_prop} allows one to transition from this to computations of least degrees of CM points on $X_1^D(N)_{/\mathbb{Q}}$. 

A closed point $x$ on a curve $X_{/\mathbb{Q}}$ is said to be \textbf{sporadic} if there are finitely many points $y$ on $X_{/\mathbb{Q}}$ with $\text{deg}(y) \leq \text{deg}(x)$. We apply our least degree computations towards the existence of sporadic CM points on $X_0^D(N)_{/\mathbb{Q}}$ with the following end result (see Theorem \ref{sp_CM_pts_thm}). 

\begin{theorem}
Let $\mathcal{F}$ be the set of all $393$ pairs $(D,N)$ appearing in Table \ref{table:no_sporadic_pts} or Table \ref{table:unknowns_table}. If $(D,N) \not \in \mathcal{F}$ consists of a quaternion discriminant $D>1$ over $\mathbb{Q}$ and a positive integer $N$ which is relatively prime to $D$, then $X_0^D(N)_{/\mathbb{Q}}$ has a sporadic CM point. If $(D,N)$ is such a pair with
\[ (D,N) \not \in \mathcal{F} \cup \{ (91, 5) \},  \]
then $X_1^D(N)_{/\mathbb{Q}}$ has a sporadic CM point. 
\end{theorem}

The appearance of the pair $(91,5)$ in this result comes down to the fact that while $X_0^{91}(5)_{/\mathbb{Q}}$ has a sporadic CM point of degree $2$, the curve $X_1^{91}(5)_{/\mathbb{Q}}$ has $4$ as the least degree of a CM point. See Theorem \ref{sp_CM_pts_thm} (4) for details. 

Our work determining residue fields of CM points on $X_0^D(N)_{/\mathbb{Q}}$ can be viewed as a generalization of prior work on the Diophantine arithmetic of Shimura curves via an alternate approach (specifically work of Jordan \cite{Jor81} and Gonz\'alez--Rotger \cite{GR06} -- see Theorem \ref{JGR}). Of course, our results are aimed towards better understanding the torsion of low-dimensional abelian varieties over number fields, via restriction to a case with extra structure. On this point, the question of which number fields admit abelian surfaces with certain specified rational torsion subgroups is closely related to our results, just as in the classical modular curve case. A result of Jordan (see Theorem \ref{Jordan_curves_to_surfaces}) clarifies this relationship. 

Unlike the modular curves $X_0(N)_{/\mathbb{Q}}$, the curves $X_0^D(N)_{/\mathbb{Q}}$ for $D>1$ have no cusps. For this reason, understanding the CM points on Shimura curves may be of even greater interest, as they provide the most accessible examples of low-degree points and could afford techniques (see, e.g., \cite{BT07}) for computing models in the absence of techniques involving expansions around cusps. 

Additionally, while our approach is in analogy to that of \cite{Cl22} and 
\cite{CS22} in the modular curve case, there are interesting deviations arising in this work due to technical differences in the $D>1$ case. Namely, while the field of moduli $\mathbb{Q}(x)$ of any CM point $x \in X(1)_{/\mathbb{Q}}$ has a real embedding, a result of Shimura \cite[Thm. 0]{Sh75} states that $X^D(1)_{/\mathbb{Q}}$ has no real points for $D>1$. This fact also opens the door for the potential of Hasse principle violations by Shimura curves, which has been a subject of significant study (see, e.g., \cite{Cl09, CSt18, RSY05, SS16}). If one aims to study the Hasse principle for Shimura curves over some fixed number field (respectively, over a fixed degree), then studying the CM points rational over that field (respectively, over number fields of that degree) seems to be a natural initial point of investigation, and so our results may be of interest in that direction. 

\subsection*{Acknowledgments}
We thank Pete L. Clark for initially suggesting this project, which served as the author's main Ph.D. thesis work. We are very grateful for Clark's excellent advising, and for innumerable helpful conversations. We also thank Dino Lorenzini, Oana Padurariu, Ciaran Schembri and John Voight for helpful comments and encouragement. Much gratitude goes as well to the anonymous referee for a careful reading of, and detailed feedback on, earlier versions of this work. This feedback led to many expository and mathematical improvements, including a strengthening of our results on sporadic points. 

We would also like to thank the Simons Foundation for access to the \cite{Magma} computational algebra system, with which all computations described in this paper were performed, as well as the Georgia Advanced Computing Resource Center at the University of Georgia for access to the high-performance computing cluster, which we used for various computations described in \S \ref{least-degrees-section}. Partial support for the author was provided by the Research and Training Group grant DMS-1344994 funded by the National Science Foundation.


\section{Background}\label{background_section}

\subsection{Shimura curves}

The main source here is the foundational work of Shimura \cite{Sh67}, while for the background material on quaternion algebras and quaternion orders we recommend the classic \cite{Vig80} as well as the modern treatment in \cite{Voi21}. Throughout, we let $B/\mathbb{Q}$ denote the indefinite quaternion algebra of discriminant $D$ over $\mathbb{Q}$. We denote by $\Psi$ an isomorphism 
\[ \Psi : B \otimes_{\mathbb{Q}} \mathbb{R} \xrightarrow{\sim} M_2(\mathbb{R}). \]
As $B$ is indefinite, the discriminant $D$ is the product of an even number of distinct rational primes, namely those at which $B$ is ramified. We will let $\mathcal{O}$ denote a maximal order in $B$, which is unique up to conjugation. We will also fix, following \cite[\S 43.1]{Voi21}, an element $\mu \in \mathcal{O}$, satisfying $\mu^2 = -D$, which induces the involution
\[ \alpha \mapsto \alpha^* := \mu^{-1} \alpha \mu \]
on $\mathcal{O}$. We refer to $\mu$ as a \textbf{principal polarization} on $\mathcal{O}$.

We start by defining the moduli spaces we are considering and discussing the moduli interpretations of those families of particular interest to us in this study. Let $\mathcal{O}^1$ denote the units of reduced norm $1$ in $\mathcal{O}$, which we realize as embedded in $\text{SL}_2(\mathbb{R})$ via $\Psi$. The subgroup $\Gamma^D(1) := \Psi(\mathcal{O}^1) \subset \text{SL}_2(\mathbb{R})$ is discrete, and it is cocompact if and only if $D>1$. Via the action of this subgroup on the upper-half plane $\mathbb{H}$ we define over $\mathbb{C}$ the Shimura curve
\[ X^D(1) := \Gamma^D(1)\backslash \mathbb{H}. \] 
For $D=1$ we have $B \cong M_2(\mathbb{Q})$, which recovers the familiar modular curve setting. We are interested in the $D>1$ case, and so moving forward we make this assumption on $D$. This impies that $X^D(1)$ is a compact Riemann surface. For any $z \in \mathbb{H}$, we get a rank $4$ lattice $\Lambda_z$ via the action of $\mathcal{O}$ on $(z,1) \in \mathbb{C}^2$ via the embedding $\Psi$ above:
\[ \Lambda_z := \mathcal{O} \cdot \begin{pmatrix} z \\ 1 \end{pmatrix} \subseteq \mathbb{C}^2. \]
From this we obtain a complex torus
\[ A_z := \mathbb{C}^2/\Lambda_z  \]
of dimension $2$, which comes equipped with an $\mathcal{O}$-action $\iota_z : \mathcal{O} \hookrightarrow \text{End}(A_z)$. We require some rigidification data, namely a Riemann form, in order to recognize $A_z$ as an abelian surface. It turns out that we always obtain such data in this setting (\cite[Lemma 43.6.23]{Voi21}); there is a \emph{unique} principal polarization $\lambda_{z,\mu}$ on $A_z$ such that the Rosati involution on $\text{End}^0(A) := \text{End}(A) \otimes \mathbb{Q}$ agrees with the involution induced by the polarization $\mu$ on $\Psi(\mathcal{O})$. 
\begin{definition}
An \textbf{$(\mathcal{O},\mu)$-QM abelian surface} over $F$ is a triple $(A,\iota,\lambda)$ consisting of an abelian surface $A$ over $F$, an embedding $\iota : \mathcal{O} \hookrightarrow \text{End}(A)$ which we will refer to as the \textbf{quaternionic multiplication (QM) structure}, and a polarization $\lambda$ on $A$ such that the following diagram is commutative 
\[
\begin{tikzcd}
B \arrow[r, "\iota"] \arrow[d,"^*"]  & \text{End}^0(A) \arrow[d,"\dag"] \\
B \arrow[r,"\iota"] & \text{End}^0(A)
\end{tikzcd}
\]
where $\dag$ denotes the Rosati involution corresponding to $\lambda$. An \textbf{isomorphism of QM-abelian surfaces} $(A,\iota,\lambda)$ and $(A',\iota',\lambda')$ is an isomorphism $f: A \rightarrow A'$ of abelian surfaces such that $f \circ \iota = \iota' \circ f$ and such that $f^* \lambda' = \lambda$. 
\end{definition}
\noindent With this definition, we have moreover (\cite[Main Thm. 43.6.14]{Voi21}) that $X^D(1)$ is the coarse moduli space of $(\mathcal{O},\mu)$-QM abelian surfaces over $\mathbb{C}$, with the association $z \mapsto [(A_z, \iota_z, \lambda_{z,\mu})].$ 

\begin{remark}
For an abelian variety $A$ over a field $F$, by $\text{End}(A)$ we mean the ring of endomorphisms defined over $F$. For an extension $F \subseteq L$, we will write $A_L$ for the base change of $A$ to $L$ and $\text{End}(A_L)$ for the ring of endomorphisms rational over $L$. 
\end{remark}

More generally, if $\Gamma \leq \Gamma^D(1) \subseteq \text{SL}_2(\mathbb{R})$ is an  arithmetic Fuchsian group, we can consider the curve $\Gamma \backslash \mathbb{H}$, and for $\Gamma' \leq \Gamma$ there is a corresponding covering of curves $\Gamma' \backslash \mathbb{H} \rightarrow \Gamma \backslash \mathbb{H}$. 
Our focus will be on the families of Shimura curves $X_0^D(N)$ and $X_1^D(N)$, for $N$ a positive integer with $\text{gcd}(D,N) = 1$, with $X^D(1) = X^D_0(1) = X^D_1(1)$ being a special case of each. 

With setup following the careful exposition of \cite[\S 1]{Buz97}, let
\[ R := \varprojlim_{\text{gcd}(m,D)=1} \mathbb{Z}/m\mathbb{Z} \]
and fix an isomorphism $\kappa : B \otimes_{\mathbb{Z}} R \rightarrow M_2(R)$. This map $\kappa$ induces, for $m$ relatively prime to $D$, a map
\[ \mathcal{O} \otimes \widehat{\mathbb{Z}} \rightarrow M_2\left(\mathbb{Z}_m\right). \]
We get from here a map 
\[ u_m : \mathcal{O}^1 \rightarrow \text{GL}_2(\mathbb{Z}_m). \]
The curve $X_0^D(N)$ can then be described as the Shimura curve corresponding to the compact, open subgroup
\[ \Gamma_0^D(N) := \Psi \left( u_N^{-1} \left( \left\{ \begin{pmatrix} a & b \\ c & d \end{pmatrix} \in \text{GL}_2(\mathbb{Z}_N) \mid c \equiv 0 \pmod{N} \right\} \right) \right) \leq \Gamma^D(1) . \]
That is, $X_0^D(N)(\mathbb{C}) = \Gamma^D_0(N) \backslash \mathbb{H}.$ Equivalently, fixing a level $N$ Eichler order $\mathcal{O}_N$ in $B$, the curve $X_0(N)$ can be described, in the manner mentioned above, as that associated to the arithmetic group of units of reduced norm $1$ in $\mathcal{O}_N$. The Shimura curve $X_1^D(N)$ corresponds to the compact, open subgroup
\[ \Gamma_1^D(N) := \Psi \left( u_N^{-1} \left( \left\{ \begin{pmatrix} a & b \\ c & d \end{pmatrix} \in \text{GL}_2(\mathbb{Z}_N) \mid c \equiv 0 \text{ and } d \equiv 1 \pmod{N} \right\} \right) \right) \leq \Gamma^D(1). \]

It follows from a celebrated result of Shimura \cite[Main Thm. I]{Sh67} that the curve $X^D_0(N)$ has a canonical model ${X^D_0(N)}_{/\mathbb{Q}}$, i.e., such that 
\[ {X^D_0(N)}_{/\mathbb{Q}} \otimes_{\mathbb{Q}} \mathbb{C} \cong X^D_0(N), \]
and similarly for the curve $X^D_1(N)$. 

Because we are assuming that $N$ is relatively prime to $D$, the notion of ``level $N$-structure'' is group-theoretically just as in the modular curve case. In particular, the natural modular map $X_1^D(N)_{/\mathbb{Q}} \rightarrow X_0^D(N)_{/\mathbb{Q}}$ is a $(\Z/N\Z)^\times/\{\pm 1\}$-cover. Hence, it is an isomorphism for $N \leq 2$ and it has degree $\phi(N)/2$ for $N \geq 3$, where $\phi$ denotes the Euler totient function. 
We now recall moduli interpretations for these families of Shimura curves as provided in, for example, \cite[\S 3]{Buz97}. 

\begin{definition}
Suppose that $(A,\iota,\lambda)$ and $(A',\iota',\lambda')$ are $(\mathcal{O},\mu)$-QM abelian surfaces over $F$, we will call an isogeny $\varphi : A \rightarrow A'$ of the underlying abelian surfaces a \textbf{QM-cyclic $N$-isogeny} if $\varphi^*(\lambda') = \lambda$ and both of the following hold:
\begin{itemize}
\item The isogeny $\varphi$ is QM-equivariant. That is, for all $\alpha \in \mathcal{O}$ we have
\[  \iota'(\alpha) \circ \varphi = \varphi \circ \iota(\alpha). \]
\item The kernel $\text{ker}(\varphi)$ is a cyclic $\mathcal{O}$-module with 
\[ \text{ker}(\varphi) \cong \mathbb{Z}/N\mathbb{Z} \times \mathbb{Z}/N\mathbb{Z}. 
\] 
\end{itemize}
\end{definition}

For example, a QM-cyclic $1$-isogeny is the same as an isomorphism of QM abelian surfaces. 

\begin{proposition}\label{moduli-interpretation}
The Shimura curve $X_0^D(N)_{/\mathbb{Q}}$ is isomorphic to the coarse moduli scheme associated to any of the following moduli problems:

\begin{enumerate}
    \item Tuples $(A,\iota,\lambda,Q)$, where $(A,\iota,\lambda)$ is an $(\mathcal{O},\mu)$-QM abelian surface 
    and $Q \leq A[N]$ is an order $N^2$ subgroup of the $N$-torsion subgroup of $A$ which is also a cyclic $\mathcal{O}$-module. 
    \item QM-cyclic $N$-isogenies $\varphi : (A,\iota,\lambda) \rightarrow (A,\iota',\lambda')$ of $(\mathcal{O},\mu)$-QM abelian surfaces. 
\end{enumerate}
The curve $X_1^D(N)_{/\mathbb{Q}}$ has the following moduli interpretation: triples $(A,\iota,\lambda,P)$, where $(A,\iota,\lambda)$ is a QM abelian surface and $P \in A[N]$ is a point of order $N$. 
\end{proposition}

These interpretations hold for any choice of principal polarization $\mu$ of $\mathcal{O}$. That is, if $\mu$ and $\mu'$ are two such polarizations then they both induce the same coarse moduli scheme $X_0^D(N)_{/\mathbb{Q}}$ up to isomorphism (as discussed, for example, in \cite[\S 6]{Rot04}). Of course, the exact moduli interpretation \emph{does} depend on $\mu$, and we refer to \cite[Prop. 4.3]{Rot04} for more on how the corresponding spaces fit into the moduli space of principally polarized abelian surfaces. Because a principal polarization $\lambda$ on a pair $(A,\iota)$ is canonically determined from a fixed $\mu$, moving forward we will suppress polarizations and refer simply to QM abelian surfaces $(A,\iota)$. By the same point, the condition on the polarizations in the definition of a QM-cyclic $N$-isogeny is redundant; it follows from the QM-equivariant condition. 

Letting $\mathcal{O}_N$ denote an Eichler order of level $N$ in $B$, the curve $X_0^D(N)_{/\mathbb{Q}}$ has the equivalent interpretation of parametrizing pairs $(A,\iota)$ where $A/\mathbb{C}$ is a QM abelian surface and $\iota : \mathcal{O}_N \hookrightarrow \text{End}(A)$. (We just stated that we would no longer remark on polarizations, but we note that the polarization corresponding to such an $\iota$ will not be principally polarized, but $(1,N)$-polarized in general.) That said, interpretations $(1)$ and $(2)$ in Proposition \ref{moduli-interpretation} will be the primary ones used in our study -- see Remark \ref{GR06_remark} for related comments. Thus, we will mainly speak of QM by maximal quaternion orders, and it will benefit us to spell out the connection between interpretations $(1)$ and $(2)$ here. Let $(A,\iota)$ be a QM abelian surface. The $N$-torsion of $A$ is acted on by $\iota(\mathcal{O})$, and the corresponding representation factors through $\mathcal{O} \otimes_{\mathbb{Z}} \mathbb{Z}/N\mathbb{Z} \cong M_2(\mathbb{Z}/N\mathbb{Z})$. The resulting map must then be equivalent to
\[ \begin{tikzcd}
    M_2(\mathbb{Z}/N\mathbb{Z}) \arrow[r] &  \textnormal{End}(A[N]) \cong  M_4(\mathbb{Z}/N\mathbb{Z}) \\
    \begin{pmatrix} a & b \\ c & d \end{pmatrix} \arrow[r,mapsto] &  \begin{pmatrix} a & 0 & b & 0 \\ 0 & a & 0 & b \\ c & 0 & d & 0 \\ 0 & c & 0 & d \end{pmatrix}.
\end{tikzcd} \]
\noindent This can be viewed as a case of Morita equivalence, but it is worth being explicit here: let $e_1$ and $e_2$ denote the standard idempotents in $M_2(\mathbb{Z}/N\mathbb{Z})$,
\[ e_1 = \begin{pmatrix} 1 & 0 \\ 0 & 0 \end{pmatrix}, \qquad e_2 = \begin{pmatrix} 0 & 0 \\ 0 & 1 \end{pmatrix}.\]
We then have $A[N] = e_1 \cdot A[N] \oplus e_2 \cdot A[N]$, and $M_2(\mathbb{Z}/N\mathbb{Z})$ acts on this direct sum in precisely the way noted by the above map.

Any proper, nontrivial, $\mathcal{O}$-stable subgroup $Q \leq A[N]$ must then have order $N^2$ (this justifies our definition of QM-cyclic isogenies, along with the equivalence of the moduli interpretations presented above). Further, such a subgroup $Q$ is determined by a cyclic order $N$ subgroup of $A[N]$: we have $Q = e_1(Q) \oplus e_2(Q)$ where each summand is cyclic of order $N$, and conversely $Q = \mathcal{O} \cdot e_i (Q)$ for $i = 1, 2$. 

For our applications in \S \ref{least-degrees-section}, the genera of our Shimura curves of interest will be of use. Let $\psi$ denote the Dedekind psi function. The derivations are standard -- for example, the formula for $X_0^D(N)$ can be found in \cite[Thm. 39.4.20]{Voi21}:

\begin{proposition}\label{genus_formula}
\begin{align*} 
g(X_0^D(N)) = 1 + \frac{\phi(D)\psi(N)}{12} - \frac{\epsilon_1(D,N)}{4} - \frac{\epsilon_3(D,N)}{3},
\end{align*}
where 
\begin{align*}
\epsilon_1(D,N) &= \begin{cases} \displaystyle{\prod_{p \mid D} \left(1 - \dlegendre{-4}{p}\right) \prod_{p \mid N} \left(1 + \dlegendre{-4}{p}\right)} \quad &\text{ if } 4 \nmid N \\ 
0 &\text{ if } 4 \mid N \end{cases} \\
\epsilon_3(D,N) &= \begin{cases} \displaystyle{\prod_{p \mid D} \left(1 - \dlegendre{-3}{p}\right) \prod_{p \mid N} \left(1 + \dlegendre{-3}{p}\right)} \quad &\text{ if } 9 \nmid N \\ 0 &\text{ if } 9 \mid N \end{cases}
\end{align*} 
are the numbers of elliptic $\mathbb{Z}[\sqrt{-1}]$-CM and elliptic $\mathbb{Z}\left[\frac{1+\sqrt{-3}}{2}\right]$-CM points on $X_0^D(N)$, respectively. For $N \leq 2$ we have $X_1^D(N) \cong X_0^D(N)$, and for $N \geq 3$ we have
\[ g(X^D_1(N)) = 1 + \dfrac{\phi(N)\phi(D)\psi(N)}{24}. \]
\end{proposition}

\subsection{CM points}\label{CM_points_section}
Let $(A,\iota)$ be a QM abelian surface over a number field $F$, such that
\[ \textnormal{End}^0(A) \cong B. \]
If $A$ is nonsimple, such that $A \sim E_1 \times E_2$ is geometrically isogenous (i.e., isogenous over $\overline{\mathbb{Q}}$) to a product of elliptic curves, then it must be the case that $E_1$ and $E_2$ are isogenous elliptic curves with complex multiplication (CM). In this case, $A \sim E^2$ where $E$ is a CM elliptic curve, say with corresponding imaginary quadratic CM field $K$. Here it is forced that $K$ splits the quaternion algebra $B$:
\[ B \otimes_{\mathbb{Q}} K \cong M_2(K). \]
In this case in which $A$ is nonsimple, we refer to $(A,\iota)$ as a \textbf{QM abelian surface with CM} and we call the induced point $[(A,\iota)] \in X^D(1)_{/\mathbb{Q}}(F)$ a \textbf{CM point}. We call a point $x$ on $X_0^D(D)_{/\mathbb{Q}}$ or $X_1^D(N)_{/\mathbb{Q}}$ a CM point if it lies over a CM point on $X^D(1)_{/\mathbb{Q}}$. 

Generalizing our definition for isogenies, we call an endomorphism $\alpha \in \text{End}(A)$ \textbf{QM-equivariant} if $\alpha \circ \iota(\gamma) = \iota(\gamma) \circ \alpha$ for all $\gamma \in \mathcal{O}$. If $(A, \iota)$ has $K$-CM, then the ring $\text{End}_{\text{QM}}(A)$ of QM-equivariant endomorphisms of $A$ is an imaginary quadratic order in $K$. This means that we have some $\ff \in \mathbb{Z}^+$ such that 
\[ \textnormal{End}_{\text{QM}}(A) \cong \mathfrak{o}(\ff) , \]
where $\mathfrak{o}(\ff) $ denotes the unique order of conductor $\ff$ in $K$. In other words, $\mathfrak{o}(\ff) $ is the unique imaginary quadratic order of discriminant $\ff^2\Delta_K$, where $\Delta_K$ denotes the discriminant of $K$, i.e., that of the maximal order $\mathfrak{o}_K = \mathfrak{o}(1)$. We will call this $\ff$ the \textbf{central conductor} of $(A,\iota)$. We will refer to $[(A,\iota)] \in X^D(1)$, or to any point in the fiber over $[A,\iota]$ under some covering of Shimura curves $X \rightarrow X^D(1)$, as an \textbf{$\mathfrak{o}(\ff) $-CM point} when we wish to make the CM order clear. Note that the QM on $A$ is by definition defined over $F$, so if $A$ is isogenous to $E^2$ over an extension $L/F$ then $E$ necessarily has its CM defined over $L$. 

\subsection{The field of moduli of a QM-cyclic isogeny}

\subsubsection{The field of moduli}
The \textbf{field of moduli} of a QM abelian surface $(A,\iota)$ defined over $\overline{\mathbb{Q}}$ is the fixed field of those automorphisms $\sigma \in \text{Gal}\left(\overline{\mathbb{Q}}/\mathbb{Q}\right)$ such that $(A,\iota)^\sigma := (A^\sigma,\iota^\sigma)$ is isomorphic to $(A,\iota)$ over $\overline{\mathbb{Q}}$. The conjugate abelian surface $A^\sigma$ is defined as the fiber product $A \otimes_{\spec  \overline{\mathbb{Q}}} \spec \overline{\mathbb{Q}}$ over $\sigma$:
\[ \begin{tikzcd}
A^{\sigma} \arrow[r] \arrow[d] & A \arrow[d] \\
\spec \overline{\mathbb{Q}} \arrow[r,"\sigma"] & \spec \overline{\mathbb{Q}} \; ,
\end{tikzcd} \]
and $\iota^{\sigma}$ is defined via the action of $\sigma$ on endomorphisms of $A$. (We are suppressing polarizations at this point, but recall this is justified as there is a unique principal polarization on $A^\sigma$ compatible with $\iota^\sigma$.) Equivalently, the field of moduli of $(A,\iota)$ is the residue field $\mathbb{Q}(x)$ of the corresponding point $x = [(A,\iota)]$ on $X^D(1)_{/\mathbb{Q}}$. 

More generally, for a QM-cyclic isogeny $\varphi : (A,\iota) \rightarrow (A',\iota')$ defined over $\overline{\mathbb{Q}}$, the \textbf{field of moduli} of $\varphi$ is the fixed field of the group
\[ H(\varphi) := \left\{ \sigma \in \text{Gal}\left(\overline{\mathbb{Q}}/\mathbb{Q}\right) \; \Bigg| \; \begin{tikzcd} (A^\sigma,\iota^\sigma) \arrow[r,"\varphi^\sigma"] \arrow[d] & ((A')^\sigma,(\iota')^\sigma) \arrow[d] \\ (A,\iota) \arrow[r,"\varphi"] & (A',\iota') \end{tikzcd} \substack{\text{ commutes, and the } \\ \text{vertical maps are} \\ \text{isomorphisms}} \right\}. \]

For clarity: the vertical maps above are those induced by $\sigma$, membership of $\sigma$ in $H(\varphi)$ means that both $(A,\iota)$ and $(A',\iota')$ are isomorphic to their conjugates by $\sigma$. In other words, the field of moduli of $\varphi$ is the minimal field over which $\varphi$ is isomorphic to all of its $\text{Gal}(\overline{\mathbb{Q}}/\mathbb{Q})$-conjugates. Equivalently, it is the residue field of the corresponding point $[\varphi]$ on $X_0^D(N)_{/\mathbb{Q}}$ (which follows from the much more general theory of \cite[Thm. 5.1]{Sh66}, as exposited more specifically towards our case in \cite[p. 60]{Sh67}). 

We call a field $F$ a \textbf{field of definition} for a QM-cyclic isogeny $\varphi$ as above, or say that $\varphi$ is defined or rational over $F$, if $\varphi$ and both $(A, \iota)$ and $(A', \iota')$ can be given by equations defined over $F$. We then have a model $\varphi'$ over $F$ so that $\varphi' \otimes_{F} \overline{\mathbb{Q}} = \varphi$. It follows that if $x \in X_0^D(N)_{/\mathbb{Q}}$ is induced by $\varphi$, then any field of definition for $\varphi$ contains the field of moduli $\mathbb{Q}(x)$.

It is not generally the case that fields of moduli are fields of definition for (polarized) abelian varieties of dimension bigger than $1$, and this is a source of difficulty and interest in the study of their arithmetic. For instance, Shimura proved that the generic principally polarized even-dimension abelian variety \emph{does not} have a model defined over its field of moduli \cite{Sh72}. Particular towards our interests here, a QM abelian surface (or, more generally, a QM-cyclic isogeny) need not have a model over its field of moduli. However, we have the following result of Jordan \cite[Thm. 2.1.3]{Jor81}:

\begin{theorem}[Jordan]\label{Jordan_curves_to_surfaces}
Suppose that $(A,\iota)/\overline{\mathbb{Q}}$ is a QM abelian surface with QM by $B$ and with $\textnormal{Aut}_{\textnormal{QM}}(A) = \{\pm 1\}$ (equivalently, $(A,\iota)$ does not have CM by $\Delta \in \{-3,-4\}$). Let $x = [(A,\iota)] \in X_0^D(1)_{/\mathbb{Q}}$ be the corresponding point. Then a field $L$ containing $\mathbb{Q}(x)$ is a field of definition for $(A,\iota)$ if and only if $L$ splits $B$. 
\end{theorem}

\subsubsection{The field of moduli in the CM case}

Our attention in this study will primarily be aimed at determining fields of moduli, particularly in the presence of CM. We now recall prior work determining the field of moduli of a CM point on $X^D(1)_{/\mathbb{Q}}$. 

The answer begins with a fundamental theorem of Shimura \cite[Main Thm. 1]{Sh67}. Fixing an imaginary quadratic field $K$ and a positive integer $\ff$, we let $\mathfrak{o}(\ff) $ denote the order in $K$ of conductor $\ff$ and $K(\ff)$ denote the ring class field corresponding to $\mathfrak{o}(\ff) $. 

\begin{theorem}[Shimura]\label{Shimura}
Let $x \in X^D(1)_{/\mathbb{Q}}$ be an $\mathfrak{o}(\ff) $-CM point with residue field $\mathbb{Q}(x)$. Then
\[ K \cdot \mathbb{Q}(x) = K(\ff)\] 
\end{theorem}

\noindent This tells us that in this setting there are two possibilities: either $\mathbb{Q}(x)$ is the ring class field $K(\ff)$, or it is an index 2 subfield thereof. In his thesis \cite[\S 3]{Jor81}, Jordan proved when each possibility occurs in the case where $x$ has CM by the maximal order of $K$ (the $\ff=1$ case). Work of Gonz\'{a}lez--Rotger allows for a generalization of Jordan's result to arbitrary CM orders \cite[\S 5]{GR06}. 

To state their result, we first set the following notation: for $D$ a quaternion discriminant over $\mathbb{Q}$ and $K$ an imaginary quadratic field splitting the quaternion algebra $B$ of discriminant $D$ over $\mathbb{Q}$, let
\[ D(K) := \prod_{p \mid D, \; \legendre{K}{p}=-1} p . \]
The assumption that $K$ splits $B$ is exactly the assumption that no prime divisor of $D$ splits in $K$. From this we see that $D(K) = 1$ if and only if all primes dividing $D$ ramify in $K$, while $D(K) > 1$ exactly when some prime dividing $D$ is inert in $K$. 

\begin{theorem}[Jordan, Gonz\'{a}lez--Rotger]\label{JGR}
Let $x \in X^D(1)_{/\mathbb{Q}}$ be an $\mathfrak{o}(\ff) $-CM point. 
\begin{enumerate}
    \item If $D(K) = 1$, then we have $\mathbb{Q}(x) = K(\ff)$.
    \item Otherwise, $[K(\ff) : \mathbb{Q}(x)] = 2$. In this case, $\mathbb{Q}(x) \subsetneq K(\ff)$ is the subfield fixed by 
    \[ \sigma = \tau \circ \sigma_\mathfrak{a} \in \textnormal{Gal}(K(\ff)/\mathbb{Q}), \]
    where $\tau$ denotes complex conjugation and $\sigma_\mathfrak{a} \in \textnormal{Gal}(K(\ff)/K)$ is the automorphism associated via the Artin map to a certain fractional ideal $\mathfrak{a}$ of $\mathfrak{o}(\ff) $ with the property that
    \[ B \cong \dlegendre{\Delta_K,N_{K/\mathbb{Q}}(\mathfrak{a})}{\mathbb{Q}} . \]
    More specifically, $\mathfrak{a}$ is such that
    \[ \omega_{D(K)}(x^{\sigma_\mathfrak{a}}) = \tau(x), \]
    where $\omega_{D(K)}$ denotes the Atkin--Lehner involution on $X^D(1)_{/\mathbb{Q}}$ corresponding to $D(K)$.
\end{enumerate}
\end{theorem}

\begin{remark}\label{GR06_remark}
In fact, Gonz\'{a}lez--Rotger provide a generalization of Jordan's result to all CM points on $X_0^D(N)_{/\mathbb{Q}}$ for squarefree $N$. We state their result only for trivial level $N=1$ in part because it is all we will need, but moreover because some translation would be needed for the statement of their result as in their work to the conventions of this work. In comparing our work to \cite{GR06}, the definition of an $\mathfrak{o}$-CM point on $X_0^D(N)_{/\mathbb{Q}}$ that they work with is different from ours; whereas our definition is that a CM point has $\mathfrak{o}$-CM for an imaginary quadratic order $\mathfrak{o}$ if it lies over an $\mathfrak{o}$-CM point on $X^D(1)_{/\mathbb{Q}}$, their definition is that $x \in X_0^D(N)_{/\mathbb{Q}}$ has $\mathfrak{o}$-CM if it corresponds to a normalized optimal embedding of $\mathfrak{o}$ into an Eichler order of level $N$ in $B$. The definition used in \cite{GR06} provides a pleasantly uniform result similar to Jordan's $N=1$ case, with every $\mathfrak{o}(\ff) $-CM point $x \in X_0^D(N)_{/\mathbb{Q}}$ having field of moduli $\mathbb{Q}(x)$ with $K \cdot \mathbb{Q}(x) \cong K(\ff)$. It will not be the case in our work, for level $N>1$, that all $\mathfrak{o}$-CM points have the same residue field. While our set of $K$-CM points on $X_0^D(N)$ is the same as that as defined in \cite{GR06}, the specific orders we attach may not agree. 

The convention used in  Gonz\'{a}lez--Rotger is common in the literature, appearing in the work of Rotger and his collaborators and also in recent work of Padurariu--Schembri \cite{PS23} in which the authors compute rational points on all Atkin--Lehner quotients of geometrically hyperelliptic Shimura curves. The difference in convention one takes is motivated by which moduli problem one chooses for the course moduli scheme $X_0^D(N)$: our choice of working with maximal orders results in having natural modular maps from $X_0^D(N)$ to $X_0^D(1)$ for all $N$, while working with Eichler orders of level $N$ naturally situates $X_0^D(N)$ as the base Shimura curve. Because we want to work with general level, we work with maximal orders.  A main difference between our work and that of \cite{GR06}, beyond the generalization from squarefree $N$ to all positive integers $N$, is that we consider not just the CM points on a fixed curve $X_0^D(N)$ but the \emph{fiber} of the covering $X_0^D(N)_{/\mathbb{Q}} \rightarrow X^D(1)_{/\mathbb{Q}}$ over any CM point. 
\end{remark}

\subsection{Decompositions of QM abelian surfaces with CM}

Restricting to the case of a QM abelian surface $(A,\iota)$ with CM over $\mathbb{C}$, we have seen that in fact $A$ is isogenous to a square of an elliptic curve with CM. Through a correspondence between QM abelian surfaces with CM and equivalance classes of certain binary quadratic forms, Shioda--Mitani \cite[Thm. 4.1]{SM74} proved the following strengthening of this fact: 

\begin{theorem}[Shioda--Mitani]\label{Shioda-Mitani}
If $(A,\iota)/\mathbb{C}$ is a QM abelian surface with $K$-CM for an imaginary quadratic field $K$, then there exist $K$-CM elliptic curves $E_1, E_2$ over $\mathbb{C}$ such that
\[ A \cong E_1 \times E_2. \]
\end{theorem}

The number of distinct decompositions of a given $A$ as above is finite, resulting from finiteness of the class number of any imaginary quadratic order in $K$. This theorem was generalized to higher dimensional complex abelian varieties isogenous to a power of a CM elliptic curve independently by Katsura \cite[Thm.]{Kat75} and Lange \cite{La75}, and Schoen later provided a simple proof as well \cite[Satz 2.4]{Sc92}. A generalization from $\mathbb{C}$ to an arbitrary field of definition $F$ is a result of Kani \cite[Thm. 2]{Kan11}:

\begin{theorem}[Kani]\label{Kani}
If $A/F$ is an abelian variety which is isogenous to $E^n$ over $F$, where $E/F$ is an elliptic curve with CM over $F$, then there exist CM elliptic curves $E_1/F,\ldots,E_n/F$ such that we have an isomorphism
\[ A \cong E_1 \times \cdots E_n \]
over the base field $F$. 
\end{theorem}
\noindent Kani in fact says more, which is relevant in the case of QM abelian surfaces with CM \cite[Thm. 67]{Kan11}: fixing a $K$-CM elliptic curve $E/F$ with endomorphism ring of conductor $\ff_E$, there is a bijection between the set of $F$-isomorphism classes $[E']$ of elliptic curves $E'$ isogenous to $E$ with CM conductor $f_{E'} \mid f_E$, and the set of $F$-isomorphism classes of abelian surfaces $A/F$ isogenous to $E^2$ with corresponding central conductor $f_A = f_E$. Explicitly, this bijection sends an $F$-isomorphism class $[E']$ to the $F$-isomorphism class $[E \times E']$. 

In order to obtain concrete decompositions of QM abelian surfaces with CM, the remaining task is to identify \emph{which} such products of CM elliptic curves have potential quaternionic multiplication (that is, which can be given QM structures), and to further describe the classes of QM abelian surfaces with CM. The following result provides the number of such classes (\cite[Thm. 6.13]{AB04} interprets this count as a certain class number, or equivalently as an embedding number, and \cite[Cor. 5.12]{Vig80} provides a formula for these class numbers which we use in the $N=1$ case). 

\begin{proposition}\label{number-of-CM-points}
Let $K$ be an imaginary quadratic field splitting $B$, and let $\ff \in \mathbb{Z}^+$. Let $b$ denote the number of primes dividing $D$ that are inert in $K$. The number of geometric $\mathfrak{o}(\ff) $-CM points on $X^D(1)$ is then $2^b \cdot h(\mathfrak{o}(\ff) )$, where $h(\mathfrak{o}(\ff) )$ denotes the class number of the order $\mathfrak{o}(\ff) $. 
\end{proposition}

In his thesis, Ufer touches on this topic of taking QM structures into account. In particular, he proves the following \cite[Thm. 2.7.12]{Uf10}: with the notation of Proposition \ref{number-of-CM-points}, there exists a $2^b$-to-$1$ correspondence 
\[ \left\{K\text{-CM points in } X^D(1)(\mathbb{C} \right\} \longrightarrow \left\{K\text{-CM elliptic curves over } \mathbb{C} \right\}/\cong. \]
\noindent Based on the proof therein, it seems that Ufer could have said more, and so we do that here with reference to his argument. As above, let $b$ denote the number of primes dividing $D$ which are inert in $K$. 

\begin{theorem}\label{fake-ell-curves-ell-curves-correspondence}
Let $(A,\iota)/\mathbb{C}$ be a QM abelian surface with CM by $\mathfrak{o}(\ff) $. There is then a unique $\mathfrak{o}(\ff) $-CM curve $E_A/\mathbb{C}$, up to isomorphism, such that
\[ A \cong \mathbb{C}/\mathfrak{o}(\ff)  \times E_A. \]
Additionally, there is a $2^b$-to-$1$ correspondence 
\[ \left\{\mathfrak{o}(\ff) \text{-CM points on } X^D(1) \right\} \longrightarrow \left\{\mathfrak{o}(\ff) \text{-CM elliptic curves over } \mathbb{C} \right\}/\cong \]
sending a point $[(A, \iota)] \in X^D(1)$ to the class of $E_A$. 
\end{theorem}

\begin{proof}
Part (2) of the proof of \cite[Thm. 2.7.12]{Uf10} details the construction of a QM-structure by a maximal order $\mathcal{O}$ in $B$ on $E \times E'$ for $E$ and $E'$ both $\mathfrak{o}(\ff) $-CM elliptic curves. The product $E \times E'$ with the constructed QM structure then corresponds to a CM point on $X^D(1)$ with central conductor $\ff$. 

Let $E, E'$ be $K$-CM elliptic curves. Part (3) of Ufer's proof explains that if the abelian surface $E \times E'$ has potential quaternionic multiplication then in fact it has $2^b$ non-isomorphic QM structures. Put differently but equivalently to therein: let $W$ be the group generated by the Atkin--Lehner involutions $\omega_p$ on $X^D(1)$ for $p \mid D$ inert in $K$. The group $W \times \textnormal{Pic}(\mathfrak{o}(\ff) )$ then acts simply transitively on the set of $\mathfrak{o}(\ff) $-CM points on $X^D(1)$. If $[(A,\iota)] \in X^D(1)$ is such a point, then the action of any element $w \in W$ leaves $[A]$ unchanged, providing the claim (this is proved by Jordan \cite{Jor81} in the $\ff=1$ case, and extended to the general case by Gonz\'{a}lez--Rotger \cite[Proposition 5.6]{GR06}). By the count of Proposition \ref{number-of-CM-points}, Theorem \ref{Shioda-Mitani} and the fact that $\mathbb{C}/\mathfrak{o}(\ff)  \times E \cong \mathbb{C}/\mathfrak{o}(\ff)  \times E'$ implies $E \cong E'$, the claimed result follows. 
\end{proof}

\begin{corollary}
Let $(A,\iota)/F$ be a QM abelian surface with CM by $\mathfrak{o} \subseteq \mathfrak{o}_K$. Suppose that we have an $F$-rational isogeny $A \sim E^2$ to the square of an elliptic curve. Fix $E_1/F$ any elliptic curve with $\mathfrak{o}$-CM. There then exists an $\mathfrak{o}$-CM elliptic curve $E_2/F$, unique up to isomorphism over $F$, such that $A \cong E_1 \times E_2$ over $F$. 
\end{corollary}

\begin{proof}
Let $f$ be the central conductor of $A$ (i.e., such that $\mathfrak{o} = \mathfrak{o}(f)$). By Theorem \ref{Kani} and the discussion of Kani's results following this theorem statement, there exists a CM elliptic curve $E_2/F$, with endomorphism ring of conductor $f_{E_2}$ satisfying $f_{E_2} \mid f$, such that $A \cong E_1 \times E_2$ over $F$. This curve $E_2$ is unique up to isomorphism over $F$. Base changing this entire picture to $\mathbb{C}$, we have
\[ A_{/\mathbb{C}} \cong {E_1}_{/\mathbb{C}} \times {E_2}_{/\mathbb{C}}. \] 
Now because $A_{/\mathbb{C}}$ (and hence $ {E_1}_{/\mathbb{C}} \times {E_2}_{/\mathbb{C}}$, by transport of structure through our isomorphism) has QM and ${E_1}_{/\mathbb{C}}$ has CM conductor $f$, Theorem \ref{fake-ell-curves-ell-curves-correspondence} implies that $f_{E_2} = f$ as well. 
\end{proof}

\section{QM-equivariant isogenies}\label{isogenies_section}

Our goal in the following section will be to determine the residue field of a CM point on $X_0^D(N)_{/\mathbb{Q}}$ for any $N$ coprime to $D$, generalizing Theorem \ref{JGR}. A main component in accomplishing this is the study of the structure of, and the action of automorphisms on, components of certain isogeny graphs. Paths in these graphs of consideration will be in correspondence with isogenies of QM abelian surfaces which commute with their QM structures. 

In this section, we prove facts about QM-equivariant isogenies needed in the proceeding section. Much of what we do in both this section and the next is in strong analogy to the case of isogenies of elliptic curves over $\overline{\mathbb{Q}}$ studied in recent work of Clark and Clark--Saia \cite{Cl22, CS22}. We provide proofs here for completeness and for clarity of said analogy.

\begin{lemma}\label{edge_count}
Let $F$ be a field of characteristic zero, and let $(A,\iota)$ be a QM abelian surface over $F$ which does not have CM by an order of discriminant $\Delta \in \{-3,-4\}$. For $\ell$ a prime number, the number of QM-cyclic $\ell$-isogenies with domain $(A,\iota)$ which are $\text{Gal}(\overline{F}/F)$-stable, up to isomorphism, is either $0,1,2,$ or $\ell+1$. 
\end{lemma}

\begin{proof}
Note that $\ell$ being prime means we are counting isomorphism classes of QM-cyclic $\ell$-isogenies. The hypotheses on $A$ are equivalent to $\textnormal{Aut}(A,\iota) = \{\pm 1\}.$ In this case, we have a bijective correspondence between isomorphism classes of QM-cyclic $\ell$-isogenies and non-trivial, proper cyclic $\mathcal{O}$-submodules of $A[\ell]$. Under this correspondence, the isogenies which are $\text{Gal}(\overline{F}/F)$-stable correspond to $\text{Gal}(\overline{F}/F)$-stable submodules.

Now we have that $e_1(Q) \leq e_1(A[\ell]) \cong \left(\mathbb{Z}/\ell\mathbb{Z}\right)^2$ is a cyclic subgroup of order $\ell$, and in this way we have a bijective correspondence between the non-trivial proper QM-stable subgroups of $A[\ell]$ and cyclic order $\ell$ subgroups of $e_1(A[\ell])$. This correspondence preserves the property of being $\text{Gal}(\overline{F}/F)$-stable. We have thus reduced to the situation of the elliptic curve case, and may proceed as such: we are counting $\text{Gal}(\overline{F}/F)$-stable cyclic order $\ell$ subgroups of $\left(\mathbb{Z}/\ell\mathbb{Z}\right)^2$. The total number of cyclic order $\ell$ subgroups is $\ell$+1, and if more than $2$ such subgroups are fixed then $\text{Gal}(\overline{F}/F)$ is forced to act by scalar matrices on $\left(\mathbb{Z}/\ell \mathbb{Z}\right)^2$. 
\end{proof}

\subsection{Compositions of QM-cyclic isogenies} 
The following result is in analogy with \cite[Prop. 3.2]{Cl22}.

\begin{proposition}\label{compositions}
Suppose that $\varphi = \varphi_2 \circ \varphi_1$ is a QM-cyclic isogeny, where $\varphi_i : (A_i,\iota_i) \rightarrow (A_{i+1},\iota_{i+1})$ is a QM-cyclic isogeny for $i=1,2$. 
\begin{enumerate}
\item We have 
\[ \mathbb{Q}(\varphi) \subseteq \mathbb{Q}(\varphi_1) \cdot \mathbb{Q}(\varphi_2). \]
\item If $(A_2,\iota_2)$ does not have \text{CM} by $\Delta \in \{-3,-4\}$, then 
\[ \mathbb{Q}(\varphi) = \mathbb{Q}(\varphi_1) \cdot \mathbb{Q}(\varphi_2). \]
\end{enumerate}
\end{proposition}

\begin{proof}
The containment of part (1) is clear. The assumption that $(A_2,\iota_2)$ does not have $-3$ or $-4$ CM is equivalent to $\text{Aut}((A_2,\iota_2)) = \{\pm 1\}$, and in this case the reverse containment in part (2) follows by the same argument as in \cite[Prop. 3.2]{Cl22}. 
\end{proof}

\subsection{Reduction to prime power degrees}

First, let us say something about rationality. Let $\varphi : (A,\iota) \rightarrow (A',\iota')$ be a QM-cyclic $N$-isogeny which is rational over $F$, where $N$ has prime-power decomposition $N = \ell_1^{a_1}\cdots \ell_r^{a_r}$. Letting $Q = \text{ker}(\varphi)$ be the kernel of this isogeny, we have that $\varphi$ is isomorphic to the quotient $(A,\iota) \rightarrow (A/Q,\iota)$. (The latter pair indeed provides an $\mathcal{O}$-QM abelian surface, as $Q$ is stable under $\iota(\mathcal{O})$ and $\mathcal{O}$ is maximal, though
we are abusing notation by referring to the QM-structure on the quotient as $\iota$.) We have a decompositon $Q = C \oplus D$ with each of $C$ and $D$ cyclic of order $N$, such that $\mathcal{O}\cdot C = \mathcal{O}\cdot D = Q$. This cyclic subgroup $C$ then decomposes as 
\[ C = \bigoplus_{i=1}^r C_i \]
where $C_i \leq C$ is the unique subgroup of order $\ell_i^{a_i}$. Letting $Q_i = \mathcal{O}\cdot C_i$, each $Q_i$ is QM stable and isomorphic to $(\mathbb{Z}/{\ell_i}^{a_i}\mathbb{Z})^2$. 

From the uniqueness of $C_i \leq C$, and hence of the corresponding $\mathcal{O}$-cyclic subgroup $Q_i \leq Q$, we get that each $Q_i$ is $F$-rational, resulting in $F$-rational QM-cyclic $\ell_i^{a_i}$-isogenies $\varphi_i : (A,\iota) \rightarrow (A/Q_i,\iota)$ for each $i$. 
On the other hand, given a collection of $F$-rational QM-cyclic $\ell_i^{a_i}$-isogenies with kernels $Q_i$, we get an $F$-rational QM-cyclic $N$-isogeny $(A,\iota) \rightarrow (A/Q,\iota)$ where $Q = \bigoplus_{i=1}^{r} Q_i$. 

As for fields of moduli, more towards our needs for the following section, we have the following:

\begin{proposition}\label{red_to_prime_powers}
Let $N_1,\ldots, N_r \in \mathbb{Z}^+$ be pairwise coprime, let $k$ be a field of characteristic $0$, and let $x \in X^D(1)_{/k}$ be a closed point which does not have CM by discriminant $\Delta \in \{-3,-4\}$. For each $i$, let $\pi_i : X_0^D(N_i)_{/k} \rightarrow X^D(1)_{/k}$ be the natural map, let $F_i = \pi_i^{-1}(x)$, and let $F$ be the fiber over $x$ of $\pi: X_0^D(N)_{/k} \rightarrow X^D(1)_{/k}$ where $N = N_1\cdots N_r$. Then 
\[ F = F_1 \otimes_{\spec k(x)} \cdots \otimes_{\spec k(x)} F_r . \]
\end{proposition}
\begin{proof}
This follows as in the $D=1$ case of \cite[Prop. 3.5]{Cl22}, using that $X_0^D(N)$ for $D>1$ is a cover of $X^D(1)$ with the same corresponding subgroup of $\text{GL}_2(\mathbb{Z}/N\mathbb{Z})/\{\pm 1\}$ as in the case of $X_0(N) \rightarrow X(1)$. 
\end{proof}

\noindent It follows that if $x \in X_0^D(N)_{/\mathbb{Q}}$ is a point which does not have $-3$ or $-4$-CM and $N = \prod_{i=1}^r \ell_i^{a_i}$, with $\pi_i : X_0^D(N)_{/\mathbb{Q}} \rightarrow X_0^D(\ell_i^{a_i})_{/\mathbb{Q}}$ the natural maps, then 
\[ \mathbb{Q}(x) = \mathbb{Q}(\pi_1(x)) \cdots \mathbb{Q}(\pi_r(x)). \]


\section{QM-equivariant isogeny volcanoes}\label{volcanoes_section}

Fixing a prime $\ell$, we describe in this section CM components of $\ell$-isogeny graphs of QM abelian surfaces over $\overline{\mathbb{Q}}$. We will use the work of this section to study CM points on the curves $X_0^D(\ell^a)_{/\mathbb{Q}}$ for $a \in \mathbb{Z}^+$ and $D>1$, in analogy to the $D=1$ modular curve case of \cite{Cl22, CS22}. 

This study, like that of \cite{Cl22, CS22}, is motivated by the foundational work on isogeny volcanoes over finite fields by Kohel in his PhD thesis \cite{Koh96} and by Fouquet \cite{Fou01} and Fouquet--Morain \cite{FM02}. We also recommend, and will refer to, a more recent, expository account of isogeny volcanoes in the finite field setting by Sutherland \cite{Sut13}.  

\subsection{The isogeny graph of QM abelian surfaces} 

Fix a prime number $\ell$ and an imaginary quadratic field $K$. In \cite{Cl22} and \cite{CS22}, the authors consider the multigraph with vertex set that of $j$-invariants of $K$-CM elliptic curves, and with edges corresponding to $\mathbb{C}$-isomorphism classes of cyclic $\ell$-isogenies.

Here, we seek an analogue for abelian surfaces with QM by a fixed maximal order $\mathcal{O}$ of the indefinite quaternion algebra $B$ of discriminant $D$ over $\mathbb{Q}$, with $\ell \nmid D$. We let $\mathcal{G}_{\ell}^D$ denote the directed multigraph with 
\begin{itemize}
\item vertex set consisting of $\mathbb{C}$-isomorphism classes of $\mathcal{O}$-QM abelian surfaces, and 
\item edges from $v_1 = [(A_1,\iota_1)]$ to $v_2 = [(A_2,\iota_2)]$ corresponding to $\mathbb{C}$-isomorphism classes of QM-cyclic $\ell$-isogenies $\varphi : (A_1,\iota_1) \rightarrow (A_2,\iota_2)$. 
\end{itemize}
A given vertex $v$ has $\ell+1$ edges emanating from it, via the correspondence of QM-stable subgroups of $A_1[\ell]$ with cyclic order $\ell$ subgroups of $e_1(A_1[\ell]) \cong \left(\mathbb{Z}/\ell\mathbb{Z}\right)^2$ discussed in Lemma \ref{edge_count}.

Because a QM structure $\iota$ determines a unique principal polarization, we have dual edges via dual isogenies as in the elliptic curve case. As long as the source vertex $v_1$ corresponds to an isomorphism class $[(A,\iota)]$ having only the single non-trivial automorphism $[-1]$, we obtain a bijection between the edges from $v_1$ to $v_2$ and those from $v_2$ to $v_1$; in this case, outward edges from $v_1$ are in bijective correspondence with QM-stable subgroups of $A_1[\ell]$ of order $\ell^2$. This occurs precisely when $[(A,\iota)]$ does not have CM by discriminant $\Delta = -3$ or $\Delta = -4$. 

Our attention will be to vertices in $\mathcal{G}^D_\ell$ corresponding to QM abelian surfaces with CM. For an abelian variety $(A,\iota)$ with QM by $\mathcal{O}$ and $K$-CM, recall from \S \ref{CM_points_section} that the central conductor of $(A,\iota)$ is defined to be the positive integer $\ff$ such that $\textnormal{End}_{\text{QM}}(A) \cong \mathfrak{o}(\ff)  \subseteq \mathfrak{o}_K$. 

\begin{lemma}
Suppose $\varphi : (A,\iota) \rightarrow (A',\iota')$ is a QM cyclic $N$-isogeny, with $(A,\iota)$ a QM abelian surface with $K$-CM. Then
\begin{enumerate}
\item The QM abelian surface $(A',\iota')$ also has $K$-CM.
\item Let $\ff$ and $\ff'$ denote the central conductors of $(A,\iota)$ and $(A',\iota')$, respectively. Then $\ff$ and $\ff'$ differ by at most a factor of $N$:
\[\ff\mid N\ff' \text{ and } \ff' \mid Nf. \]
\end{enumerate}
\end{lemma}

\begin{proof}
The argument is similar to that of the elliptic curve case. In our context, we need only remember that we care specifically about those endomorphisms commuting with the QM. 

Consider the homomorphism
\begin{align*}
    F: \textnormal{End}(A,\iota) &\longrightarrow \textnormal{End}(A',\iota') \\
    \psi &\longmapsto \varphi \circ \psi \circ \hat{\varphi}. 
\end{align*}
Because $\varphi$ is assumed to be QM-equivariant, this restricts to a homomorphism
\[  \textnormal{End}_{\text{QM}}(A,\iota) \longrightarrow \textnormal{End}_{\text{QM}}(A',\iota') . \]
As in the argument in the elliptic curves case, the algebras of endomorphisms commuting with the quaternionic multiplication are isomorphic by the multiple $\frac{1}{N} F$ of the map above. That is,
\[ K \cong \textnormal{End}_{\text{QM}}(A,\iota) \otimes \mathbb{Q} \cong \textnormal{End}_{\text{QM}}(A',\iota') \otimes \mathbb{Q}. \]
This completes part (a). Moreover, that 
\[ \frac{1}{N} F : \textnormal{End}_{\text{QM}}(A,\iota) \otimes \mathbb{Q} \rightarrow \textnormal{End}_{\text{QM}}(A',\iota') \] 
is an isomorphism tells us that
\[ N \cdot \textnormal{End}_{\text{QM}}(A,\iota) \subseteq \textnormal{End}_{\text{QM}}(A',\iota'), \]
yielding $\ff' \mid N f$. Via the dual argument, we obtain $\ff \mid N \ff'$. 
\end{proof}

For an imaginary quadratic field $K$, we are therefore justified in defining $\mathcal{G}^D_{K,\ell}$ to be the subgraph of $\mathcal{G}^D_\ell$ consisting of vertices corresponding to QM abelian surfaces with $K$-CM. An edge in $\mathcal{G}^D_{K,\ell}$ corresponds to a class of QM-cyclic $\ell$-isogenies $[\varphi: (A,\iota) \rightarrow (A',\iota')]$ between QM abelian surfaces with $K$-CM, and the above lemma tells us that as we move along paths in $\mathcal{G}^D_{K,\ell}$, the central conductors of vertices met have the same prime-to-$\ell$ part. It follows that $\mathcal{G}^D_{K,\ell}$ has a decomposition
\[ \mathcal{G}^D_{K,\ell} = \bigsqcup_{(\ff_0,\ell) = 1} \mathcal{G}^D_{K,\ell,\ff_0}, \]
where $\mathcal{G}^D_{K,\ell,\ff_0}$ denotes the subgraph of $\mathcal{G}^D_{K,\ell}$ with vertices having corresponding central conductors of the form $\ff_0\ell^a$ for some $a \in \mathbb{N}$. 

Any edge in $\mathcal{G}_{K,\ell,\ff_0}$  has vertices with corresponding central conductors $\ff$ and $\ff'$ satisfying $\ff/\ff' \in \{1,\ell,\ell^{-1}\}$. Defining the \textbf{level} of a vertex in $\mathcal{G}^D_{K,\ell,\ff_0}$ having central conductor $\ff$ to be $\text{ord}_\ell(\ff)$, we note that a directed edge can do one of three things:
\begin{itemize}
    \item increase the level by one, in which case we will call the edge \textbf{ascending},
    \item decrease the level by one, in which case we will call the edge \textbf{descending}, or
    \item leave the level unchanged, in which case we will call the edge \textbf{horizontal}. 
\end{itemize}

We will refer to ascending and descending edges collectively as \textbf{vertical} edges. For a connected component of $\mathcal{G}^D_{K,\ell,\ff_0}$, we refer to the subgraph consisting of level $0$ vertices and horizontal edges between them as the \textbf{surface} of that component. In other words, the vertex set of the surface consists of vertices with corresponding central conductor $\ff_0$. This choice of terminology is reflective of the fact that we cannot have an ascending isogeny starting at level $0$, and of fact that horizontal edges can only occur between surface vertices, as the following lemma states. 

\begin{lemma}
Suppose that there is a horizontal edge in $\mathcal{G}^D_{K,\ell,\ff_0}$ connecting vertices $v_1$ and $v_2$. Letting $\ff_i$ denote the central conductor corresponding to $v_i$ for $i=1,2$, we then have $\ff_1=\ff_2=\ff_0$. Furthermore, the number of horizontal edges emanating from a given surface vertex in $\mathcal{G}^D_{K,\ell,\ff_0}$ is $1 + \legendre{\Delta_K}{\ell}$, hence is
\begin{itemize}
    \item $0$ if $\ell$ is inert in $K$, 
    \item $1$ if $\ell$ ramified in $K$, and
    \item $2$ if $\ell$ is split in $K$. 
\end{itemize}
\end{lemma}

\begin{proof}
That $\ff_1=\ff_2$ is part of our definition of horizontal edges. What we must prove is that $\ell$ does not divide $\ff := \ff_1 = \ff_2$. 

The given edge corresponds to a QM-cyclic $\ell$ isogeny
\[ \varphi: (A_1,\iota_1) \rightarrow (A_2,\iota_2), \]
where $(A_i,\iota_i)$ has central conductor $\ff$ for $i=1,2$. By Theorem \ref{fake-ell-curves-ell-curves-correspondence}, we have a decomposition of these two QM abelian surfaces resulting in an isomorphic isogeny $\psi$ as below:
\[ \begin{tikzcd} 
(A_1,\iota_1) \arrow[r,"\varphi"] \arrow[d,"\cong"] & (A_2,\iota_2) \arrow[d,"\cong"] \\
(E_1 \times E_1', \iota_1) \arrow[r,"\psi"] & (E_2 \times E_2', \iota_2), 
\end{tikzcd} \]
where each $E_i$ and each $E_i'$ is an elliptic curve with $K$-CM by conductor $\ff$ for $i=1,2$. Restricting $\psi$ to $E_1$ and to $E_1'$, respectively, yields isogenies of $K$-CM elliptic curves
\begin{align}\label{ec-isogenies}
\begin{split}
E_1 &\longrightarrow \psi(E_1) =: E \subseteq E_2 \times E_2' \\ 
E_1' &\longrightarrow \psi(E_1') =: E' \subseteq E_2 \times E_2'. 
\end{split}
\end{align}
This provides the decomposition
\[ E_2 \times E_2' \cong E \times E' \] 
The conductors of the endomorphism rings of $E$ and $E'$, each of which must divide $\ff$ and have the same coprime to $\ell$ part as $\ff$, must then have least common multiple $\ff$. This provides that either $E$ or $E'$ must have CM conductor $\ff$. 

The conductors of the endomorphism rings of $E$ and $E'$ must each be in the set $\{\ff, \ell \ff, \frac{1}{\ell} \ff\}$, and must have least common multiple $f$. This provides that either $E$ or $E'$ must have CM conductor $\ff$. 

We now consider the corresponding isogeny of $K$-CM elliptic curves of conductor $\ff$ from (\ref{ec-isogenies}). In doing so, \cite[Lemma 4.1]{CS22} tells us that we must have $\ell \nmid f$. There, the result is reached using the correspondence between horizontal $\ell$-isogenies of $\mathfrak{o}(\ff_0)$-CM elliptic curves over $\mathbb{C}$ with proper $\mathfrak{o}(\ff_0)$-ideals of norm $\ell$. This also gives us the count of horizontal isogenies mentioned; we have the count in the elliptic curve case as in \cite{CS22}, and from a horizontal isogeny of elliptic curves as in (\ref{ec-isogenies}) we generate a QM-cyclic isogeny of our QM-abelian surfaces via the QM action. 
\end{proof}

Each surface vertex has $1 + \legendre{\Delta_K}{\ell}$ horizontal edges emanating from it, and therefore has $\ell - \legendre{\Delta_K}{\ell}$ descending edges to level $1$ vertices. For vertices away from the surface, we have the following:

\begin{lemma}
If $v$ is a vertex in $\mathcal{G}^D_{K,\ell,\ff_0}$ at level $L > 0$. then there is one ascending vertex from $v$ to a vertex in level $L-1$, and the remaining $\ell$ edges from $v$ are descending edges to distinct vertices in level $L+1$. 
\end{lemma}

\begin{proof}
We will use the same type of counting argument one may use in the elliptic curve case, as in \cite[Lemma 6]{Sut13}. The action of $\text{Gal}(\overline{\mathbb{Q}}/\mathbb{Q})$ on $\mathcal{G}^D_{K,\ell,\ff_0}$ preserves the level of a given vertex, and hence preserves the notions of horizontal, ascending, and descending for edges. As a result, the number of ascending, respectively descending, edges out of $v$ must be the same as for any other vertex at level $L$ by transitivity of this action on vertices at each level. 

For $L=1$, there are 
\[ \left( \ell - \legendre{\Delta_K}{\ell} \right) 2^b h(\mathfrak{o}(\ff_0)) = 2^b h(\mathfrak{o}(\ell \ff_0)) \]
total descending vertices from surface vertices (where $b$ is as in Proposition \ref{number-of-CM-points}). The equality above states that this is equal to the total number of level $1$ vertices, and so the edges must all be to distinct level $1$ vertices. For $L > 1$, the result follows inductively using the same counting argument along with the fact that 
\[ h(\mathfrak{o}(\ell^{L}\ff_0)) = \ell \cdot h(\mathfrak{o}(\ell^{L-1}\ff_0)). \qedhere \] 
\end{proof}

\subsection{QM-equivariant isogeny volcanoes}

For a prime number $\ell$, we define here the notion of an $\ell$-volcano. This notion for the most part agrees with that in the existing literature, with the only caveat being that in the original context of isogeny volcanoes over a finite field one has volcanoes of finite depth. In our case, working over an algebraically closed field as in \cite{Cl22, CS22}, we adjust the definition to allow for infinite depth volcanoes. 

\begin{definition}
Let $V$ be a connected graph with vertices partitioned into levels 
\[ V = \bigsqcup_{i \geq 0} V_i ,\]
such that if $V_d = \emptyset$ for some $d$, then $V_i = \emptyset$ for all $i \geq d.$ If such a $d$ exists, we will refer to the smallest such $d$ as the \textbf{depth} of $V$ and to $V_d$ for $d$ the depth as the \textbf{floor} of $V$, and otherwise we will say that the depth of $V$ is infinite. 

Fixing a prime number $\ell$, the graph $V$ with its partioning is an \textbf{$\ell$-volcano} if the following properties hold:
\begin{enumerate}
\item Each vertex not in the floor of $V$ has degree $\ell + 1$, while any floor vertex has degree $1$. 
\item The subgraph $V_0$, which we call \textbf{the surface}, is regular of degree $0, 1$ or $2$. 
\item For $0 < i < d$ (colloquially: ``below the surface'' and ``above the floor''), a vertex in $V_i$ has one \textbf{ascending} edge to a vertex in $V_{i-1}$, and $\ell$ \textbf{descending} edges to distinct vertices in $V_{i+1}$. This accounts for all edges of $V$ which are not \textbf{horizontal}, by which we mean edges which are not between two surface vertices. 
\end{enumerate}
\end{definition}

The results of the previous section immediately imply the following theorem, declaring that in most cases connected components of the subgraphs $\mathcal{G}^D_{K,\ell,\ff_0}$ of $\mathcal{G}^D_{K,\ell}$ are isogeny volcanoes. In such a case, we will refer to this graph as a \textbf{QM-equivariant isogeny volcano}. This justifies our use of terminology regarding edges and vertices in these subgraphs.  

\begin{theorem}\label{volcano_thm}
Fix an imaginary quadratic field $K$, a prime $\ell$ and a natural number $\ff_0$ with $(\ell,\ff_0) = 1$ and $\ff_0^2\Delta_K < -4$. Consider the graph $\mathcal{G}^D_{K,\ell,\ff_0}$ as an undirected graph by identifying edges with their dual edges as described above. Each connected component of this graph has the structure of an $\ell$-volcano of infinite depth. 
\end{theorem}

A \textbf{path} in $\mathcal{G}^D_{K,\ell,\ff_0}$ refers to a finite sequence of directed edges, say $e_1, \ldots, e_r$, such that the terminal vertex of $e_i$ is the initial vertex of $e_{i+1}$ for all $1 \leq i \leq r-1$. In the $\ff_0^2\Delta_K < -4$ case, because the edges in $\mathcal{G}^D_{K,\ell,\ff_0}$ all have canonical inverse edges we are justified in using the following terminology: we call an edge \textbf{backtracking} if $e_{i+1}$ is inverse to $e_i$ for some edge $e_i$ in the path. Note that in the case of $\ell$ ramified in $K$, a path consisting of two surface edges always is backtracking. If $\ell$ is split in $K$, then there is a horizontal cycle at the surface. In this case, concatenation of this cycle with itself any number of times does not result in backtracking. 

Our definitions and the results of this section lead us to the following correspondence:

\begin{lemma}\label{paths-to-isogenies-correspondence}
Suppose that $\ff_0^2\Delta_K < -4$. We then have a bijective correspondence between the set of geometric isomorphism classes of QM-cyclic $\ell^a$-isogenies of QM abelian surfaces with $K$-CM and central conductor with prime-to-$\ell$ part $f_0$, and the set of length $a$ non-backtracking paths in $\mathcal{G}^D_{K,\ell,\ff_0}$. This associates to an isogeny its corresponding path in this isogeny graph. 
\end{lemma}
\begin{proof}
This result is in exact analogy to \cite[Lemma 4.2]{Cl22}, and the proof is as therein. 
\end{proof}

In \S \ref{CM_points_prime_power_section}, we will describe the Galois orbits of such paths in order to describe the $K$-CM locus on $X_0^D(\ell^a)$ via the above correspondence. For this, the following observation will be of use: any non-backtracking length $a$ path in $\mathcal{G}^D_{K,\ell,\ff_0}$ for $\ff_0^2\Delta_K < -4$ can be written as a concatenation of paths $P_1, P_2$ and $P_3$, where $P_1$ is strictly ascending, $P_2$ is strictly horizontal and hence consists entirely of surface edges, and $P_3$ is strictly descending, such that the lengths of these paths (which may be $0$) sum to $a$. 

\subsection{The field of moduli of a QM-cylic $\ell$ isogeny}

A QM-cyclic $\ell$ isogeny $\varphi$ of $K$-CM abelian surfaces with $\ell \nmid D$ corresponds to an edge $e$ in $\mathcal{G}^D_{K,\ell,\ff_0}$, say between vertices $v$ and $v'$ in levels $L$ and $L'$, respectively. Assume that the path is non-descending ($L \geq L'$), so either it is horizontal ($L = L'$) or ascending ($L = L'+1$). 

An automorphism fixing $e$ must fix both $v$ and $v'$, and so by Theorems \ref{Shimura} and \ref{JGR} we have that either $\mathbb{Q}(\varphi) = K(\ell^{L}\ff_0)$, or $[K(\ell^{L}\ff_0) : \mathbb{Q}(\varphi)] = 2$. In the latter case, there exists an involution $\sigma \in \text{Gal}(K(\ell^{L}\ff_0)/\mathbb{Q})$ fixing $v$, and we know precisely when this occurs by Theorem \ref{JGR} -- that is, when $D(K) = 1$. 

Assume that $\ff_0^2\Delta_K < -4$, such that $\mathcal{G}^D_{K,\ell,\ff_0}$ has the structure of an $\ell$-volcano. (We will deal with the case of $\ff_0^2\Delta_K \in \{-3,-4\}$ in the remarks leading up to Proposition \ref{fom-prime-power}.) If $e$ is the unique edge between $v$ and a vertex in level $L'$, then $e$ is fixed by $\sigma$ if and only if $v$ is. This is the case unless $L=L' = 0$ and $\ell$ splits in $K$, in which case there are two edges from $v$ to surface vertices (which are not necessarily unique, or distinct from $v$). In either of these cases, consider $[(E \times E', \iota)],$ with $E$ having CM by $\mathfrak{o}(\ff_0)$, a decomposition of our QM abelian surface corresponding to $v_1$. The two outward edges from $v$ then have corresponding kernels $\iota(\mathcal{O})\cdot E[\mathfrak{p}]$ and $\iota(\mathcal{O})\cdot E[\overline{\mathfrak{p}}]$, with $\mathfrak{p}$ a prime ideal in $\mathfrak{o}(\ff_0)$ of norm $\ell$.

We claim that, in this situation, the involution $\sigma \in \text{Gal}(K(\ff_0)/\mathbb{Q})$ fixing $v$ cannot fix $\mathfrak{p}$, and hence cannot fix our edge $e$. Indeed, the exact statement of Theorem $\ref{JGR}$ says that $\sigma = \tau \sigma_\mathfrak{a}$ for a certain ideal $\mathfrak{a}$ of $\mathfrak{o}(\ff_0)$, so to fix $e$ it would have to be the case that $\sigma_\mathfrak{a}$ acts on $e$ and hence on $v$ by complex conjugation. It follows from \cite[Lemma 5.10]{GR06} that this cannot be the case, as $\omega_{D(K)}$ acts non-trivially on $v$. From this discussion, we reach the following result regarding fields of moduli corresponding to our edges.

\begin{proposition}\label{fom-edge}
Let $\varphi$ be a QM cyclic $\ell$-isogeny corresponding to an edge $e$ from $v$ to $v'$ in $\mathcal{G}^D_{K,\ell,\ff_0}$ as above, with $\ff_0^2\Delta_K<-4$.
\begin{itemize}
    \item If $D(K) \neq 1$, i.e., if there is a prime $p \mid D$ which is inert in $K$, then $\mathbb{Q}(\varphi) = K(\ell^{L}\ff_0)$. 
    \item Suppose that $D(K) = 1$. 
    
    \begin{itemize}
        \item If $\varphi$ is a QM cyclic isogeny of QM abelian surfaces with CM by $\mathfrak{o}(\ff_0)$ and $\ell$ splits in $K$, then $\mathbb{Q}(\varphi) = K(\ff_0)$.
        
        \item Otherwise, $[K(\ell^{L}\ff_0) : \mathbb{Q}(\varphi)] = 2$, with $\mathbb{Q}(\varphi)$ equal to the field of moduli corresponding to $v$ as described in Theorem \ref{JGR}. 
    \end{itemize}
\end{itemize}
\end{proposition}


\section{The action of Galois on $\mathcal{G}^D_{K,\ell,\ff_0}$}\label{Galois_action_section}

\subsection{Action of $\textnormal{Gal}(\overline{\mathbb{Q}}/\mathbb{Q})$}
We have an action of $\textnormal{Aut}(\mathbb{C})$ on $\mathcal{G}^D_{K,\ell,\ff_0}$: an automorphism $\sigma$ maps a vertex $v$ corresponding to an isomorphism class of QM abelian surfaces $[(A,\iota)]$ to the vertex corresponding to $[(\sigma(A),\sigma(\iota))]$, and edges are mapped to edges via the action on the corresponding isomorphism classes of isogenies. This action factors through $\text{Gal}(\overline{\mathbb{Q}}/\mathbb{Q})$, and preserves the level of a vertex. It follows that it also preserves the notions of ascending, descending and horizontal for paths. 

For a fixed level $L \geq 0$, let $\mathcal{G}^D_{K,\ell,\ff_0,L}$ denote the portion of $\mathcal{G}^D_{K,\ell,\ff_0}$ from the surface (level $0$) to level $L$:
\[ \mathcal{G}^D_{K,\ell,\ff_0,L} := \bigsqcup_{i=0}^{L} V_i \subseteq \mathcal{G}^D_{K,\ell,\ff_0}. \]
By Theorem \ref{Shimura}, the action of $\text{Gal}(\overline{\mathbb{Q}}/\mathbb{Q})$ on $\mathcal{G}^D_{K,\ell,\ff_0,L}$ factors through $\text{Gal}(K(\ell^L \ff_0)/\mathbb{Q})$. If $D(K) \neq 1$, i.e., if there is some prime $p \mid D$ which is inert in $K$, then Theorem \ref{JGR} says that the action of this group on $V_L$ is free. Otherwise, each vertex $v$ in level $L$ is fixed by some involution $\sigma$, and the class of QM abelian surfaces corresponding to $v$ has field of moduli isomorphic to $K(\ell^L \ff_0)^\sigma$. 

We now fix a vertex $v$ in level $L$ in $\mathcal{G}^D_{K,\ell,\ff_0}$, and suppose that $\sigma \in \text{Gal}(K(\ell^L \ff_0)/\mathbb{Q})$ is an involution fixing $v$. (This forces $D(K) = 1$.) In the following two sections, we provide an explicit description of the action of $\sigma$ on $\mathcal{G}^D_{K,\ell,\ff_0,L}$ in all cases. First, we note here the number of vertices at each level fixed by $\sigma$. 

\begin{proposition}\label{fixed_count}
Let $x \in X^D(1)_{/\mathbb{Q}}$ be an $\mathfrak{o}(\ell^L\ff_0)$-CM point fixed by an involution $\sigma \in \textnormal{Gal}(K(\ell^L\ff_0)/\mathbb{Q})$. Let $b$ denote the number of prime divisors of $D$ which are inert in $K$. For $0 \leq L' \leq L$, the number of vertices of $\mathcal{G}^D_{K,\ell,\ff_0}$ in level $L'$ fixed by $\sigma$ is 

\[ 2^b \cdot \#\textnormal{Pic}(\mathfrak{o}(\ell^{L'}\ff_0))[2]. \]
\end{proposition}

\begin{proof}
By Theorem $\ref{JGR}$, the involution $\sigma$ is of the form $\sigma = \tau \circ \sigma_0$ for some $\sigma_0 \in \textnormal{Pic}(\mathfrak{o}(\ell^L\ff_0))$, where $\tau$ denotes complex conjugation. The set of vertices of $\mathcal{G}^D_{K,\ell,\ff_0}$ at level $L'$ has cardinality $2^b \cdot h(\mathfrak{o}(\ell^{L'} \ff_0))$, consisting of $2^b$ orbits under the action of $\textnormal{Pic}(\mathfrak{o}(\ell^{L'}\ff_0))$. Each orbit is a $\textnormal{Pic}(\mathfrak{o}(\ell^{L'}\ff_0))$-torsor, and $\sigma_0$ yields a bijection on each. 

As a result, we have that the number of level $L'$ vertices in a given orbit which are fixed by $\sigma$ is the same as the number of elements of $\textnormal{Pic}(\mathfrak{o}(\ell^{L'} \ff_0))$ fixed by $\tau$. As shown in \cite[Prop. 2.6]{Cl22}, this count is equal to $\#\textnormal{Pic}(\mathfrak{o}(\ell^{L'}\ff_0))[2]$, as $\tau$ acts on $\textnormal{Pic}(\mathfrak{o}(\ell^{L'}\ff_0))$ by inverting ideals. 
\end{proof}

Regarding this count, by \cite[Prop. 3.11]{Cox13} we have the following:

\begin{lemma}\label{2-torsion}
Let $r$ denote the number of distinct odd prime divisors of a fixed imaginary quadratic discriminant $\Delta$. Then $\textnormal{Pic}(\mathfrak{o}_\Delta)[2] \cong \left(\mathbb{Z}/2\mathbb{Z}\right)^\mu$, where
\[ \mu = \begin{cases} r-1 &\text{ if } \Delta \equiv 1 \phantom{, 12} \pmod{4} \phantom{6} \text{ or } \Delta \equiv 4 \phantom{1} \pmod{16}, \\
r &\text{ if } \Delta \equiv 8, 12 \pmod{16} \text{ or } \Delta \equiv 16 \pmod{32}, \\
r+1 &\text{ if } \Delta \equiv 0 \phantom{, 12} \pmod{32}.  
\end{cases} \]
\end{lemma}

\subsection{The field of moduli of a QM-cyclic $\ell^a$ isogeny}

Let $\varphi$ be a QM cyclic $\ell^a$ isogeny of $K$-CM abelian surfaces inducing a $\Delta = \ff^2\Delta_K$-CM point on $X_0^D(\ell^a)_{/\mathbb{Q}}$, with $\ell \nmid D$ and $D>1$. Let $P$ be the length $a$ non-backtracking path in $\mathcal{G}^D_{K,\ell,\ff_0}$ corresponding to $\varphi$, via Lemma \ref{paths-to-isogenies-correspondence}, for the appropriate $\ff_0 \in \mathbb{Z}^+$. The ordered edges in $P$ correspond to a decomposition
\[ \varphi = \varphi_1 \circ \cdots \circ \varphi_a, \]
where each $\varphi_i$ is a QM-cyclic $\ell$-isogeny. If $\Delta < -4$, then Lemma \ref{compositions} provides
\[ \mathbb{Q}(\varphi) = \mathbb{Q}(\varphi_1)\cdots\mathbb{Q}(\varphi_a), \]
and for $\ff_0^2\Delta_K < -4$ Proposition \ref{fom-edge} determines $\mathbb{Q}(\varphi_i)$ for each $i$. Note that if $\mathbb{Q}(\varphi_i)$ is a ring class field for any $i$, then $\mathbb{Q}(\varphi)$ must contain $K$

For $\ff_0^2\Delta_K \in \{-3,-4\}$, it is impossible to have $D(K) = 1$, as $\Delta_K$ has only a single prime divisor while $D$ has at least $2$. This is of course consistent with, and can be seen from, the general fact that Shimura curves have no real points; the residue field of a $-3$-CM or $-4$-CM point on $X^D(1)_{/\mathbb{Q}}$ must be $K$ in this situation. By these observations and the discussion of the Galois action in the previous section, we have the following proposition. 

\begin{proposition}\label{fom-prime-power}
Let $\varphi: (A,\iota) \rightarrow (A',\iota')$ be a QM-cyclic $\ell^a$ isogeny. Suppose that $(A,\iota)$ has $K$-CM with central conductor $\ff_A = \ell^a \ff_0$ and that $(A',\iota')$ has central conductor $f_{A'} = \ell^{a'}\ff_0$. Let $L = \textnormal{max}\{ a,a'\}$. Let $P$ be the path corresponding to $\varphi$ in $\mathcal{G}^D_{K,\ell,\ff_0,L}$. 
\begin{itemize}
\item If $D(K) \neq 1$, i.e., if there is a prime $p \mid D$ which is inert in $K$, then $\mathbb{Q}(\varphi) = K(\ell^{L} \ff_0)$.
\item Suppose that $D(K) = 1$. 
\begin{itemize}
\item If $\ell$ splits in $K$ and $\varphi$ factors through an $\ell$-isogeny of QM abelian surfaces with $\ff_0^2\Delta_K$-CM, then $\mathbb{Q}(\varphi) = K(\ell^{L}\ff_0)$. 
\item Suppose that we are not in the previous case. Let $\sigma \in \textnormal{Gal}(K(\ell^{L}\ff_0)/\mathbb{Q})$ be an involution fixing the class of $(A,\iota)$ or $(A',\iota')$. If $\sigma$ fixes the path $P$, then $\mathbb{Q}(\varphi) = K(\ell^{L}\ff_0)^\sigma$. Otherwise, $\mathbb{Q}(\varphi) = K(\ell^{L}\ff_0)$. 
\end{itemize}
\end{itemize}
\end{proposition}

We now explicitly analyze the Galois action in all cases as done in \cite[\S 5.3]{Cl22}  and \cite[\S 4.2]{CS22} in the $D=1$ case. Borrowing the notation therein, for a specified $K, \ff_0,$ and $\ell$ we let
\[ \tau_L := \#\textnormal{Pic}(\mathfrak{o}(\ell^L\ff_0)[2].\]
By Proposition \ref{fixed_count}, the number of vertices in level $L'$ in $\mathcal{G}^D_{K,\ell,\ff_0}$ which are fixed by an involution $\sigma \in \textnormal{Pic}(\ell^{L}\ff_0)$ of the type we are studying is $2^b \cdot \tau_{L'}$. 

\subsection{Explicit description: $\ff_0^2\Delta_K < -4$}

In the current section, we assume $\ff_0^2\Delta_K < -4$, such that each component of $\mathcal{G}^D_{K,\ell,\ff_0}$ has the structure of an $\ell$-volcano of infinite depth. This is in exact parallel to \cite[\S 5.3]{Cl22}, baring the same structure of results. 

Let $0 \leq L' \leq L$, and let $\sigma \in \textnormal{Pic}(\ell^L \ff_0)$ be an involution fixing a vertex $v$ in $\mathcal{G}_{K,\ell,\ff_0}^{D}$ in level $L$. In the following lemmas, we describe the action of $\sigma$ on $\mathcal{G}^D_{K,\ell,\ff_0,L}$. In each case, we provide example figures of a component of $\mathcal{G}_{K,\ell,\ff_0}^D$ (up to some finite level). In these graphs, vertices and edges colored purple are fixed by the action of the designated involution $\sigma$, while black edges and vertices are acted on non-trivially by $\sigma$. Without loss of generality based on the symmetry of our graph components, we will always take $v$ to be the left-most vertex in level $L$ in our figures. 

\begin{lemma}\label{l_odd_unram}
Let $\ell > 2$ be a prime which is unramified in $K$ and $\ff_0 \in \mathbb{Z}^+$ with $\ff_0^2\Delta_K < -4$. Let $v, L$ and $\sigma$ be as above with $L \geq 1$, and consider the action of $\sigma$ on 
\[\bigsqcup_{i=0}^{L} V_i \subseteq \mathcal{G}^D_{K,\ell,\ff_0}. \]
Each surface vertex has two descendants fixed by $\sigma$ in level $1$. For $1 \leq L' < L$, each fixed vertex in level $L'$ has a unique fixed descendent in level $L'+1$. 
\end{lemma}
\begin{proof}
By Lemma \ref{2-torsion} we have $\tau_{1} = 2\tau_{0}$, while $\tau_{L'} = \tau_{L'+1}$ for $1 \leq L' < L$. The number of edges descending from a given vertex in level $L' \geq 1$ is $\ell$, hence is odd, and so we immediately see that each fixed vertex in level $L'$ with $1 \leq L' \leq L$ must have at least one fixed descendant in level $L'+1$, hence exactly one by our count. 

The number of descending edges from a given surface vertex is either $\ell+1$ or $\ell-1$ depending on whether $\ell$ is inert or split in $K$, hence is even in both cases. With our involution being of the form $\sigma = \tau \sigma_0$, a translated version of the argument of \cite[Cor. 5.5]{CS22} gives that each fixed surface vertex has at least one fixed descendant in level $1$. Therefore, each fixed surface vertex must have at least two fixed descendants in level $1$ by parity, giving the result.
\end{proof}

\begin{figure}[H] 
\includegraphics[scale=0.7]{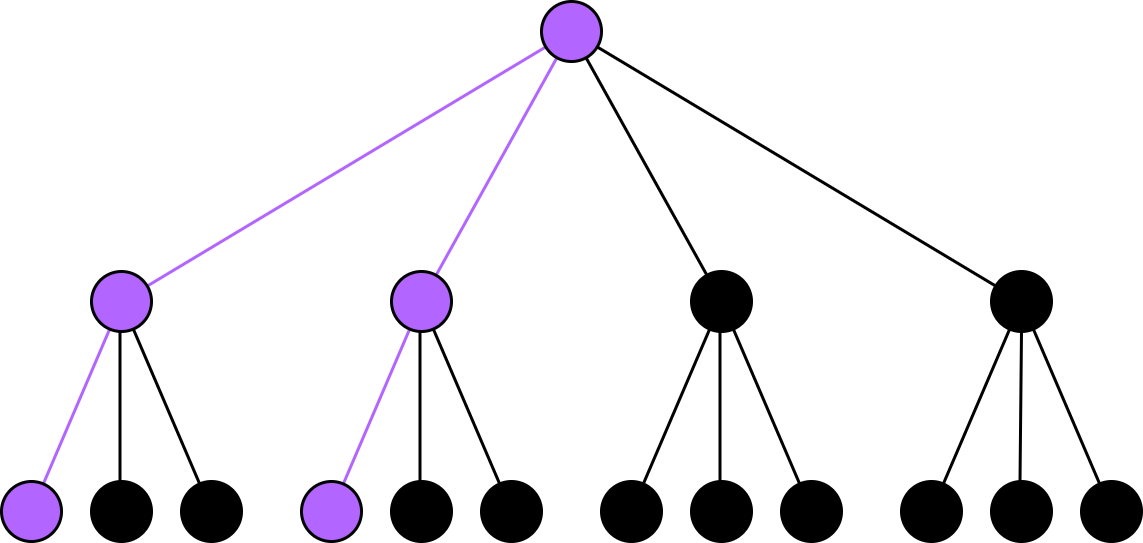}
\setlength{\abovecaptionskip}{10pt plus 0pt minus 0pt}
\caption{$\ell = 3$ inert in $K$ with $L = 2$}\label{l_odd_inert}
\end{figure}

\begin{figure}[H] 
\includegraphics[scale=0.7]{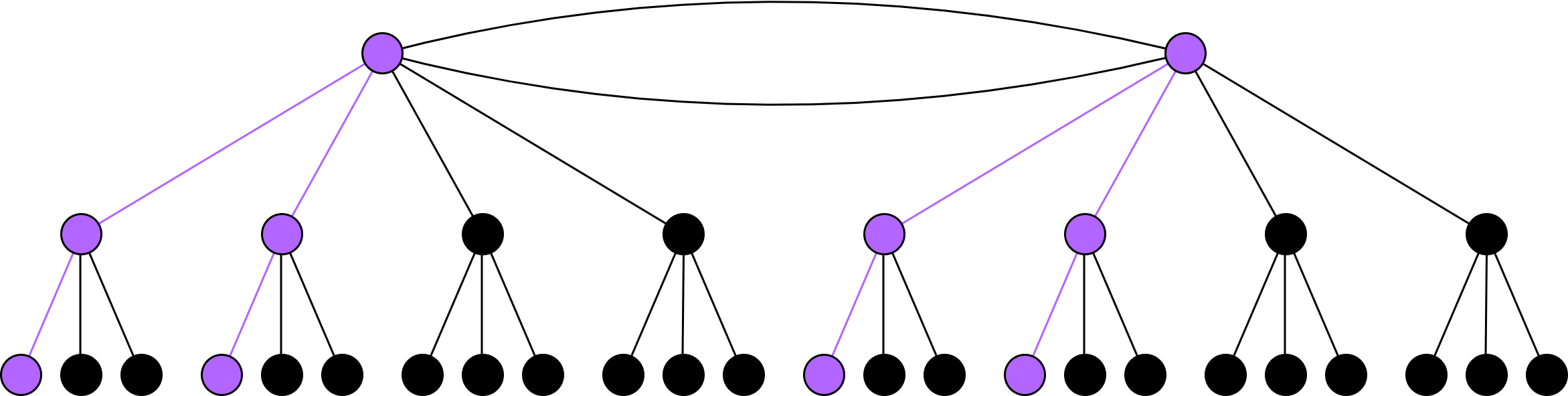}
\setlength{\abovecaptionskip}{10pt plus 0pt minus 0pt}
\caption{$\ell = 3$ split in $K$ with $L = 2$}\label{l_odd_split}
\end{figure}

\begin{lemma}
Let $\ell > 2$ be a prime that ramifies in $K$ and $\ff_0 \in \mathbb{Z}^+$ with $\ff_0^2\Delta_K < -4$. Let $v, L$ and $\sigma$ be as above, and consider the action of $\sigma$ on 
\[\bigsqcup_{i=0}^{L} V_i \subseteq \mathcal{G}^D_{K,\ell,\ff_0}. \]
Any vertex $v'$ in level $L'$ with $0 \leq L' < L$ which is fixed by $\sigma$ has exactly one descendant in level $L'+1$ fixed by $\sigma$. 
\end{lemma} 

\begin{proof}
Each vertex in level $L'$ has $\ell$ descendants in level $L'+1$. A descendant of $v'$ must be sent to another descendant of $v'$ by $\sigma$, by virtue of $v'$ being fixed by $\sigma$. At least one descendant must be fixed by $\sigma$ by the assumption that $\ell$ is odd. Lemma \ref{2-torsion} gives that $\tau_{L'} = \tau_{L'+1}$, and so there must be exactly one fixed descendant of $v'$. 
\end{proof}

\begin{figure}[H] 
\includegraphics[scale=0.7]{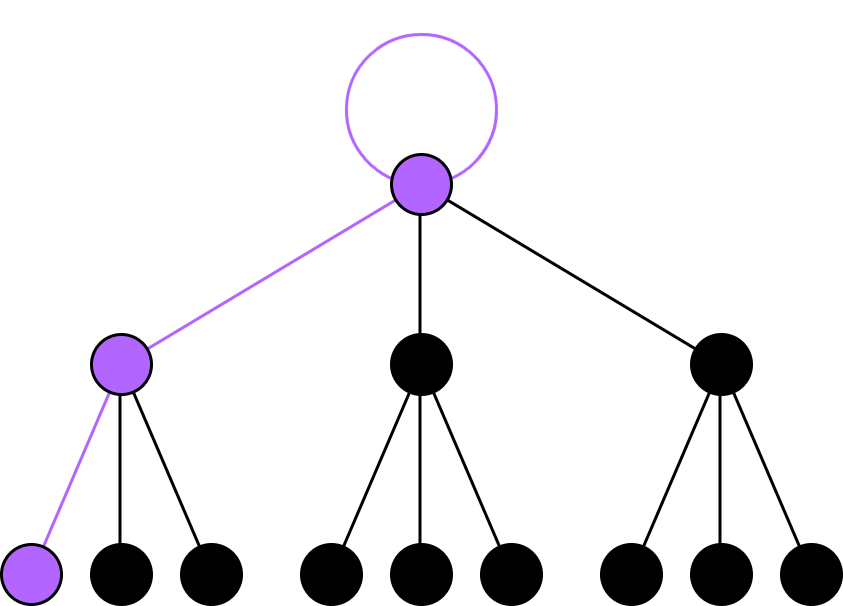}
\setlength{\abovecaptionskip}{10pt plus 0pt minus 0pt}
\caption{$\ell = 3$ ramified in $K$ with $|V_0| = 1$ and $L = 2$}\label{l_odd_ram}
\end{figure}

\begin{figure}[H]
\includegraphics[scale=0.7]{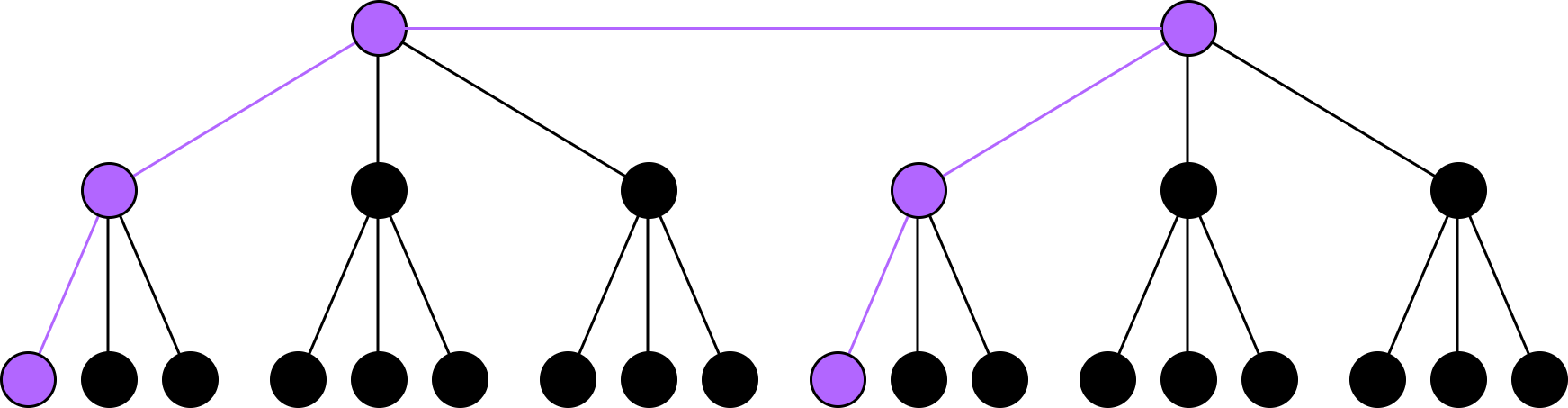}
\setlength{\abovecaptionskip}{10pt plus 0pt minus 0pt}
\caption{$\ell = 3$ ramified in $K$ with $|V_0| = 2$ and $L = 2$}\label{l_odd_ram2} 
\end{figure}

\begin{lemma}\label{2-unram-lemma}
Suppose that $\ell = 2$ is unramified in $K$ and that $\ff_0^2\Delta_K \neq -3$. Let $v, L$ and $\sigma$ be as above with $L \geq 1$, and consider the action of $\sigma$ on
\[\bigsqcup_{i=0}^{L} V_i \subseteq \mathcal{G}^D_{K,2,\ff_0}. \]
\begin{enumerate}
\item Every surface vertex fixed by $\sigma$ has a unique fixed descendant in level $1$.
\item Suppose $L \geq 2$. Each vertex in level $1$ which is fixed by $\sigma$ has all of its descendants in levels $2$ to $\text{min}(L,3)$ fixed by $\sigma$. 
\item Let $3 \leq L' < L$. If $v'$ is a vertex in level $L'$ fixed by $\sigma$, then the vertex $w$ in level $L'$ which shares a neighbor in level $L'-1$ with $v'$ is also fixed by $\sigma$, and exactly one of $v'$ and $w$ has its two descendants in level $L'+1$ fixed by $\sigma$. 
\end{enumerate}
\end{lemma} 

\begin{proof}
(1) Lemma \ref{2-torsion} provides $\tau_1 = \tau_0$. If $2$ is inert in $K$, then each fixed surface vertex has three neighbors in level $1$ and hence at least one must be fixed. The count then implies exactly one of these neighbors must be fixed. If $2$ splits in $K$, then each fixed surface vertex has exactly one neighbor in level $1$ which then must be fixed. 

(2) Lemma \ref{2-torsion} provides $\tau_3 = 2\tau_2$ and $\tau_2 = 2\tau_1$. As each non-surface vertex has two immediate descendants in the next level, the claim follows. 

(3) For $3 \leq L' < L$, we have $\tau_{L'+1} = \tau_{L}$. Let $v_{L'}$ be a fixed vertex in level $L'$ having a fixed neighbor vertex in level $L'-1$. By a parity argument, there must then be another fixed vertex $w_{L'}$ in level $L'$ with the same neighbor in level $L'-1$ as $v_{L'}$. By the count, it suffices to show that $v_{L'}$ and $w_{L'}$ cannot both have descendants fixed by $\sigma$. 

Suppose to the contrary that $v_{L'+1}$ and $w_{L'+1}$ are $\sigma$-fixed neighbors of $v_{L'}$ and $w_{L'}$, respectively, in level $L'+1$. We find that this cannot be the case as in \cite[Lemma 5.6 c]{Cl22}; this would imply that we have a QM-cyclic $2^4$-isogeny which , upon restriction, would provide a cyclic, real $2^4$-isogeny of elliptic curves with CM by $\Delta = 2^{2L+2}\ff_0^2\Delta_K$. This in turn implies the existence of a primitive, proper real $\mathfrak{o}(2^{L+1}\ff_0)$-ideal of index $16$, which does not exist. 
\end{proof}

\begin{figure}[H]
\includegraphics[scale=0.7]{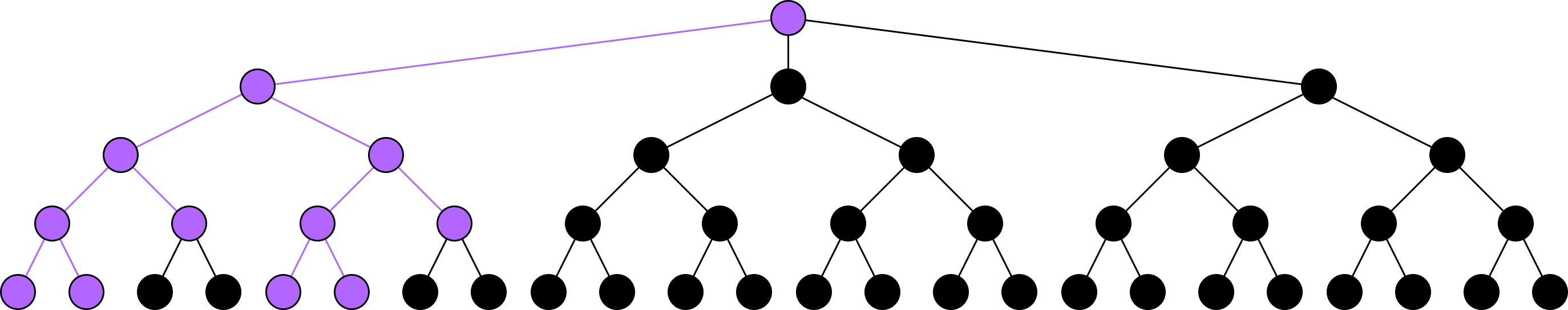}
\setlength{\abovecaptionskip}{10pt plus 0pt minus 0pt}
\caption{$\ell = 2$ inert with $L = 4$}\label{2_inert} 
\end{figure}

\begin{figure}[H]
\includegraphics[scale=0.7]{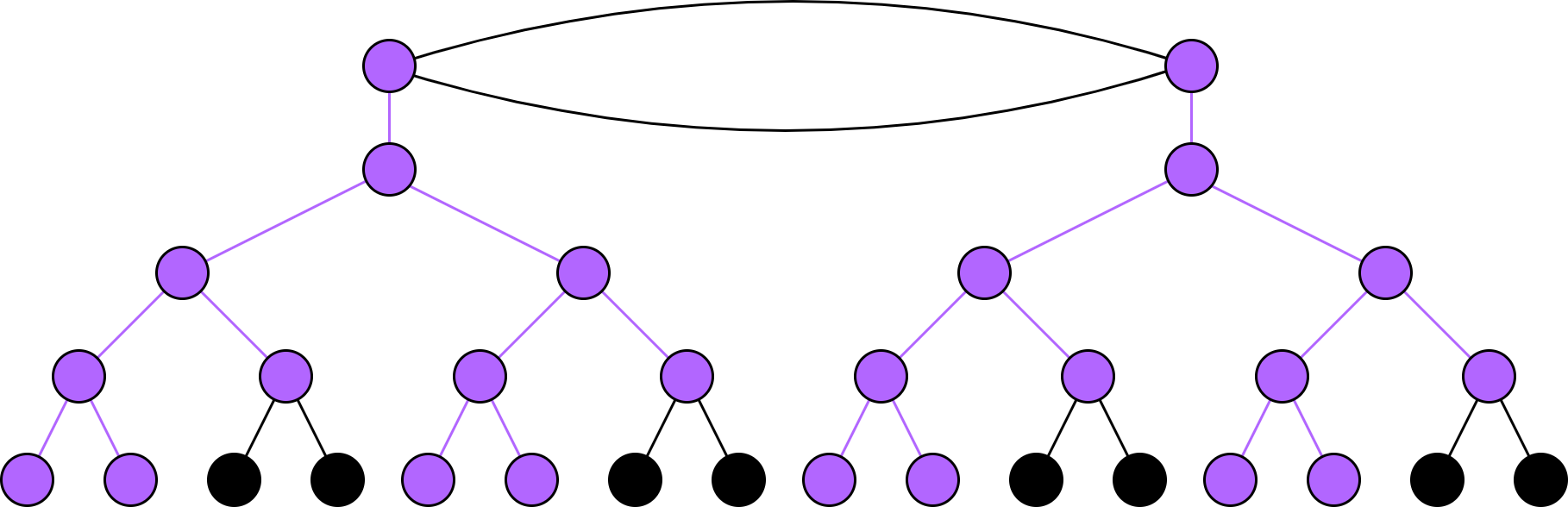}
\setlength{\abovecaptionskip}{10pt plus 0pt minus 0pt}
\caption{$\ell = 2$ split with $L = 4$}\label{2_split} 
\end{figure}

In the case of $\ell = 2$ ramifying in $K$, the discriminant of $K$ must be of the form $\Delta_K = 4m$ for $m \equiv 2$ or $3 \pmod{4}$, and so $\Delta_K \equiv 8$ or $12 \pmod{16}$. Hence, the discriminant of the order $\mathfrak{o}(\ff_0)$ corresponding to the surface of $\mathcal{G}^D_{K,2,\ff_0}$ will also lie in one of these congruence classes mod $16$. Whether these components have a surface loop is answered by the following lemma.

\begin{lemma}\label{2-ram-loops}
Consider a component of $\mathcal{G}^D_{K,2,\ff_0}$ with $2$ ramified in $K$. The surface $V_0$ of this component consists of a single vertex with a single self-loop if and only if $\Delta_K \in \{-4,-8\}$ and $\ff_0 = 1$.
\end{lemma}
\begin{proof}
This proof comes down to a simple argument about ideals of norm $2$ in $\mathfrak{o}(f_0)$, as in \cite[Lemma 5.7]{Cl22}
\end{proof}

\noindent The following lemmas therefore cover all possible cases.

\begin{lemma}
Let $\Delta_K = -8$ and $\ell = 2$, and let $v, L$ and $\sigma$ be as above with $L \geq 1$. Consider the action of $\sigma$ on 
\[\bigsqcup_{i=0}^{L} V_i \subseteq \mathcal{G}^D_{K,2,1}. \]
\begin{enumerate}
    \item The two descendants in level $1$ of the single surface vertex are fixed by $\sigma$.
    \item For $1 < L' < L$, there are $2$ vertices in level $L'$ fixed by $\sigma$ and they have a common neighbor vertex in level $L'-1$. One of these must have both descendants in level $L'+1$ fixed by $\sigma$, while the other has its direct descendants swapped by $\sigma$. 
\end{enumerate}
\end{lemma}

\begin{proof}
There is a single vertex on the surface, as the class number of $K$ is $1$. Lemma \ref{2-torsion} tells us that $\tau_1 = 2\tau_0$ in this case, so both descendants of the surface vertex are fixed by $\sigma$. For $1 \leq L' < L$, we have 
\[ \tau_{L'+1} = \tau_{L'} = 2, \]
so one of the fixed vertices in level $L'$ must have both descendants in level $L'+1$ fixed by $\sigma$, while the other has its vertices swapped by $\sigma$. 
\end{proof}

\begin{figure}[H]
\includegraphics[scale=0.7]{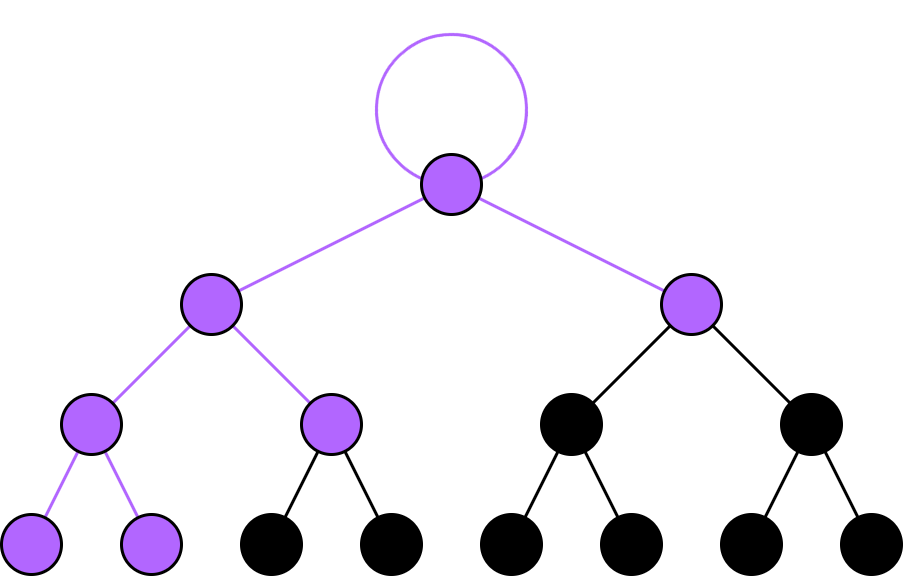}
\setlength{\abovecaptionskip}{10pt plus 0pt minus 0pt}
\caption{$\ff_0^2\Delta_K = -8$ and $\ell = 2$ with $L = 3$}\label{-8_ram} 
\end{figure}

\begin{lemma}\label{2-ram-12mod16}
Suppose that $\Delta_K \equiv 12 \pmod{16}$ and $\ff_0^2\Delta_K \neq -4$ with $\ell = 2$. Let $v, L$ and $\sigma$ be as above with $L \geq 1$. Consider the action of $\sigma$ on 
\[\bigsqcup_{i=0}^{L} V_i \subseteq \mathcal{G}^D_{K,2,\ff_0}. \]
\begin{enumerate}
    \item There are two surface vertices, both fixed by $\sigma$. One surface vertex, which we will denote by $v_0$, has both descendants in level $1$ fixed by $\sigma$, while the other has its level $1$ descendants swapped by $\sigma$.
    \item If $L \geq 2$ (such that the action of $\sigma$ is defined at level $2$), then each of the $4$ vertices in level $2$ which descend from $v_0$ are fixed by $\sigma$. 
    \item For $2 \leq L' < L$ and for a vertex $v'$ in level $L'$ fixed by $\sigma$, let $w$ denote the other level $L'$ vertex sharing a neighbor vertex in level $L'-1$ with $v'$ (which must also be fixed by $\sigma$). Exactly one of $v'$ or $w$ has both descendants in level $L'+1$ fixed by $\sigma$, while the other vertex has its direct descendants swapped by $\sigma$. 
\end{enumerate}
\end{lemma}

\begin{proof}
In this case the surface has two $\sigma$-fixed vertices with a single edge between them. We have 
\[ \tau_1 = \tau_0 \quad \text{ and } \quad \tau_2 = 2\tau_1 \] 
by Lemma \ref{2-torsion}, giving parts (1) and (2).
For $2 \leq L' < L$, we have
\[ \tau_{L'} = \tau_{L'-1}, \]
so half of the $\sigma$-fixed vertices in level $L'-1$ must have both descendants in level $L'$ fixed by $\sigma$, while the other half have their descendants in level $L'$ swapped by $\sigma$. That there must be exactly one pair of fixed vertices in level $L'$ descending from a given fixed vertex in level $L'-2$ follows as in part (3) of Lemma \ref{2-unram-lemma}.
\end{proof}

\begin{figure}[H]
\includegraphics[scale=0.7]{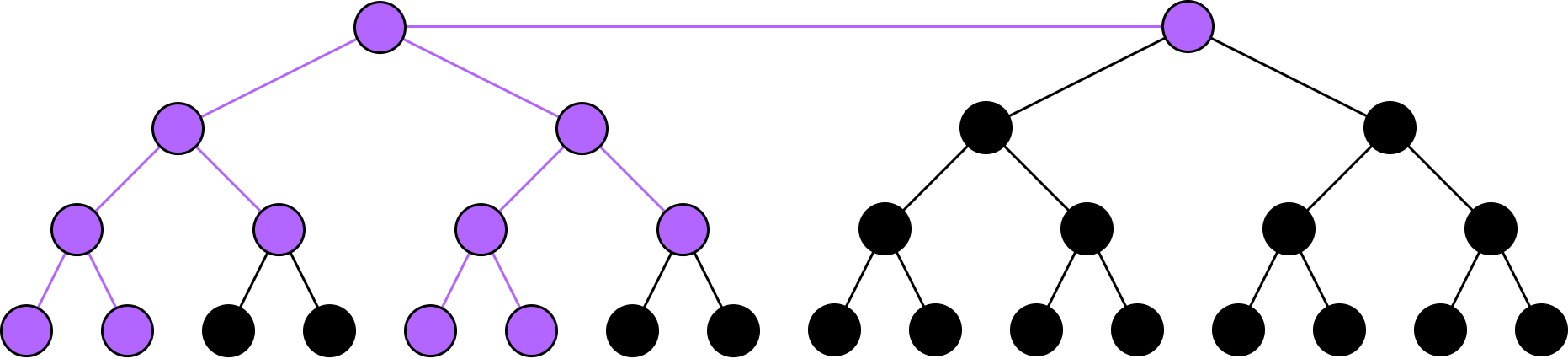}
\setlength{\abovecaptionskip}{10pt plus 0pt minus 0pt}
\caption{$\Delta_K \neq -4$ with $\ell = 2,  \text{ord}_2(\Delta_K) = 2$ and $L = 3$}\label{12_2ram} 
\end{figure}

\begin{lemma}\label{2-ram-8mod16}
Suppose that $\Delta_K \equiv 8 \pmod{16}$ with $\Delta_K < -8$ and $\ell = 2$. Let $v, L$ and $\sigma$ be as above with $L \geq 1$. Consider the action of $\sigma$ on 
\[\bigsqcup_{i=0}^{L} V_i \subseteq \mathcal{G}^D_{K,2,\ff_0}. \]
\begin{enumerate}
    \item There are two surface vertices, both fixed by $\sigma$, and all $4$ vertices in level $1$ are fixed by $\sigma$. 
    \item For $1 \leq L' < L$ and for a vertex $v'$ in level $L'$ fixed by $\sigma$, let $w$ denote the other level $L'$ vertex sharing a neighbor vertex in level $L'-1$ with $v'$. Exactly one of $v'$ or $w$ has both descendants in level $L'+1$ fixed by $\sigma$, while the other vertex has its direct descendants swapped by $\sigma$. 
\end{enumerate}
\end{lemma}

\begin{proof}
In this case again we have two $\sigma$-fixed vertices comprising our surface. Here Lemma \ref{2-torsion} gives $\tau_1 = 2\tau_0$, providing part (1). For $1 \leq L' < L$, Lemma $\ref{2-torsion}$ gives $\tau_{L'} = \tau_{L'-1}$. The same argument as in part (3) of Lemma \ref{2-ram-12mod16} then provides part (2). 
\end{proof}

\begin{figure}[H]
\includegraphics[scale=0.7]{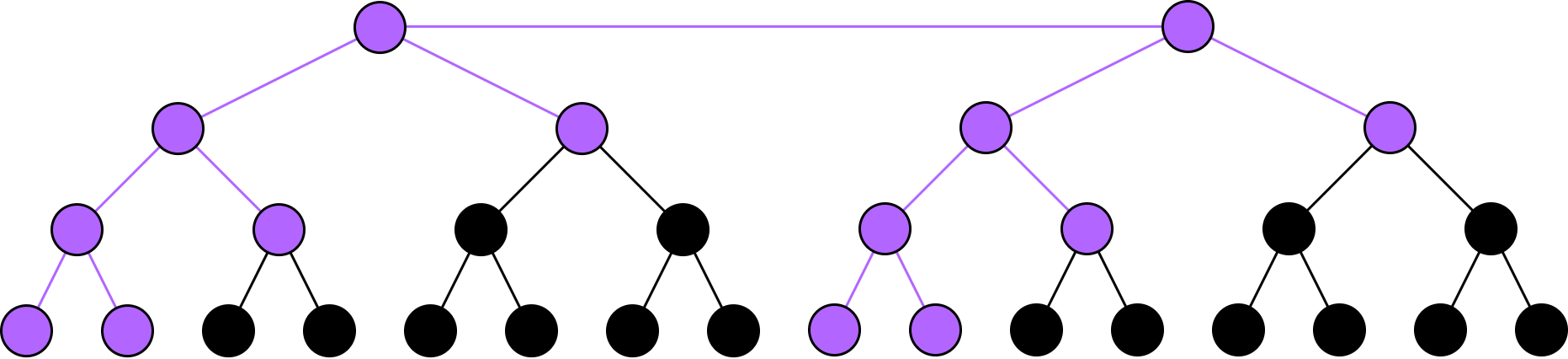}
\setlength{\abovecaptionskip}{10pt plus 0pt minus 0pt}
\caption{$\Delta_K < -8$ with $\ell = 2,  \text{ord}_2(\Delta_K) = 3$ and $L = 3$}\label{8_2ram} 
\end{figure}

\subsection{Explicit description part II: $\ff_0^2\Delta_K \in \{-3,-4\}$}

Keeping our notation from the previous section, we now assume $\ff_0=1$ and $\Delta_K \in \{-3,-4\}$. As mentioned earlier in this section, we always have $D(K) \neq 1$ in this case. Therefore, the action of $\text{Gal}(K(\ell^L \ff)/\mathbb{Q})$ on $V_L$ is free for all $L \geq 0$. This is splendid news for us; while the CM fields $\mathbb{Q}(i)$ and $\mathbb{Q}(\sqrt{-3})$ require extra attention at other points in this study, they cause absolutely no difficulties as far as determining the explicit Galois action on $\mathcal{G}^D_{K,\ell,1}$. This is to be compared with the $D=1$ case of \cite[\S 4]{CS22}, wherein much care goes into defining and explicitly describing a meaningful action of complex conjugation on CM components of isogeny graphs in these cases. 

Still, we provide here example figures of components of $\mathcal{G}_{K,\ell,1}$ (up to finite level $L$) for each case as reference for the reader for the path type analysis and enumeration done in \S \ref{CM_points_prime_power_section}. In these cases, edges from level $0$ to $1$ have multiplicity as exposited in \cite[\S 3]{CS22} due to the presence of automorphisms that do not fix kernels of isogenies. We therefore \emph{do not} have a one-to-one identification between edges and ``dual'' edges in this case, and so as in the referenced study we clearly denote edges with orientation and multiplicity between levels $0$ and $1$.

\begin{figure}[H]\label{-4ram} 
\includegraphics[scale=0.7]{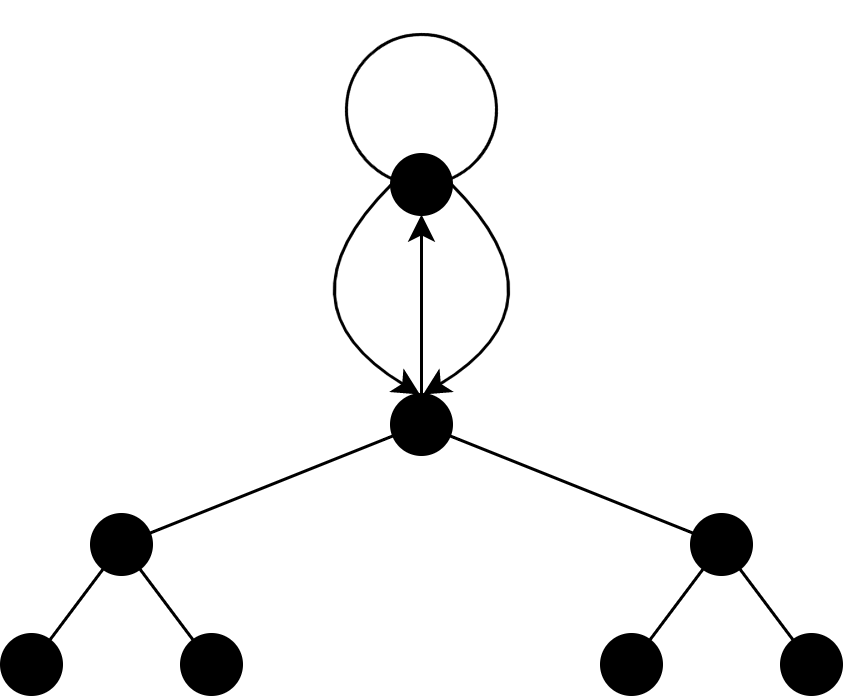}
\setlength{\abovecaptionskip}{10pt plus 0pt minus 0pt}
\caption{$\ff_0^2\Delta_K = -4, \ell = 2$ up to level $3$} 
\end{figure}

\vspace{-2em}

\begin{figure}[H]\label{-4split_and_inert} 
\includegraphics[scale=0.7]{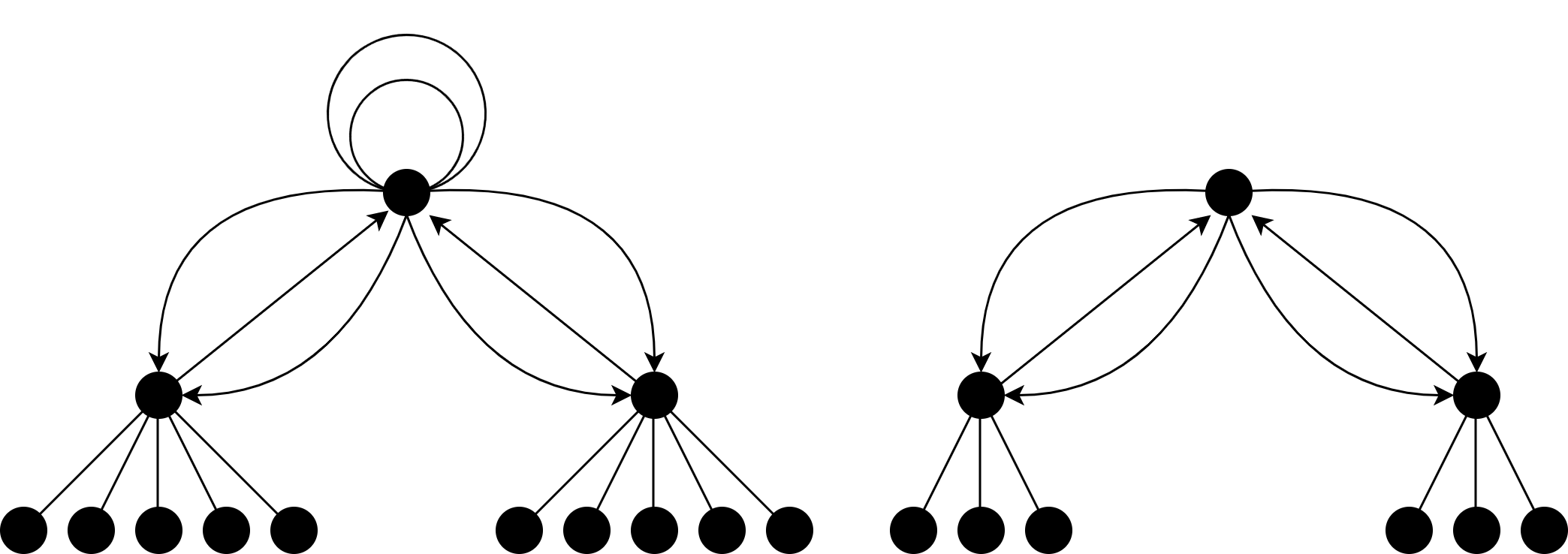}
\setlength{\abovecaptionskip}{10pt plus 0pt minus 0pt}
\caption{$\ff_0^2\Delta_K = -4$, $\ell$ split ($\ell=5$, left) and inert ($\ell=3$, right) up to level $2$} 
\end{figure}

\begin{figure}[H]\label{-3ram_and_2} 
\includegraphics[scale=0.7]{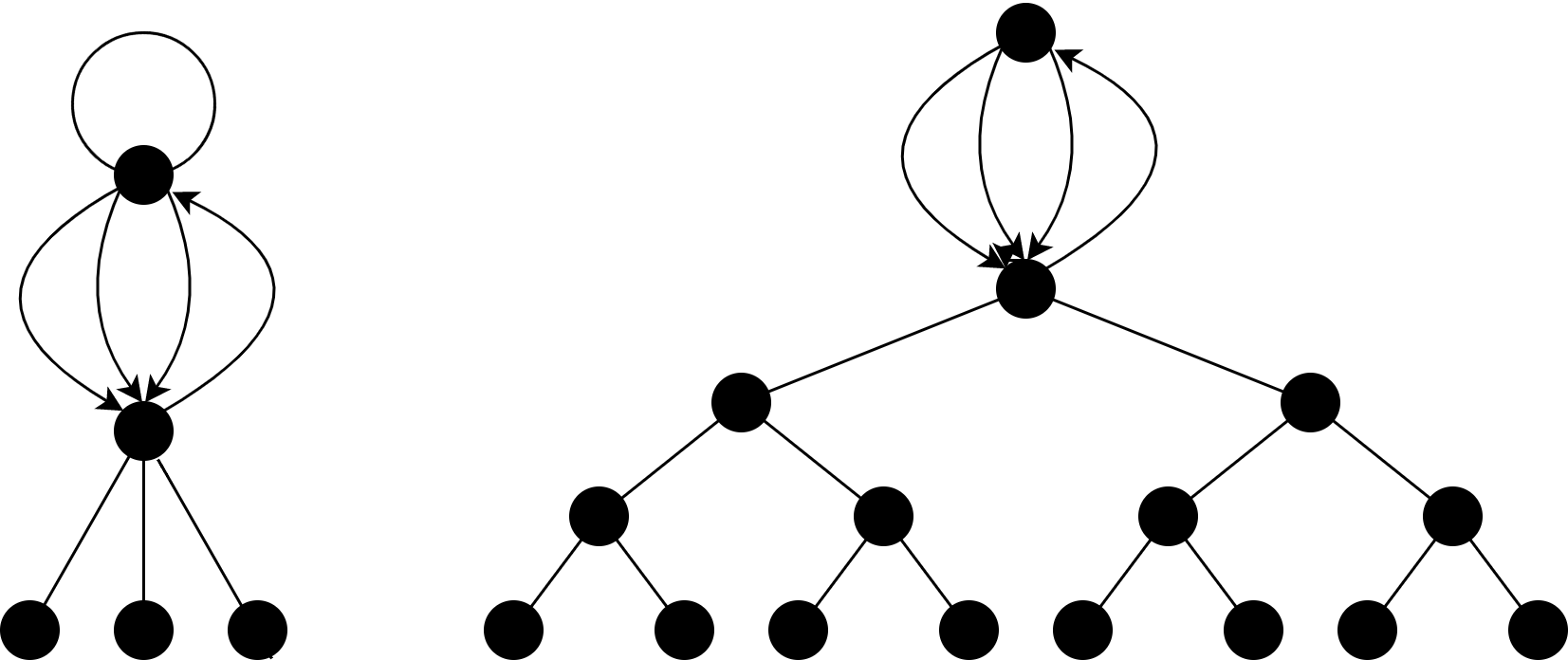}
\setlength{\abovecaptionskip}{10pt plus 0pt minus 0pt}
\caption{$\ff_0^2\Delta_K = -3, \ell = 3$ up to level $2$ (left) and $\ell=2$ up to level $3$ (right)} 
\end{figure}

\begin{figure}[H]\label{-3split_and_inert} 
\includegraphics[scale=0.7]{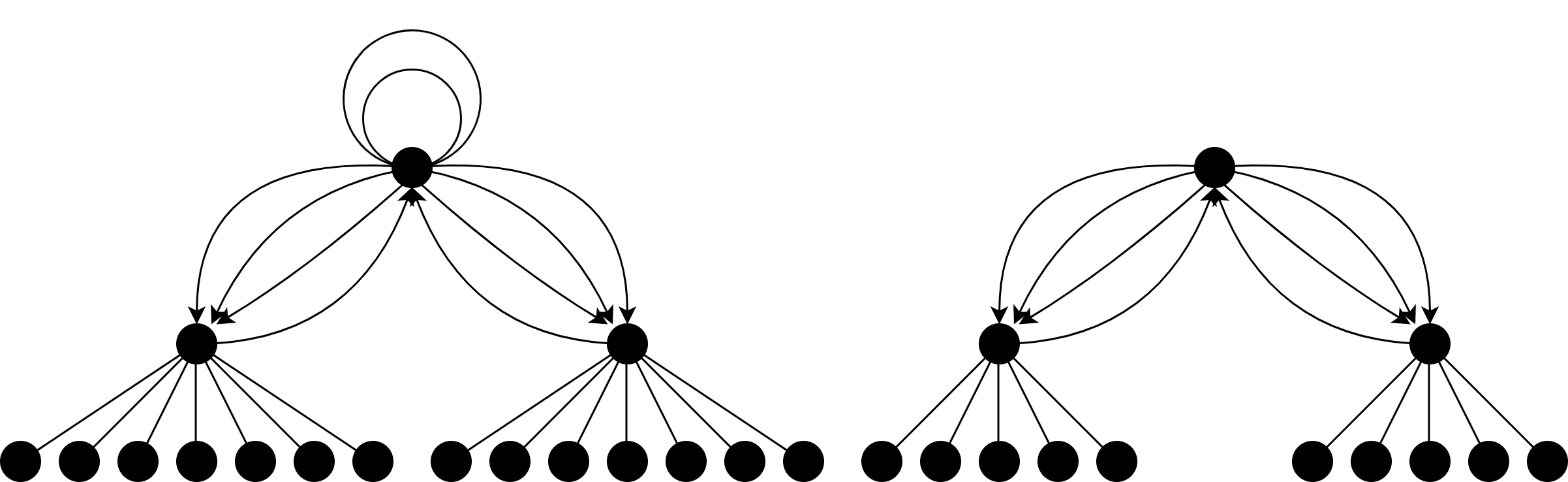}
\setlength{\abovecaptionskip}{10pt plus 0pt minus 0pt}
\caption{$\ff_0^2\Delta_K = -3$, $\ell$ split ($\ell=7$, left) and inert ($\ell=5$, right) up to level $2$} 
\end{figure}


\section{CM points on $X_0^D(\ell^a)_{/\mathbb{Q}}$}\label{CM_points_prime_power_section}

We fix $\ell^a$ a prime power and $\Delta = \ff^2\Delta_K = \ell^{2L} \ff_0^2\Delta_K$, with $\text{gcd}(\ff_0,\ell) = 1$, an imaginary quadratic discriminant. In this section, we describe the $\Delta$-CM locus on $X_0^D(\ell^a)_{/\mathbb{Q}}$. To this aim, we fully classify all closed point equivalence classes, by which we mean $\textnormal{Gal}(\overline{\mathbb{Q}}/\mathbb{Q})$ orbits, of non-backtracking, length $a$ paths in $\mathcal{G}^D_{K,\ell,\ff_0}$. We record the number of classes of each type with each possible residue field (up to isomorphism). 

In the $\ff_0^2\Delta_K \in \{-3,-4\}$ cases, the notion of backtracking in $\mathcal{G}^D_{K,\ell,1}$ has subtlety between levels $0$ and $1$ that is not present in isogeny volcanoes. We address this now: traversing \emph{any} edge from a vertex $v$ in level $0$ to a vertex $w$ in level $1$ followed by the single edge from $w$ to $v$ corresponds to a composition of dual isogenies, and thus is backtracking. On the other hand, for a given isogeny $\varphi$ corresponding to the edge $e$ from $w$ to $v$, there is a single edge from $v$ to $w$ corresponding to its dual $\widehat{\varphi}$. Therefore, traversing $e$ followed by the other edge (respectively, either of the two other edges) from $v$ to $w$ \emph{does not} count as backtracking in the case of $\ff_0^2\Delta_K = -4$ (respectively, $\ff_0^2\Delta_K = -3$). 

With $b$ denoting the number of prime divisors of $D$ which are inert in $K$, we have $2^b$ closed $\Delta$-CM points on $X^D(1)_{/\mathbb{Q}}$, with the fibers over each under the natural map from $X_0^D(\ell^a)_{/\mathbb{Q}}$ to $X^D(1)_{/\mathbb{Q}}$ being isomorphic via Atkin-Lehner involutions. In all cases, we then have 
\[ \sum_{C(\varphi)} e_\varphi d_\varphi = 2^b\text{deg}(X_0(\ell^a) \rightarrow X(1)) =  2^b \psi(\ell^a) = 2^b(\ell^a + \ell^{a-1}), \]
where our sum is over closed-point equivalence classes $C(\varphi)$ of QM-cyclic $\ell^a$ isogenies $\varphi$ with corresponding CM discriminant $\Delta$. 

The map $X_0^D(\ell^a)_{/\mathbb{Q}} \rightarrow X^D(1)_{/\mathbb{Q}}$ has nontrivial ramification over a closed $\Delta$-CM point if and only if $\Delta \in \{-3,-4\}$. For $\Delta \in \{-3,-4\}$ and path length $a$, we have that a closed point equivalence class has ramification, of index $2$ or $3$ in the respective cases of $\Delta = -4$ and $-3$, if and only if the paths in the class include a descending edge from level $0$ to level $1$. This allows for a check on the classifications and counts that we provide. 

If $D(K) = 1$, then the path types showing up in our analysis of each $\mathcal{G}^D_{K,\ell,\ff_0}$ are exactly those appearing in \cite{Cl22} and \cite{CS22}. In this case, each graph $\mathcal{G}^D_{K,\ell,\ff_0}$ consists of $2^b$ copies of the analogous graph $\mathcal{G}_{K,\ell,\ff_0}$ from the $D=1$ modular curve case. Moreover, we have shown that the action of relevant involutions on each component is identical to the action of complex conjugation in the $D=1$ case, up to symmetry of our graphs. In each place where the isomorphism class of a residue field in the referenced $D=1$ analysis is a rational ring class field, we have in its place here some totally complex, index $2$ subfield of a ring class field as described in Theorem \ref{JGR}. 

If at least one prime dividing $D$ is inert in $K$, i.e., if $D(K) > 1$, then all of the residue fields of $\Delta$-CM points on $X^D(1)_{/\mathbb{Q}}$, and hence on $X_0^D(\ell^a)_{/\mathbb{Q}}$, are ring class fields. The path types showing up are exactly those in \cite{CS22}, but the counts will in general differ from the case of the previous paragraph. Specifically, a given path type in our analysis in the case of $D(K) = 1$ consists of $m$ classes with corresponding residue field $K(\ff')$ and $n$ classes with corresponding residue field an index $2$ subfield of $K(\ff')$ for some $\ff' \in \mathbb{Z}^+$ and $m, n \geq 0$. In the case of $D(K)>1$, the same path type then consists of $2m+n$ classes, each with corresponding residue field $K(\ff')$. 

\begin{example}
Suppose that $K = \mathbb{Q}(i)$ splits $B$, and consider the case of $\Delta = -4$ and $\ell^a = 2$. We have $2^{b+1}$ closed-point equivalence classes of QM-cyclic $2$-isogenies of QM abelian surfaces with $-4$-CM. Each corresponding point on $X_0^D(2)_{/\mathbb{Q}}$ has residue degree $1$ over its image on $X^D(1)_{/\mathbb{Q}}$, having residue field $K$. Half of these classes, corresponding to self-loop edges at the surface, have no ramification, while each of the $2^b$ classes $C(\varphi)$ corresponding to a pair of descending edges to level $1$ has $e_\varphi = 2$.
\end{example}

A non-backtracking length $a$ path in $\mathcal{G}_{K,\ell,1}$ starting in level $L$ consists of $c$ ascending edges, followed by $h$ horizontal edges, followed by $d$ descending edges for some $c,h,d \geq 0$ with $c+h+d = a$. We denote this decomposition type of the path with the ordered triple $(c,h,d)$. 

\subsection{Path type analysis: general case}\label{path-analysis-general}
We begin here by considering the portion of the path type analysis that is independent of $\ell$ and $\Delta_K$. \\

{\bf{I.}} There are classes consisting of strictly descending paths, i.e., with $(c,h,d) = (0,0,a)$. If $D(K) \neq 1$, then there are $2^b$ such classes, each with residue field $K(\ell^{a}\ff)$. Otherwise, there are $2^{b}$ such classes, each with corresponding residue field an index $2$ subfield of $K(\ell^a \ff)$. 

{\bf{II.}} If $a \leq L$, there are classes of strictly ascending paths, i.e., with $(c,h,d) = (a,0,0)$. If $D(K) \neq 1$, then there are $2^b$ such classes, each with corresponding residue field $K(\ff)$. Otherwise, there are $2^b$ such classes, each with corresponding residue field an index $2$ subfield of $K(\ff)$. 

{\bf{III.}} If $L = 0$ and $\legendre{\Delta_K}{\ell} = 0$, then there classes of paths with $(c,h,d) = (0,1,a-1)$. If $D(K) \neq 1$, then there are $2^b$ such classes, each with corresponding residue field $K(\ell^{a-1} \ff)$. Otherwise, there are $2^{b}$ such classes, each with corresponding residue field an index $2$ subfield of $K(\ell^{a-1} \ff)$. 

{\bf{IV.}} If $L = 0$ and $\legendre{\Delta_K}{\ell} = 1$, then for each $h$ with $1 \leq h \leq a$ there are classes of paths with $(c,h,d) = (0,h,a-h)$ and residue field $K(\ell^{a-h} \ff)$. There are $2^{b+1}$ such classes if $D(K) \neq 1$, and there are $2^{b}$ such classes otherwise. 

{\bf{X.}} If $a > L \geq 1$ and $\legendre{\Delta_K}{\ell} = 1$, then there are classes of paths with $(c,h,d) = (L,a-L,0)$ and residue field $K(\ff)$. There are $2^{b+1}$ such classes if $D(K) \neq 1$, and there are $2^{b}$ such classes otherwise. 

\subsection{Path type analysis: $\ell > 2$}\label{path-analysis-lodd}
Here we assume that $\ell$ is an odd prime. 

{\bf{V.}} If $L \geq 2$, then for each $c$ with $1 \leq c \leq \text{min}\{a-1,L-1\}$ there are paths which ascend at least one edge but not all the way to the surface, and then immediately descend at least one edge, with $(c,h,d) = (c,0,a-c)$. Each such class has corresponding residue field $K(\ell^{\text{max}\{a-2c,0\}} \ff)$. There are $2^{b}(\ell-1) \ell^{\text{min}\{c,a-c\}-1}$ such paths if $D(K) \neq 1$, and $2^{b-1}(\ell-1) \ell^{\text{min}\{c,a-c\}-1}$ such paths otherwise. 

{\bf{VI.}} If $a \geq L+1 \geq 2$ and $\legendre{\Delta_K}{\ell} = -1$, then there are paths which ascend to the surface and then immediately descend at least one edge, with $(c,h,d) = (L,0,a-L)$. If $D(K) \neq 1$, then there are $2^{b} \ell^{\text{min}\{L,a-L\}}$ classes of such paths with corresponding residue field $K(\ell^{\text{max}\{a-2L,0\}}\ff)$. Otherwise, there are $2^{b-1}\left(\ell^{\text{min}\{L,a-L\}}-1\right)$ classes of such paths with corresponding residue field $K(\ell^{\text{max}\{a-2L,0\}}\ff)$, and $2^b$ classes of such paths with corresponding residue field an index $2$ subfield of $K(\ell^{\text{max}\{a-2L,0\}}\ff)$. 

{\bf{VII.}} If $a \geq L+1 \geq 2$ and $\legendre{\Delta_K}{\ell} = 0$, then there are paths which ascend to the surface and then immediately descend at least one edge, with $(c,h,d) = (L,0,a-L)$. Each such path has corresponding residue field $K(\ell^{\text{max}\{a-2L,0\}}\ff)$. If $D(K) \neq 1$, then there are $2^{b}(\ell-1)\ell^{\text{min}\{L,a-L\}-1}$ classes of such paths. Otherwise, there are $2^{b-1}(\ell-1)\ell^{\text{min}\{L,a-L\}-1}$ classes. 

{\bf{VIII.}} If $a \geq L+1 \geq 2$ and $\legendre{\Delta_K}{\ell} = 0$, then there are paths which ascend to the surface, follow one surface edge, and then possibly descend, with $(c,h,d) = (L,1,a-L-1)$. If $D(K) \neq 1$, then there are $2^{b} \ell^{\text{min}\{L,a-L-1\}}$ classes of such paths with corresponding residue field $K(\ell^{\text{max}\{a-2L-1,0\}}\ff)$. Otherwise, there are $2^{b-1}\left(\ell^{\text{min}\{L,a-L-1\}}-1\right)$ classes of such paths with corresponding residue field $K(\ell^{\text{max}\{a-2L-1,0\}}\ff)$, and $2^b$ classes of such paths with corresponding residue field an index $2$ subfield of $K(\ell^{\text{max}\{a-2L-1,0\}}\ff)$. 

{\bf{IX.}} If $a \geq L+1 \geq 2$ and $\legendre{\Delta_K}{\ell} = 1$, then there are paths which ascend to the surface and then immediately descend at least one edge, with $(c,h,d) = (L,0,a-L)$. If $D(K) \neq 1$, then there are $2^{b} (\ell-2)\ell^{\text{min}\{L,a-L\}-1}$ classes of such paths with corresponding residue field $K(\ell^{\text{max}\{a-2L,0\}}\ff)$. Otherwise, there are $2^{b-1}\left((\ell-2)\ell^{\text{min}\{L,a-L\}-1}-1\right)$ classes of such paths with corresponding residue field $K(\ell^{\text{max}\{a-2L,0\}}\ff)$, and $2^b$ classes of such paths with corresponding residue field an index $2$ subfield of $K(\ell^{\text{max}\{a-2L,0\}}\ff)$. 

{\bf{XI.}} If $a \geq L+2 \geq 3$ and $\legendre{\Delta_K}{\ell} = 1$, then for each $1 \leq h \leq a-L-1$ there are paths which ascend to the surface, traverse $h$ edges on the surface, and then descend at least one edge, with $(c,h,d) = (L,h,a-L-h)$. Each such path has corresponding residue field $K(\ell^{\text{max}\{a-2L-h,0\}}\ff)$. If $D(K) \neq 1$, then there are $2^{b+1}(\ell-1)\ell^{\text{min}\{L,a-L-h\}-1}$ classes of such paths. Otherwise, there are $2^b(\ell-1)\ell^{\text{min}\{L,a-L-h\}-1}$ classes.

\subsection{Path type analysis: $\ell = 2, \legendre{\Delta_K}{2} \neq 0$}\label{path-analysis-2-ord0}

Here we assume that $\ell = 2$ with $\Delta_K$ odd. 

{\bf{V.}} If $L \geq 2$, we have classes consisting of paths which ascend at least one edge but not all the way to the surface, and then immediately descend at least one edge. We have the following types:

\begin{itemize}
    \item[V$_1$.] If $a \geq 2$, then there are classes with $(c,h,d) = (1,0,a-1)$. If $D(K) \neq 1$, then there are $2^{b}$ such classes, each with corresponding residue field $K(2^{a-2}\ff)$. Otherwise, there are $2^b$ such classes, each with corresponding residue field an index $2$ subfield of $K(2^{a-2}\ff)$. 
    
    \item[V$_2$.] If $L \geq a \geq 3$, then there are classes with $(c,h,d) = (a-1,0,1)$. If $D(K) \neq 1$, then there are $2^{b}$ such classes, each with corresponding residue field $K(2^{a-2}\ff)$. Otherwise, there are $2^b$ such classes, each with corresponding residue field an index $2$ subfield of $K(2^{a-2}\ff)$. 
    
    \item[V$_3$.] If $a \geq L+1 \geq 4$, then there are paths with $(c,h,d) = (L-1,0,a-L+1)$. If $D(K) \neq 1$, there are $2^{\text{min}\{a-L+1,L-1\}+b-1}$ classes of such paths with corresponding residue field $K(2^{\text{max}\{a-2L+2,0\}}\ff)$. Otherwise, there are $2^b\left(2^{\text{min}\{a-L+1,L-1\}-2}-1\right)$ classes of such paths with corresponding residue field $K(2^{\text{max}\{a-2L+2,0\}}\ff)$, and $2^{b+1}$ classes of such paths with corresponding residue field an index $2$ subfield of $K(2^{\text{max}\{a-2L+2,0\}}\ff)$. 
    
    \item[V$_4$.] For each $c$ with $2 \leq c \leq \text{min}\{L-2, a-2\}$, there are paths with $(c,h,d) = (c,0,a-c)$. Each such path has corresponding residue field $K(2^{\text{max}\{a-2c,0\}} \ff)$. There are $2^{\text{min}\{c,a-c\}+b-1}$ equivalence classes of such paths if $D(K) \neq 1$. Otherwise, there are $2^{\text{min}\{c,a-c\}+b-2}$ such classes. 
    
\end{itemize}

{\bf{VI.}} If $a \geq L+1 \geq 2$ and $\legendre{\Delta_K}{2} = -1$, there are paths that ascend to the surface and then immediately descend at least one edge, with $(c,h,d) = (L,0,a-L)$. Each such class has corresponding residue field $K(2^{\text{max}\{a-2L,0\}}\ff)$. If $D(K) \neq 1$, then there are $2^{\text{min}\{L,a-L\}+b}$ classes of such paths. Otherwise, there are $2^{\text{min}\{L,a-L\}-1+b}$ such classes. 

{\bf{XI.}} If $a \geq L+2 \geq 3$ and $\legendre{\Delta_K}{2} = 1$, then for all $1 \leq h \leq a-L-1$ there are paths which ascend to the surface, traverse $h$ horizontal edges, and then descend at least once, with $(c,h,d) = (L,h,a-L-h)$. Each such class has corresponding residue field $K(2^{\text{max}\{a-2L-h,0\}}\ff)$. If $D(K) \neq 1$, then there are $2^{\text{min}\{L,a-L-h\}+b}$ classes of such paths. Otherwise, there are $2^{\text{min}\{L,a-L-h\}+b-1}$ such classes.

\subsection{Path type analysis: $\ell = 2, \textnormal{ord}_2(\Delta_K) = 2$}\label{path-analysis-2-ord2}

Here we assume that $\ell = 2$ with $\text{ord}_2(\Delta_K) = 2$. 

{\bf{V.}} If $L \geq 2$, we have classes consisting of paths which ascend at least one edge but not all the way to the surface, and then immediately descend at least one edge. We have the following types:

\begin{itemize}
    \item[V$_1$.] If $a \geq 2$, then there are classes with $(c,h,d) = (1,0,a-1)$. If $D(K) \neq 1$, then there are $2^{b}$ such classes, each with corresponding residue field $K(2^{a-2}\ff)$. Otherwise, there are $2^{b}$ such classes, each with corresponding residue field an index $2$ subfield of $K(2^{a-2}\ff)$. 
    
    \item[V$_2$.] If $L \geq a \geq 3$, then there are classes with $(c,h,d) = (a-1,0,1)$. If $D(K) \neq 1$, then there are $2^{b}$ classes of such paths, each with corresponding residue field $K(\ff)$. Otherwise, there are $2^{b}$ classes of such paths, each with corresponding residue field an index $2$ subfield of $K(\ff)$. 
    
    \item[V$_3$.] For each $c$ with $2 \leq c \leq \text{min}\{L-1, a-2\}$, there are paths $(c,h,d) = (c,0,a-c)$. Each such class has corresponding residue field $K(2^{\text{max}\{a-2c,0\}} \ff)$. If $D(K) \neq 1$, then there are $2^{\text{min}\{c,a-c\}+b-1}$ classes of such paths. Otherwise, there are $2^{\text{min}\{c,a-c\}+b-2}$ such classes. 
    
\end{itemize}

{\bf{VI.}} If $L \geq 1$, then we have paths which ascend to the surface and then immediately descend at least one edge, with $(c,h,d) = (L,0,a-L)$. We have the following cases:

\begin{itemize}
    \item[VI$_1$.] Suppose $L=1$ and $a \ge 2$. If $D(K) \neq 1$, then there are $2^{b}$ classes of such paths, each with corresponding residue field $K(2^{a-2}\ff)$. Otherwise, there are $2^b$ such classes, each with corresponding residue field an index $2$ subfield of $K(2^{a-2}\ff)$. 
    
    \item[VI$_2$.] Suppose $a = L+1 \geq 3$. If $D(K) \neq 1$, then there are $2^{b}$ classes of such paths, each with corresponding residue field $K(\ff)$. Otherwise, there are $2^b$ such classes, each with corresponding residue field an index $2$ subfield of $K(\ff)$. 
    
    \item[VI$_3$.] Suppose $a \geq L+2 \geq 4$. If $D(K) \neq 1$, then there are $2^{\text{min}\{L,a-L\}+b-1}$ classes of such paths, each with corresponding residue field $K(2^{\text{max}\{a-2L,0\}}\ff)$. Otherwise, there are $2^b\left(2^{\text{min}\{L,a-L\}-2}-1\right)$ classes of such paths with corresponding residue field $K(2^{\text{max}\{a-2L,0\}}\ff)$, and $2^{b+1}$ classes of such paths with corresponding residue field an index $2$ subfield of $K(2^{\text{max}\{a-2L,0\}}\ff)$. 
    
\end{itemize}

{\bf{VIII.}} If $a \geq L+1 \geq 2$, then we have paths which ascend to the surface, and then traverse the unique surface edge, and then possibly descend, with $(c,h,d) = (L,1,a-L-1)$. We have the following cases:

\begin{itemize}
    \item[VIII$_1$.] Suppose $a = L+1$. If $D(K) \neq 1$, then there are $2^{b}$ classes of such paths, each with corresponding residue field $K(\ff)$. Otherwise, there are $2^b$ such classes, each with corresponding residue field an index $2$ subfield of $K(\ff)$. 
    
    \item[VIII$_2$.] Suppose $a \geq L+2$. Each such path has corresponding residue field $K(2^{\text{max}\{a-2L-1,0\}}\ff)$. If $D(K) \neq 1$, then there are $2^{\text{min}\{L,a-1-L\}+b}$ classes of such paths. Otherwise, there are $2^{\text{min}\{L,a-1-L\}+b-1}$ such classes. 
\end{itemize}

\subsection{Path type analysis: $\ell = 2, \textnormal{ord}_2(\Delta_K) = 3$}

Here we assume that $\ell = 2$ with $\text{ord}_2(\Delta_K) = 3$. The types of paths occurring here are the same as in the previous section, owing to the fact that the structure of $\mathcal{G}^D_{K,\ell,\ff_0}$ here is the same as therein. The corresponding residue field counts may differ, though, as the Galois action differs. 

{\bf{V.}} The analysis of this type is exactly as in \S \ref{path-analysis-2-ord2}.

{\bf{VI.}} If $L \geq 1$, then we have paths which ascend to the surface and then immediately descend at least one edge, with $(c,h,d) = (L,0,a-L)$. We have the following cases:

\begin{itemize}
    \item[VI$_1$.] Suppose $L=1$ and $a \ge 2$. If $D(K) \neq 1$, then there are $2^{b}$ classes of such paths, each with corresponding residue field $K(2^{\text{a-2}}\ff)$. Otherwise, there are $2^{b}$ such classes, each with corresponding residue field an index $2$ subfield of $K(2^{\text{a-2}}\ff)$.
    
    \item[VI$_2$.] Suppose $a = L+1 \geq 3$. If $D(K) \neq 1$, then there are $2^{b}$ classes of such paths, each with corresponding residue field $K(\ff)$. Otherwise, there are $2^b$ such classes, each with corresponding residue field an index $2$ subfield of $K(\ff)$. 
    
    \item[VI$_3$.] If $a \geq L+2 \geq 4$, then each such class has corresponding residue field $K(2^{\text{max}\{a-2L,0\}}\ff)$. If $D(K) \neq 1$, then there are $2^{\text{min}\{L,a-L\}+b-1}$ such classes. Otherwise, there are $2^{\text{min}\{L,a-L\}+b-2}$ such classes.
    
\end{itemize}

{\bf{VIII.}} If $a \geq L+1 \geq 2$, then we have paths which ascend to the surface, and then traverse the unique surface edge, and then possibly descend, with $(c,h,d) = (L,1,a-L-1)$. We have the following cases:

\begin{itemize}
    \item[VIII$_1$.] Suppose $a = L+1$. If $D(K) \neq 1$, then there are $2^{b}$ classes of such paths, each with corresponding residue field $K(\ff)$. Otherwise, there are $2^{b}$ such classes, each with corresponding residue field an index $2$ subfield of $K(\ff)$. 
    
    \item[VIII$_2$.] Suppose that $a \geq L+2$. If $D(K) \neq -1$, then there are $2^{\text{min}\{L,a-1-L\}+b}$ classes of such paths, each with corresponding residue field $K(2^{\text{max}\{a-2L-1,0\}}\ff)$. Otherwise, there are $2^b\left(2^{\text{min}\{L,a-1-L\}-1}-1\right)$ classes of such paths with corresponding residue field $K(2^{\text{max}\{a-2L-1,0\}}\ff)$, and $2^{b+1}$ classes with corresponding residue field an index $2$ subfield of $K(2^{\text{max}\{a-2L-1,0\}}\ff)$. 
   
\end{itemize}

\subsection{Primitive residue fields of CM points on $X_0^D(\ell^a)_{/\mathbb{Q}}$} 

Fixing $\Delta$ an imaginary quadratic discriminant and $N \in \mathbb{Z}^+$ relatively prime to $D$, we say that a field $F$ is a \textbf{primitive residue field of a $\Delta$-CM point on $X_0^D(N)_{/\mathbb{Q}}$} if 
\begin{itemize}
\item there is a $\Delta$-CM point $x \in X_0^D(N)_{/\mathbb{Q}}$ with $\mathbb{Q}(x) \cong F$, and
\item there does not exists a $\Delta$-CM point $y \in X_0^D(N)_{/\mathbb{Q}}$ with $\mathbb{Q}(y) \cong L$ with $L \subsetneq F$. 
\end{itemize} 
The preceding path type analysis in this section allows us to determine primitive residue fields for prime power levels $N = \ell^a$. It follows from this analysis that, In all cases, there are at most $2$ primitive residue fields, and that each primitive residue field is either a ring classes field or an index $2$ subfield of a ring class field. 

The cases occurring here are in line with those in \cite{Cl22} and \cite{CS22}, though here the primitive residue fields depend on whether $D(K) = 1$. In particular, if some prime dividing $D$ is inert in $K$, then all residue fields of CM points on $X_0^D(\ell^a)$ are ring class fields and hence there can only be one primitive residue field. This necessarily happens, for instance, if the class number of $K$ is odd. We provide Case 1.5b with the alternative title of ``The dreaded case,'' as in \cite{Cl22}, to warn the reader that it will have an important role in later results on primitive residue fields and degrees. \\

\noindent \textbf{Case 1.1.} Suppose $\ell^a = 2$. 
\begin{itemize}[leftmargin=7em]
\item[\textbf{Case 1.1a.}] Suppose $\legendre{\Delta}{2} \neq -1$. If $D(K) = 1$, then the only primitive residue field is an index $2$ subfield of $K(\ff)$. Otherwise, the only primitive residue field is $K(\ff)$. 

\item[\textbf{Case 1.1b.}] Suppose $\legendre{\Delta}{2} = -1$. If $D(K) = 1$, then the only primitive residue field is an index $2$ subfield of $K(2\ff)$. Otherwise, the only primitive residue field is $K(2\ff)$. 
\end{itemize}

\noindent \textbf{Case 1.2.} Suppose $\ell^a > 2, L = 0$ and $\legendre{\Delta}{\ell} = 1$. If $D(K) = 1$, then the primitive residue fields are $K(\ff)$ and an index $2$ subfield of $K(\ell^a \ff)$. Otherwise, the only primitive residue field is $K(\ff)$. 

\noindent \textbf{Case 1.3.} Suppose $\ell^a > 2, L = 0$ and $\legendre{\Delta}{\ell} = -1$. If $D(K) = 1$, then the only primitive residue field is an index $2$ subfield of $K(\ell^a \ff)$. Otherwise, the only primitive residue field is $K(\ell^a \ff)$. 

\noindent \textbf{Case 1.4.} Suppose $\ell^a > 2, L = 0$ and $\legendre{\Delta}{\ell} = 0$. If $D(K) = 1$, then the only primitive residue field is an index $2$ subfield of $K(\ell^{a-1} \ff)$. Otherwise, the only primitive residue field is $K(\ell^{a-1} \ff)$. 

\noindent \textbf{Case 1.5.} Suppose $\ell > 2, L \geq 1$ and $\legendre{\Delta_K}{\ell} = 1$. 
\begin{itemize}[leftmargin=7em]
\item[\textbf{Case 1.5a.}] Suppose $a \leq 2L$. If $D(K) = 1$, then the only primitive residue field is an index $2$ subfield of $K(\ff)$. Otherwise, the only primitive residue field is $K(\ff)$. 

\item[\textbf{Case 1.5b\phantom{.}}] \textbf{(The dreaded case).} Suppose $a \geq 2L+1$. If $D(K) = 1$, then the primitive residue fields are $K(\ff)$ and an index $2$ subfield of $K(\ell^{a-2L}\ff)$. Otherwise, the only primitive residue field is $K(\ff)$. 
\end{itemize}

\noindent \textbf{Case 1.6.} Suppose $\ell > 2, L \geq 1$ and $\legendre{\Delta_K}{\ell} = -1$. 
\begin{itemize}[leftmargin=7em]
\item[\textbf{Case 1.6a.}] Suppose $a \leq 2L$. If $D(K) = 1$, then the only primitive residue field is an index $2$ subfield of $K(\ff)$. Otherwise, the only primitive residue field is $K(\ff)$. 

\item[\textbf{Case 1.6b.}] Suppose $a \geq 2L+1$. If $D(K) = 1$, then the only primitive residue field is an index $2$ subfield of $K(\ell^{a-2L}\ff)$. Otherwise, the only primitive residue field is $K(\ell^{a-2L}\ff)$.
\end{itemize}

\noindent \textbf{Case 1.7.} Suppose $\ell > 2, L \geq 1$ and $\legendre{\Delta_K}{\ell} = 0$. 
\begin{itemize}[leftmargin=7em]
\item[\textbf{Case 1.7a.}] Suppose $a \leq 2L+1$. If $D(K) = 1$, then the only primitive residue field is an index $2$ subfield of $K(\ff)$. Otherwise, the only primitive residue field is $K(\ff)$. 

\item[\textbf{Case 1.7b.}] Suppose $a \geq 2L+2$. If $D(K) = 1$, then the only primitive residue field is an index $2$ subfield of $K(\ell^{a-2L-1}\ff)$. Otherwise, the only primitive residue field is $K(\ell^{a-2L-1}\ff)$.
\end{itemize}

\noindent \textbf{Case 1.8.} Suppose $\ell = 2, a \geq 2, L \geq 1$ and $\legendre{\Delta_K}{2} = 1$. 
\begin{itemize}[leftmargin=7em]
\item[\textbf{Case 1.8a.}] Suppose $L = 1$. If $D(K) = 1$, then the primitive residue fields are $K(\ff)$ and an index $2$ subfield of $K(2^af)$. Otherwise, the only primitive residue field is $K(\ff)$.

\item[\textbf{Case 1.8b.}] Suppose $L \geq 2$ and $a \leq 2L-2$. If $D(K) = 1$, then the only primitive residue field is an index $2$ subfield of $K(\ff)$. Otherwise, the only primitive residue field is $K(\ff)$. 

\item[\textbf{Case 1.8c.}] Suppose $L \geq 2$ and $a \geq 2L-1$. If $D(K) = 1$, then the primitive residue fields are $K(\ff)$ and an index $2$ subfield of $K(2^{a-2L+2}\ff)$. Otherwise, the only primitive residue field is $K(\ff)$. 
\end{itemize} 

\noindent \textbf{Case 1.9} Suppose $\ell = 2, a \geq 2, L \geq 1$ and $\legendre{\Delta_K}{2} = -1$. 

\begin{itemize}[leftmargin=7em]
\item[\textbf{Case 1.9a.}] Suppose $L = 1$. If $D(K) = 1$, then the primitive residue fields are $K(2^{a-2}\ff)$ and an index $2$ subfield of $K(2^{a}\ff)$. Otherwise, the only primitive residue field is $K(2^{a-2}\ff)$. 

\item[\textbf{Case 1.9b.}] Suppose $L \geq 2$ and $a \leq 2L-2$. If $D(K) = 1$, then the only primitive residue field is an index $2$ subfield of $K(\ff)$. Otherwise, the only primitive residue field is $K(\ff)$. 

\item[\textbf{Case 1.9c.}] Suppose $L \geq 2$ and $a \geq 2L-1$. If $D(K) \neq 1$, then the primitive residue fields are $K(2^{\text{max}\{a-2L,0\}}\ff)$ and an index $2$ subfield of $K(2^{a-2L+2}\ff)$. Otherwise, the only primitive residue field is $K(2^{\text{max}\{a-2L,0\}}\ff)$. 
\end{itemize}

\noindent \textbf{Case 1.10} Suppose $\ell = 2, a \geq 2, L \geq 1, \legendre{\Delta_K}{2} = 0$ and $\text{ord}_2(\Delta_K) = 2$. 

\begin{itemize}[leftmargin=7em]
\item[\textbf{Case 1.10a.}] Suppose $a \leq 2L$. If $D(K) \neq 1$, then the only primitive residue field is an index $2$ subfield of $K(\ff)$. Otherwise, the only primitive residue field is $K(\ff)$. 

\item[\textbf{Case 1.10b.}] Suppose $a \geq 2L+1$. If $D(K) \neq 1$, then the primitive residue fields are $K(2^{a-2L-1}\ff)$ and an index $2$ subfield of $K(2^{a-2L}\ff)$. Otherwise, the only primitive residue field is $K(2^{a-2L-1}\ff)$. 
\end{itemize}

\noindent \textbf{Case 1.11} Suppose $\ell = 2, a \geq 2, L \geq 1, \legendre{\Delta_K}{2} = 0$ and $\text{ord}_2(\Delta_K) = 3$. 

\begin{itemize}[leftmargin=7em]
\item[\textbf{Case 1.11a.}] Suppose $a \leq 2L+1$. If $D(K) = 1$, then the only primitive residue field is an index $2$ subfield of $K(\ff)$. Otherwise, the only primitive residue field is $K(\ff)$. 

\item[\textbf{Case 1.11b.}] Suppose $a \geq 2L+2$. If $D(K) = 1$, then the only primitive residue field is an index $2$ subfield of $K(2^{a-2L-1}\ff)$. Otherwise, the only primitive residue field is $K(2^{a-2L-1}\ff)$. 
\end{itemize}

\subsection{Primitive degrees of CM points on $X_0^D(\ell^a)_{/\mathbb{Q}}$.}

A positive integer $d$ is a \textbf{primitive degree for a $\Delta$-CM point on $X_0^D(N)_{/\mathbb{Q}}$} if 
\begin{itemize}
\item there is a a $\Delta$-CM point of degree $d$ on $X_0^D(N)_{/\mathbb{Q}}$, and
\item there does not exist a $\Delta$-CM point on $X_0^D(N)_{/\mathbb{Q}}$ of degree properly dividing $d$. 
\end{itemize}
If $d$ is such a degree, then the residue field of a degree $d$ point on $X_0^D(N)_{/\mathbb{Q}}$ is a primitive residue field of a $\Delta$-CM point on $X_0^D(N)_{/\mathbb{Q}}$. For $N = \ell^a$ a prime power, we then have from the previous section that there are at most two primitive degrees. 

While there are several cases that admit two primitive residue fields when $D(K) = 1$, the only case admitting two primitive degrees is Case 1.5b (The dreaded case). In Case 1.5b, our two primitive residue fields are $K(\ff)$ and an index $2$ subfield $L$ of $K(\ell^{a-2L}\ff)$, with respective degrees $[K(\ff) : \mathbb{Q}] = 2h(\mathfrak{o}(\ff) )$ and $[L : \mathbb{Q}] = \ell^{a-2L} h(\mathfrak{o}(\ff) )$. As $\ell$ is odd, we indeed have two primitive degrees in this case.


\section{Algebraic results on residue fields of CM points on $X^D(1)_{/\mathbb{Q}}$}\label{algebraic-results-section}

We develop here algebraic number theoretic results on fields which arise as residue fields of CM points on $X^D(1)_{/\mathbb{Q}}$ which will feed into our main results. In particular, a determination of composita and tensor products of such fields will be needed in determining information about the CM locus on $X_0^D(N)_{/\mathbb{Q}}$ for general $N$ from information at prime-power levels. 

For an imaginary quadratic field $K$, we let $K(\ff)$ denote the ring class field corresponding to the imaginary quadratic order $\mathfrak{o}(\ff)$ of conductor $\ff$ in $K$, i.e., that of discriminant $\ff^2\Delta_K$. 

\begin{proposition}\label{composita}
Let $K$ denote an imaginary quadratic field of discriminant $\Delta_K$. 
\begin{enumerate}
\item If $\Delta_K \not \in \{-3,-4\}$, then for any $\ff_1, \ff_2 \in \mathbb{Z}^+$ we have
\[ K(\ff_1) \cdot K(\ff_2) = K(\textnormal{lcm}(\ff_1,\ff_2)). \]

\item Suppose $\Delta_K \in \{-3,-4\}$. 

\begin{enumerate}
\item For any $\ff_1, \ff_2 \in \mathbb{Z}^+$ with $\text{gcd}(\ff_1,\ff_2) > 1$, we have
\[ K(\ff_1) \cdot K(\ff_2) = K(\textnormal{lcm}(\ff_1,\ff_2)). \]

\item If the class number of the order of discriminant $\ff_1^2\Delta_K$ is $1$, i.e., if $\ff_1^2\Delta_K \in S = \{-3,-4,-12,-16,-27\}$, then 
\[ K(\ff_1) \cdot K(\ff_2) = K(\ff_2). \]

\item Suppose we have positive integers $\ff_1, \ldots ,f_r$ which are all pairwise relatively prime and not in the $S$ defined above. Then $K(\ff_1) \cdots K(f_r) \subsetneq K(\ff_1 \cdots f_r)$, with 

\[ [K(\ff_1 \cdots f_r) : K(\ff_1) \cdots K(f_r)] = \begin{cases} 2^{r-1} \text{ if } \Delta_K = -4 \\
3^{r-1} \text{ if } \Delta_K = -3 \end{cases}. \]

\end{enumerate}

\item In all cases, $K(\ff_1)$ and $K(\ff_2)$ are linearly disjoint over $K(\textnormal{gcd}(\ff_1,\ff_2))$. 

\end{enumerate}
\end{proposition}

\begin{proof}
Part (1) is \cite[Prop. 2.10]{Cl22}, while part (2) is \cite[Prop. 2.1]{CS22} and part (3) follows from the combination of these two propositions. 
\end{proof}

We now use Proposition \ref{composita} to get analogs of \cite[Prop. 2.10]{Cl22} and \cite[Prop. 2.2]{CS22}, in which ``rational ring class fields'' are exchanged for those index $2$ subfields of rings class fields which arise as residue fields of CM points on $X^D(1)_{/\mathbb{Q}}$. 

\begin{corollary}\label{tensor_products1}
Suppose that $x_1 \in X_0^D(N_1)_{/\mathbb{Q}}$ and $x_2 \in X_0^D(N_2)_{/\mathbb{Q}}$ are $\mathfrak{o}(\ff) $-CM points, where $\mathfrak{o}(\ff) $ is an imaginary quadratic order in $K$. For $i = 1, 2$, let $\ff_i \in \mathbb{Z}^+$ such that 
\[ K \cdot \mathbb{Q}(x_i) \cong K(\ff_i). \]
Let $M = \textnormal{gcd}(N_1,N_2)$ and $m = \textnormal{lcm}(N_1,N_2)$, and suppose that $x \in X_0^D(M)$ is a point lying above $x_1$ and $x_2$ which is fixed by an involution $\sigma \in \textnormal{Gal}(K(M)/K)$. Let $\pi : X_0^D(M)_{/\mathbb{Q}} \rightarrow X^D(1)_{/\mathbb{Q}}$ denote the natural map. Then 
\begin{enumerate}
\item The fields $\mathbb{Q}(x_1)$ and $\mathbb{Q}(x_2)$ are linearly disjoint over $\mathbb{Q}(\pi(x))$. 
\item We have 
\[ \mathbb{Q}(x_1) \otimes_{\mathbb{Q}(\pi(x))} \mathbb{Q}(x_2) \cong \mathbb{Q}(x). \]
\item  We have 
\[ \mathbb{Q}(x_1) \otimes_{\mathbb{Q}(\pi(x))} K(x_2) \cong K(x). \]
\end{enumerate}
\end{corollary} 

\begin{proof}
    The ring class fields $K(\ff_1)$ and $K(\ff_2)$ are linearly disjoint over $K(m)$ by Proposition \ref{composita}. That $\mathbb{Q}(x_1)$ and $\mathbb{Q}(x_2)$ are linearly disjoint over $\mathbb{Q}(\pi(x))$, and that
    \[ [\mathbb{Q}(x) : \mathbb{Q}(x_1) \cdot \mathbb{Q}(x_2)] = [K(x) : K(x_1) \cdot K(x_2] = [K(m) : K(\ff_1) \cdot K(\ff_2)], \]  
follow by the same type of arguments as in the analogous case of rational ring class fields in \cite[Prop. 2.10]{Cl22} and \cite[Prop. 2.2]{CS22}, using that $K(x) \cong K(x_1)\cdot K(x_2) \cong K(\ff_1)\cdot K(\ff_2)$ via Proposition \ref{red_to_prime_powers} (note that the assumption that $x$ is fixed by $\sigma$ forces $\ff^2\Delta_K < -4$, so this proposition applies.). 

Part (2) now follows from the preceding remarks, combined with Proposition \ref{composita}. As for part (3), first note that the fact that $\mathbb{Q}(x)$ is fixed by some involution $\sigma \in \textnormal{Gal}(K(M)/K)$ immediately implies that $h(\mathfrak{o}(\ff)) > 1$ (as $X^D(1)_{/\mathbb{Q}}$ has no real points). We note that the map
        \begin{align*} 
        \mathbb{Q}(x_1) \times K(x_2) &\longrightarrow K(x_1)\cdot K(x_2) \\
        (x_1,x_2) &\longmapsto x_1\cdot x_2
        \end{align*}
        is $\mathbb{Q}(\pi(x))$-bilinear, and the induced map on the tensor product over $\mathbb{Q}(\pi(x))$ must be an isomorphism 
        \[ \mathbb{Q}(x_1) \otimes_{\mathbb{Q}(\pi(x))}  K(x_2) \cong K(x_1) \cdot K(x_2) \]
        as the two finite $\mathbb{Q}(\ff)$-algebras here have the same dimension. The result then follows as $K(x_1) \cdot K(x_2) \cong K(x)$. 
\end{proof}

\begin{corollary}\label{tensor_products2}
Suppose that $x_1, x_2, \ldots, x_r$ are $\mathfrak{o}(\ff) $-CM points with $x_i \in X_0^D(N_i)_{/\mathbb{Q}}$ for each $i=1,\ldots, r$, where $\mathfrak{o}(\ff) $ is an imaginary quadratic order in $K$. For each $i = 1, \ldots, r$, let $\ff_i \in \mathbb{Z}^+$ such that 
\[ K \cdot \mathbb{Q}(x_i) \cong K(\ff_i). \]
Let $M = \textnormal{gcd}(N_1,\ldots, N_r)$ and $m = \textnormal{lcm}(N_1,\ldots, N_r)$. Let $\pi : X_0^D(M)_{/\mathbb{Q}} \rightarrow X^D(1)_{/\mathbb{Q}}$ denote the natural map. Let $S = \{-3, -4, -12, -16, -27\}$ be the set of discriminants of imaginary quadratic orders of class number $1$ with $\Delta_K \in \{-3,-4\}$. 
\begin{enumerate}
\item Suppose that $r=2$. If $\ff_1 \in S$, then we have
\[ K(x_1) \otimes_{\mathbb{Q}(\pi(x))} K(x_2) \cong K(x_2) \times K(x_2). \]
Now assuming $\ff_1, \ff_2 \not \in S$, if $\Delta_K < -4$ or if $\text{gcd}(\ff_1,\ff_2) > 1$ then 
\[ K(x_1) \otimes_{\mathbb{Q}(\pi(x))} K(x_2) \cong K(M) \times K(M). \]
\item Suppose that $\Delta_K \in \{-3,-4\}$, that $\ff_1, \ldots, f_r \not \in S$, and that $\ff_1, \ldots, f_r$ are all pairwise relatively prime. We then have
\[ \mathbb{Q}(x_1) \otimes_{\mathbb{Q}(\pi(x))} \ldots \otimes_{\mathbb{Q}(\pi(x))} \mathbb{Q}(x_r) \cong  K(x_1) \otimes_{\mathbb{Q}(\pi(x))} \ldots \otimes_{\mathbb{Q}(\pi(x))} K(x_r) \cong L^r, \]
with $L$ a subfield of $K(M)$ of index $2^{r-1}$ if $\Delta_K = -4$ and index $3^{r-1}$ if $\Delta_K = -3$. 
\end{enumerate} 
\end{corollary}
\begin{proof}
\begin{enumerate}
\item Using part $(3)$ of Corollary \ref{tensor_products1}, we have
        \begin{align*} 
        K(x_1) \otimes_{\mathbb{Q}(\pi(x))} K(x_2) &\cong \left(\mathbb{Q}(x_1) \otimes_{\mathbb{Q}(\pi(x))} K(x)\right) \otimes_{\mathbb{Q}(\pi(x))} \left(\mathbb{Q}(x_2) \otimes_{\mathbb{Q}(\pi(x))} K(x)\right) \\
        &\cong \left(\mathbb{Q}(x_1) \otimes_{\mathbb{Q}(\pi(x))} \mathbb{Q}(x_2)\right) \otimes_{\mathbb{Q}(\pi(x))} \left( K(x) \otimes_{\mathbb{Q}(\pi(x))} K(x) \right) \\
        &\cong \left(\mathbb{Q}(x_1) \otimes_{\mathbb{Q}(\pi(x))} \mathbb{Q}(x_2)\right) \otimes_{\mathbb{Q}(\pi(x))} \left( K(x) \times K(x) \right) \\
        &\cong \left(\mathbb{Q}(x_1) \otimes_{\mathbb{Q}(\pi(x))} K(x_2)\right) \times \left(\mathbb{Q}(x_1) \otimes_{\mathbb{Q}(\pi(x))} K(x_2)\right). 
        \end{align*}
The stated result then follows from another use of Corollary \ref{tensor_products1} part (3) if $\mathbb{Q}(x_1)$ properly embeds in the ring class field $K(\ff_1)$. Otherwise, $\mathbb{Q}(x_i) \cong K(\ff_i)$ for $i = 1, 2$ and $\mathbb{Q}(\pi(x)) \cong K(\ff)$. The case of $\ff_1 \in S$ is then clear, so assume $\ff_1, \ff_2 \not \in S$ and at least one of $\Delta_K < -4$ or $\text{gcd}(\ff_1,\ff_2) > 1$ holds. It then follows from Proposition \ref{composita} that 
\begin{align*} 
K(x_1) \otimes_{\mathbb{Q}(\pi(x))} K(x_2) &\cong \left( K(\ff_1) \otimes_{K(\ff)} K(\ff_2)\right) \times \left(K(x_1) \otimes_{K(\ff)} K(\ff_2)\right) \\
&\cong K(M) \times K(M).
\end{align*}
\item This result follows similarly to the above argument using Proposition \ref{composita} once more. Note that our assumption that the $\ff_i$ are relatively prime forces $\mathbb{Q}(x_i)$ to be a ring class field for each $i$; this assumption gives $K \cdot \mathbb{Q}(\pi(x)) \cong K(1) = K$ as $\Delta_K \in \{-3,-4\}$, and our Shimura curves have no real points so indeed $\mathbb{Q}(\pi(x)) \cong K$. \qedhere
\end{enumerate}
\end{proof}


\section{CM points on $X_0^D(N)_{/\mathbb{Q}}$}\label{CM_points_level_N}

In this section, we describe the $\Delta$-CM locus on $X_0^D(N)_{/\mathbb{Q}}$ for any $N \in \mathbb{Z}^+$ relatively prime to $D$ and any imaginary quadratic discriminant $\Delta$. For $\Delta < -4$, this description is possible using the foundations we have built thus far, specifically Propositions \ref{red_to_prime_powers} and \ref{fom-prime-power}, along with the path type analysis in \S \ref{CM_points_prime_power_section}. For $\Delta = \Delta_K \in \{-3,-4\}$, however, Proposition \ref{red_to_prime_powers} does not apply. 

We first elaborate on the description in former case, and then provide a result for compiling across prime powers in the case of $\Delta \in \{-3,-4\}$. Following this, we discuss primitive residue fields and degrees of $\Delta$-CM points on $X_0^D(N)_{/\mathbb{Q}}$. 

\subsection{Compiling across prime powers: $\Delta < -4$}\label{compiling_section1}

For a fixed prime $\ell$ relatively prime to $D$, let $\Delta = \ell^{2L}\ff_0^2\Delta_K$ with $\text{gcd}(\ff_0,\ell) = 1$ be an imaginary quadratic discriminant. Fixing $a \in \mathbb{Z}^+$, consider the natural map $\pi :X^D_0(\ell^a)_{/\mathbb{Q}} \rightarrow X^D(1)_{/\mathbb{Q}}$ and the fiber $\pi^{-1}(x)$ over a $\Delta$-CM point $x \in X_0^D(1)_{/\mathbb{Q}}$. Tthere are $2^b$ such fibers by Theorem \ref{fake-ell-curves-ell-curves-correspondence}, and any two are isomorphic via an Atkin--Lehner involution $w_p$ for some prime $p \mid D$ which is inert in $K$. We then have $\pi^{-1}(x) \cong \spec A$ with 
\begin{equation}\label{A_finite_etale_form}
A = \prod_{j=0}^{a} {L_j}^{b_j} \times \prod_{k=0}^{a} K(\ell^k \ff)^{c_k}
\end{equation}
for some non-negative integers $b_j, c_k$, where $L_j$ is an index $2$ subfield of $K(\ell^j \ff)$ for all $0 \leq j \leq a$. The explicit values $b_j$ and $c_k$, based on $\ell^a$ and $\Delta$, are determined by our path type analysis in \S \ref{CM_points_prime_power_section}. 

Next assume $\Delta < -4$, let $N$ denote a positive integer relatively prime to $D$, and consider the fiber $\pi^{-1}(x)$ of the map $\pi: X^D_0(N)_{/\mathbb{Q}} \rightarrow X(1)_{/\mathbb{Q}}$ over a $\Delta$-CM point $x \in X^D(1)_{/\mathbb{Q}}$. Let $N = \ell_1^{a_1} \cdots \ell_r^{a_r}$ be the prime-power factorization of $N$, and for each $1 \leq i \leq r$ consider the fiber $\pi_i^{-1}(x)$ of $\pi_i : X^D_0(\ell_i^{a_i})_{\mathbb{Q}} \rightarrow X^D(1)_{/\mathbb{Q}}$ over $x$. We then have 
\[ \pi_i^{-1}(x) \cong \spec A_i ,\] 
with each $A_i$ of the form given in (\ref{A_finite_etale_form}). Proposition \ref{red_to_prime_powers} then provides that $\pi^{-1}(x) \cong \spec A$ with 
\[ A = A_1 \otimes_{\mathbb{Q}(x)} \cdots \otimes_{\mathbb{Q}(x)} A_r. \]
It follows that $A$ is a direct sum of terms of the form 
\[ M = M_1 \otimes_{\mathbb{Q}(x)} \cdots \otimes_{\mathbb{Q}(x)} M_r, \]
where for each $1 \leq i \leq r$ we have that $M_i$ is isomorphic to $K(\ell_i^{j_i} \ff)$, or a totally complex index $2$ subfield thereof, for some $0 \leq j_i \leq a_i$. 

Let $s$ be the number of indices $1 \leq i \leq r$ such that $K$ is contained in $M_i$, i.e., such that $M_i \cong K(\ell_i^{j_i})$ is a ring class field. The results of \S \ref{algebraic-results-section} then tell us that
\[ M \cong \begin{cases} L \quad &\text{ if } s = 0 \\
K(\ell_1^{j_1} \cdots \ell_r^{j_r} \ff)^{2^{s-1}} \quad &\text{ otherwise,} \end{cases}  \]
where $L \subsetneq K(\ell_1^{j_1} \cdots \ell_r^{j_r})$ is a totally complex, index $2$ subfield in the $s = 0$ case. (Note that $\ell_i^{j_i} \Delta \in \{-12,-16,-27\}$ can only occur, due to the $\Delta < -4$ assumption, if $j_i = 0$, so these possibilities do not require special attention here.) 

\subsection{Compiling across prime powers: $\Delta \in \{-3,-4\}$}\label{compiling_section2}

Here, we determine how to compile residue field information across prime-power level for $\Delta \in \{-3,-4\}$. Our result here should be compared to \cite[Prop. 82, Thm. 8.3]{CS22}, wherein more work is required due to the fact that residue fields of $-3$ and $-4$-CM points on $X_0(N)_{/\mathbb{Q}}$ do not always contain the CM field $K$. 

\begin{theorem}\label{compiling_thm_extended}
Let $N \in \mathbb{Z}^+$ coprime to $D$ with prime-power factorization $N = \ell_1^{a_1} \cdots \ell_r^{a_r}$, and suppose $x \in X_0^D(N)_{/\mathbb{Q}}$ is a $\Delta$-CM point with $\Delta \in \{-3,-4\}$. Let $\pi_i: X_0^D(N)_{/\mathbb{Q}} \rightarrow X_0^D(\ell_1^{a_1})_{/\mathbb{Q}}$ denote the natural map and let $x_i = \pi_i(x)$ for each $1 \leq i \leq r$. Let $P_i$ be any path in the closed-point equivalence class of paths in $\mathcal{G}^D_{K,\ell_i,1}$ corresponding to $x_i$, and let $d_i \geq 0$ be the number of descending edges in $P_i$ (which is independent of the representative path). We then have 
\[ \mathbb{Q}(x) \cong K(\ell_1^{d_1} \cdots \ell_r^{d_r}). \] 
\end{theorem}
\begin{proof}
Because $\Delta \in \{-3,-4\}$, we know that the residue field of the image of $x$ under the natural map to $X^D(1)_{/\mathbb{Q}}$ is necessarily $K$. Therefore, $K \subseteq \mathbb{Q}(x_i)$ for each $i$ and hence $\mathbb{Q}(x_i) \cong K(\ell_i^{d_i})$ for each $1 \leq i \leq r$. 

Let $\varphi: (A,\iota) \rightarrow (A',\iota')$ be a QM-cyclic $N$-isogeny over $\mathbb{Q}(x)$ inducing $x$ (necessarily there is such an isogeny, as $K \subseteq \mathbb{Q}(x)$). Let $Q = \text{ker}(\varphi)$, let $C = e_1(Q)$, and for each $1 \leq i \leq r$ let $C_i \leq C$ be the Sylow $\ell_i$ subgroup of $C$. Let $\varphi_i : (A,\iota) \rightarrow (A/(\mathcal{O}\cdot C_i), \iota_i)$ be the $\ell_i$-primary part of $\varphi$, and let $\ff_i$ denote the central conductor of $(A/(\mathcal{O} \cdot C_i), \iota)$ (where by $\iota$ here we really mean the induced QM structure on the quotient). Put
\[ \mathcal{I} := \{i \mid d_i > 0\} = \{i \mid  \text{ord}_{\ell_i}(\ff_i) > 0 \} \subseteq \{1, \ldots, r\} \]
and 
\[ Q' := \left \langle \left \{\mathcal{O} \cdot C_i \right \}_{i \in \mathcal{I}} \right \rangle \leq Q. \]
Our original isogeny $\varphi$ then factors as $\varphi = \varphi'' \circ \varphi'$ where $\varphi' : (A,\iota) \rightarrow (A/Q', \iota)$. Because a QM-cyclic $\ell_i$-isogeny preserves the prime-to-$\ell_i$ part of the central conductor, the central conductor of $(A/Q', \iota)$ must be divisible by $\ell_1^{d_1} \cdots \ell_r^{d_r}$. We then have 
\[ K(\ell_1^{d_1} \cdots \ell_r^{d_r}) \subseteq \mathbb{Q}(\varphi') \subseteq \mathbb{Q}(\varphi), \]
and it remains to show the reverse containment. If the central conductor of $(A',\iota')$ is also $1$, then (up to isomorphism on the target) $\varphi$ is a QM-equivariant endomorphism of $(A,\iota)$ and therefore $\mathbb{Q}(\varphi) \subseteq K$ as desired. Otherwise, the dual isogeny $\varphi^\vee$ induces a $\Delta'$-CM point $x' \in X_0^D(N)_{/\mathbb{Q}}$ with $\Delta' < -4$. We have $\mathbb{Q}(x) \cong \mathbb{Q}(x')$, and the claim then holds via an application of Proposition \ref{red_to_prime_powers} to $x'$. 
\end{proof}

\subsection{The main algorithm}\label{algorithm_proof}

We have now built up all we need to prove our main result, Theorem \ref{algorithm_thm}. 

\begin{proof}[proof of Theorem \ref{algorithm_thm}]
The existence and structure of this algorithm follows from our prior results. We summarize the steps of the algorithm with appropriate references for individual steps here: 

\begin{algorithm}[The $\mathfrak{o}$-CM-locus on $X_0^D(N)_{/\mathbb{Q}}$]\label{algorithm}\ \\

\noindent \textnormal{\textbf{Input:} an indefinite quaternion discriminant $D$ over $\mathbb{Q}$, a positive integer $N$ coprime to $D$, an imaginary quadratic discriminant $\Delta_K$ and a positive integer $f$}  

\noindent \textnormal{\textbf{Output:} the complete list of tuples $(\texttt{is{\_}fixed},f',e,c)$, consisting of a boolean $\texttt{is{\_}fixed}$, a positive integer $f'$, an integer $e \in \{1,2,3\}$ and a positive integer $c$, such that there exist exactly $c$ closed $f^2\Delta_K$-CM points $x$ on $X_0^D(N)_{/\mathbb{Q}}$ with $K(x) \cong K(f')$, with $\mathbb{Q}(x) \cong K(f')$ if $\texttt{is{\_}fixed}$ is False and with $[K(f') : \mathbb{Q}(x)] = 2$ otherwise and with ramification index $e$ with respect to the natural map to $X^D(1)_{/\mathbb{Q}}$.}

\noindent \textnormal{\textbf{Steps:}}
\begin{itemize}
\item \textnormal{Compute the prime-power factorization $N = \ell_1^{a_1} \cdots \ell_r^{a_r}$ of $N$. }
\item \textnormal{For each index $i \in \{1,\ldots, r\}$, compute using the path type enumeration results of \S \ref{CM_points_prime_power_section} information on all $f^2\Delta_K$-CM points on $X_0^D(\ell_i^{a_i})_{/\mathbb{Q}}$. This information is stored as a list of lists $(\texttt{is{\_}fixed}_i,f_i,e_i,c_i)$ as in our desired output at general level. (If $D=1$, this information is originally obtained in the path-type analysis at prime-power level given in \cite{Cl22} and \cite{CS22}.)}
\item \textnormal{For each tuple $(P_1, \ldots, P_r)$, in which each $P_i$ is the information of an $f^2\Delta_K$-CM point on $X_0^D(\ell_i^{a_i})_{/\mathbb{Q}}$ of the form $(\texttt{is{\_}fixed}_i,f_i,e_i,c_i)$ as computed in the previous part, compute the information $(\texttt{is{\_}fixed},f',e,c)$ of all $f^2\Delta_K$-CM points on $X_0^D(N)_{/\mathbb{Q}}$ with image a point with information given by $P_i$ under the natural map to $X_0^D(\ell_i^{a_i})_{/\mathbb{Q}}$ for all $i \in \{1,\ldots, r\}$. This is done as follows:}
\begin{itemize}
\item \textnormal{The boolean $\texttt{is{\_}fixed}$ is true if and only if the boolean $\texttt{is{\_}fixed}_i$ is true for all $i \in \{1,\ldots, r\}$ by Proposition \ref{red_to_prime_powers} and the results of \S \ref{algebraic-results-section}.} 
\item \textnormal{The CM conductor $f'$ of such a point is equal to the least common multiple of the conductors $f_1,\ldots, f_r$ at each prime-power level. This is by Proposition \ref{red_to_prime_powers} and the results of \S \ref{algebraic-results-section}, as also spelt out at the start of \S \ref{compiling_section1}, if $\Delta_K < -4$, and is Theorem \ref{compiling_thm_extended} in the case of $\Delta_K \in \{-3,-4\}$.}
\item \textnormal{The ramification index $e$ is equal to the maximum among the indices $e_i$ (so in particular is $2$ or $3$ if and only if $f_0^2\Delta_K \in \{-3,-4\}$ and at least one of the $P_i$ has $e_i = 2$ or $e_i = 3$).}
\item \textnormal{If $\Delta < -4$, then the count $c$ is given by the results of \S \ref{compiling_section1}. If $\Delta \in \{-3,-4\}$, then this count is given by the results of \S \ref{compiling_section2} if $D>1$ and is given by the results of \cite{CS22} if $D=1$.} \qedhere
\end{itemize}
\end{itemize}
\end{algorithm}
\end{proof}

This algorithm has been implemented, and is available as the function $\texttt{CM{\_}points{\_}XD0}$ in the file $\texttt{shimura{\_}curve{\_}CM{\_}locus.m}$ in \cite{Rep}.

\subsection{Primitive residue fields of CM points on $X^D_0(N)_{/\mathbb{Q}}$}\label{prim_res_flds_N}

The preceding results imply that the residue field of any $\Delta$-CM point on $X^D_0(N)_{/\mathbb{Q}}$ is isomorphic to either a ring class field or a totally complex, index $2$ subfield of a ring class field as described in Theorem \ref{JGR}. As a result, there are at most two primitive residue fields of $\Delta$-CM points on $X^D_0(N)_{/\mathbb{Q}}$. Moreover, there exists a positive integer $C$ such that an index $2$ subfield of $K(Cf)$ is a primitive residue field of a $\Delta$-CM point on $X^D_0(N)_{/\mathbb{Q}}$ if and only if for each $1 \leq i \leq r$ there exists a positive integer $C_i$ such that an index $2$ subfield of $K(C_i \ff)$ is a primitive residue field of a $\Delta$-CM point on $X^D_0(\ell_i^{a_i})$. 

We begin by investigating the cases in which we do have such a field as a primitive residue field, determining when we have two primitive residue fields and, if so, whether we have two primitive degrees of residue fields. Note that this assumption requires $D(K) = 1$, and hence $\Delta < -4$. Let $H_i = \ell_i^{h_i} \mid \ell_i^{a_i}$ be the unique positive integer such that an index $2$ subfield $L_i$ of $K(H_i \ff)$ is a primitive residue field of a $\Delta$-CM point on $X^D_0(\ell_i^{a_i})_{/\mathbb{Q}}$ for each $1 \leq i \leq r$. Setting
\[ H = H_1\cdots H_r ,\]
we have that a totally complex, index $2$ subfield $L$ of $K(H \ff)$ is a primitive residue field of a $\Delta$-CM point on $X^D_0(N)_{/\mathbb{Q}}$ by the results of \S \ref{compiling_section1}. 

If $L_i$ is the unique primitive residue field of a $\Delta$-CM point on $X^D_0(\ell_i^{a_i})_{/\mathbb{Q}}$ for each $1 \leq i \leq r$, then $L$ is the unique primitive residue field for $X^D_0(N)_{/\mathbb{Q}}$. Otherwise, let $C_i = \ell_i^{c_i} \mid \ell_i^{a_i}$ be the smallest positive integer such that there is a $\Delta$-CM point on $X^D_0(\ell_i^{a_i})_{/\mathbb{Q}}$ with residue field isomorphic to either $K(C_i \ff)$ or an index $2$ subfield thereof for each $1 \leq i \leq r$. Setting 
\[ C = C_1 \cdots C_r, \]
we then have that $K(Cf)$ is also a primitive residue field for $X^D_0(N)_{/\mathbb{Q}}$. 

Now assume that we have two primitive residue fields, $L \subsetneq K(Hf)$ with $[K(Hf): L] = 2$ and $K(C \ff)$, of $\Delta$-CM points on $X^D_0(N)_{/\mathbb{Q}}$. Set
\[ d_1 := [L : \mathbb{Q}] \quad \text{ and } \quad d_2 := [K(Cf) : \mathbb{Q}]. \]
We note $C_i \leq H_i$ for each $1 \leq i \leq r$ by the definitions of these quantities. Further, by assumption we have at least one value of $i$ such that $K(C_i \ff)$ is a primitive residue field for $X^D_0(\ell_i^{a_i})_{/\mathbb{Q}}$, and thus
\[ [K(C_i \ff) :  \mathbb{Q}] \leq \frac{[K(H_i) : \mathbb{Q}]}{2} = [L_i: \mathbb{Q}]. \]
It follows that $d_2 \leq d_1$. Therefore, we have a unique primitive degree of a $\Delta$-CM point on $X^D_0(N)_{/\mathbb{Q}}$ if and only if $d_2 \mid d_1$, in which case $d_2$ is the unique primitive degree. The following result determines when this occurs:

\begin{theorem}\label{prim_res_flds_thm}
With the setup and notation as above, let $s$ be the number of indices $1 \leq i \leq r$ such that $K(H_i \ff)$ is a primitive residue field of a $\Delta$-CM point on $X^D_0(\ell_i^{a_i})$ (or equivalently, such that $C_i < H_i$). 

\begin{enumerate}
\item If $s=0$, then $L$ is the unique primitive residue field of a $\Delta$-CM point on $X^D_0(N)_{/\mathbb{Q}}$, and $d_1$ is the unique primitive degree. 
\item Suppose that $s \geq 1$ and that for some $1 \leq i \leq r$ with $C_i < H_i$ we are not in Case 1.5b (The dreaded case) with respect to $\Delta$ and the prime power $\ell_i^{a_i}$. We then have that $L$ and $K(Cf)$ are the two primitive residue fields of $\Delta$-CM points on $X^D_0(N)_{/\mathbb{Q}}$, while $d_2$ is the unique primitive degree. 
\item Suppose that $s \geq 1$ and that for all $1 \leq i \leq r$ with $C_i < H_i$ we are in Case 1.5b (The dreaded case) with respect to $\Delta$ and the prime power $\ell_i^{a_i}$. We then have that $L$ and $K(Cf)$ are the two primitive residue fields of $\Delta$-CM points on $X^D_0(N)_{/\mathbb{Q}}$, and that $d_1$ and $d_2$ are the two primitive degrees of such points. 
\end{enumerate}
\end{theorem}

\begin{proof}
The proof follows exactly as in \cite[Thm. 9.2]{Cl22}; the main inputs here are the degrees of our residue fields, which are the same for our totally complex index $2$ subfields of ring class fields as they are for the rational ring class fields appearing in the $D=1$ modular curve study. 
\end{proof}


\section{CM points on $X_1^D(N)_{/\mathbb{Q}}$}\label{X_1_section}

In this section, we prove Theorem \ref{inertness_thm_general}, showing that there is a very close relationship between CM points on the Shimura curves $X_0^D(N)_{/\mathbb{Q}}$ and $X_1^D(N)_{/\mathbb{Q}}$. This is a generalization of \cite[Thm. 1.2]{CS22}, which was specific to the $D=1$ case, and allows us to go from our understanding of the $\Delta$-CM locus on $X_0^D(N)_{/\mathbb{Q}}$ based on \S \ref{CM_points_level_N} to an understanding of, at the very least, degrees of CM points on $X_1^D(N)_{/\mathbb{Q}}$.

\begin{proof}[Proof of Theorem \ref{inertness_thm_general}]
We first recall some relevant facts about ramification under the natural map $\pi : X_1^D(N)_{/\mathbb{Q}} \to X_0^D(N)_{/\mathbb{Q}}$. All points on $X_1^D(N)_{/\mathbb{Q}}$ not having CM by discriminant $\Delta \in \{-3,-4\}$ are unramified over their image on $X^D(1)_{/\mathbb{Q}}$. For $N \geq 4$, just as in the $D=1$ case, the curve $X_1^D(N)$ over $\mathbb{C}$ has no elliptic points of periods $2$ or $3$, from which it follows that all $-4$ and $-3$-CM points on $X_1^D(N)_{/\mathbb{Q}}$ are ramified with ramification index $2$ or $3$, respectively. The curve $X_1^D(2)_{/\mathbb{Q}}$ has a single elliptic point of period $2$, unramified with respect to $\pi$, lying over each of the $2^b$ points on $X^D(1)_{/\mathbb{Q}}$ with $-4$-CM. The curve $X_1^D(3)_{/\mathbb{Q}}$ has a single elliptic point of period $3$, unramified with respect to $\pi$, lying over each of the $2^b$ points on $X^D(1)_{/\mathbb{Q}}$ with $-3$-CM. (One can see these claims regarding elliptic points and ramification from elementary arguments involving congruence subgroups. For example, for $D=1$ this is \cite[Exc. 2.3.7]{DS05}.)

First, suppose that $\Delta < -4$. If $N \leq 2$ then the map $\pi$ is an isomorphism, so assume $N \geq 3$ in which case it is a $(\mathbb{Z}/N\mathbb{Z})^*/\{\pm 1\}$-Galois covering, hence has degree $\phi(N)/2$. Let $\ff$ be the conductor of $\Delta$, such that $\Delta = \ff^2\Delta_K$, and consider a point $\tilde{x} \in \pi^{-1}(x)$. It suffices to show that $[K(\tilde{x}) : K(x)] = \frac{\phi(N)}{2}$, viewing $\pi$ as a morphism over $K$.

Take $\varphi : (A, \iota) \rightarrow (A',\iota')$ to be a QM-cyclic $N$-isogeny over $K(x)$ inducing $x \in X_0^D(N)_{/K}$. We know such an isogeny exists over $K(x)$ by Theorem \ref{Jordan_curves_to_surfaces}, because $K(x)$ contains $\mathbb{Q}(x)$ and splits $B$. By Theorem \ref{Shimura} we have $K(x) = K(\ff)$. We have a well defined $\pm 1$ Galois representation
\[ \overline{\rho}_{N} : \text{Gal}(\overline{K}/K(\ff)) \rightarrow \text{GL}_2(\mathbb{Z}/N\mathbb{Z})/\{\pm 1\} \]
not depending on our choice of representative for $x$, as $\textnormal{Aut}(A,\iota) = \{\pm 1\}$. 
Let $Q = \text{ker}(\varphi) \leq A[N]$ and letting $P \in Q$ be a choice of generator (of $e_1(Q)$ as an abelian group, or equivalently of $Q$ as an $\mathcal{O}$-module). The action of $\text{Gal}(\overline{K}/K(\ff))$ on $P$ is then tracked by an isogeny character
\[ \lambda: \text{Gal}(\overline{K}/K(\ff)) \rightarrow \left(\mathbb{Z}/N\mathbb{Z}\right) / \{\pm 1\}. \]
Theorem \ref{fake-ell-curves-ell-curves-correspondence} gives that $A_{\mathbb{C}} := A \otimes_{\spec K(x)} \spec \mathbb{C}$ has a decomposition $\psi: A_{\mathbb{C}} \xrightarrow{\sim} \mathbb{C}/\mathfrak{o}(\ff)  \times E_A$, where $E_A$ is a $\Delta$-CM elliptic curve over $\mathbb{C}$. The elliptic curves in this decomposition both have models over $K(\ff)$, as moreover they both have models over $\mathbb{Q}(\ff) \cong \mathbb{Q}(j_\Delta)$ where $j_{\Delta}$ is the $j$-invariant of a $\Delta$-CM elliptic curve. Hence, a $K(\ff)$-rational model for this product is a twist of $A$.

It then suffices, as our representation is independent of the choice of $K(\ff)$-rational model, to consider the case $A = E \times E'$ with $E$ and $E'$ being $\Delta$-CM elliptic curves over $K(\ff)$. Here, our QM-stable subgroup $Q \leq A[N]$ corresponds to a cyclic subgroup of $E[N]$, and $\lambda$ is induced by the Galois action on this cyclic subgroup. This $\pm 1$ character $\lambda$ is surjective by \cite[Thm. 1.4]{BC20} (in which the authors state a result of \cite{St01} in this form). Therefore, if $\{P,-P\}$ is stable over an extension $L$ of $K(\ff)$, such that $\text{Gal}(\overline{\mathbb{Q}}/L)$ is in the kernel of $\lambda$, we have 
\[ \frac{\phi(N)}{2} \mid [L : K(\ff)], \]
and so indeed we have $[K(\tilde{x}):K(x)] = \phi(N)/2$. 

We next tackle case (2)(a), assuming that $x$ is a ramified point of the map $X_0^D(N)_{/\mathbb{Q}} \rightarrow X^D(1)_{/\mathbb{Q}}$. In this case, we have that a representative  $(A,\iota,Q)_{/K(x)}$ inducing $x$, where $Q \leq A[N]$ is a QM-cyclic subgroup, is well-defined up to quadratic twist, as all models for $(A,\iota)$ are defined over $K(x)$. This is because, working geometrically for a second, a $-3$ or $-4$ CM point $x \in X_0^D(N)$ over $\mathbb{C}$ is ramified with respect to the natural map to $X^D(1)$ if and only if it is non-elliptic; it has the trivial stabilizer $\{\pm 1\}$, while its image is an elliptic point of order $3$ or $4$. The same argument as in the $\Delta < -4$ case above then applies. 

We now assume that $x$ is a $\Delta$-CM point on $X_0^D(N)_{/\mathbb{Q}}$ with $\Delta \in \{-3,-4\}$ which is unramified with respect to the map to $X^D(1)_{/\mathbb{Q}}$. If $N = 2$, then $\pi$ is an isomorphism, so the claim is trivial. If $N=3$, the fact mentioned above that there is one point lying over each elliptic point on $X^D(1)$ is exactly the inertness claim. For $N \geq 4$, we know that every point in $\pi^{-1}(x)$ is ramified with respect to the map $X_1^D(N)_{/\mathbb{Q}} \rightarrow X^D(1)_{/\mathbb{Q}}$, giving the claimed ramification index. The residue degree is therefore \emph{at most} the claimed residue degree in each case. 

To provide the lower bound on the residue degree, we modify the argument of the $\Delta < -4$ case slightly in a predictable way. If $\Delta = -4$, then a representative for $x$ is well-defined up to quartic twist. We consider a representative of the form $(E_1 \times E_2, \iota, \mathcal{O}\cdot C)$ where $E_1, E_2$ are $\mathfrak{o}(\ff) $-CM elliptic curves and $C \leq E_1[N]$ is a cyclic order $N$ subgroup (again, via the type of argument as in the $\Delta < -4$ case using Theorem \ref{fake-ell-curves-ell-curves-correspondence}). Let $q_N : \mathfrak{o}_K \rightarrow \mathfrak{o}_K/N\mathfrak{o}_K$ denote the quotient map. By tracking the action of Galois on a generator $P$ of $C$ we get a well-defined reduced mod $N$ Galois representation 
\[ \overline{\rho_N} : \text{Gal}(\mathbb{Q}(x)/\mathbb{Q}) \rightarrow \left(\mathfrak{o}_K / N\mathfrak{o}_K\right)^\times / q_N(\mathfrak{o}_K^\times) \]
which is surjective (see \cite[\S 1.3]{BC20}). As the set $\{P,-P,iP,-iP\}$ is stable under the action of $\text{Gal}(\mathbb{Q}(\tilde{x})/\mathbb{Q})$ for $\tilde{x} \in \pi^{-1}(x)$, we must have 
\[ \frac{\phi(N)}{4} = \# \left( \overline{\rho_N}\left(\text{Gal}(\mathbb{Q}(\tilde{x})/\mathbb{Q})\right)\right) \mid [\mathbb{Q}(\tilde{x}) : \mathbb{Q}(x)], \]
giving the result for $\Delta = -4$. For $\Delta = -3$, exchanging ``quartic'' for ``cubic'' and $\mu_4$ for $\mu_3$ results in the required divisibility $\frac{\phi(N)}{6} \mid f_\pi(x)$. 
\end{proof}

\section{Sporadic CM points on Shimura Curves }\label{least-degrees-section}

Fix $D > 1$ an indefinite quaternion discriminant over $\mathbb{Q}$ and $N \in \mathbb{Z}^+$ relatively prime to $D$. In analogy to prior work on degrees of CM points on certain classical families of modular curves \cite{CGPS22}, we may consider the least degree $d_{\text{CM}}(X)$ of a CM-point on a Shimura curve $X$ for the modular Shimura curves $X = X_0^D(N)_{/\mathbb{Q}}$ and $X = X_1^D(N)_{/\mathbb{Q}}$. For an imaginary quadratic order $\mathfrak{o}$, the results of \S \ref{prim_res_flds_N} allow us to compute all primitive residue fields and degrees of $\mathfrak{o}$-CM points on $X_0^D(N)_{/\mathbb{Q}}$, and hence to compute the least degree $\text{d}_{\mathfrak{o},\text{CM}}(X_0^D(N))$ of an $\mathfrak{o}$-CM point on $X_0^D(N)_{/\mathbb{Q}}$. Note that the least degree of an $\mathfrak{o}$-CM point on $X_0^D(N)_{/\mathbb{Q}}$ always satisfies 
\[ h(\mathfrak{o}) \mid d_{\mathfrak{o},\text{CM}}(X_0^D(N)). \]
Using a complete list of all imaginary quadratic orders $\mathfrak{o}$ of class number up to $100$, it then follows that if we have some order $\mathfrak{o}_0$ with 
\[ d_{\mathfrak{o}_0,\text{CM}}(X_0^D(N)) \leq 100, \]
then we can solve the minimization over orders problem to compute the least degree of a CM point on $X_0^D(N)_{/\mathbb{Q}}$:
\[ d_{\text{CM}}(X_0^D(N)) = \text{min} \left\{ d_{\mathfrak{o},\text{CM}}(X_0^D(N)) \mid h(\mathfrak{o}) \leq 100 \right\}. \]
We have implemented an algorithm to compute least degrees over specified orders and, when possible, to compute $d_{\mathfrak{o},\text{CM}}(X_0^D(N))$ exactly as described above. The relevant code, along with all other code used for the computational tasks described in this section, can be found at the repository \cite{Rep}. One may also find there a list of computed exact values of $d_{\text{CM}}(X_0^D(N))$, along with an order minimizing the degree, for all relevant pairs $(D,N)$ with $DN < 10^5$. All computations described in this section are performed using \cite{Magma}. 

Theorem \ref{inertness_thm_general} provides all of the information we need to go from least degrees of CM points on $X_0^D(N)_{/\mathbb{Q}}$ to least degrees of CM points on $X_1^D(N)_{/\mathbb{Q}}$. For ease of the relevant statement, we first generalize some terminology from \cite{CGPS22}: we will call a pair $(D,N)$ with $N \geq 4$ 
\begin{itemize}
\item \textbf{Type I} if $D$ splits $\mathbb{Q}(\sqrt{-3})$, we have $\text{ord}_3(N) \leq 1$, and $N$ is not divisible by any prime $\ell \equiv 2 \pmod{3}$, and
\item \textbf{Type II} if $D$ splits $\mathbb{Q}(\sqrt{-1})$, we have $\text{ord}_2(N) \leq 1$, and $N$ is not divisible by any prime $\ell \equiv 3 \pmod{4}$. 
\end{itemize}

\begin{proposition}\label{least_deg_X1_prop}
Let $D>1$ be a quaternion discriminant over $\mathbb{Q}$ and $N \in \mathbb{Z}^+$ coprime to $D$. 
\begin{enumerate}
\item If $(D,N)$ is Type I, then 
\[ d_{\textnormal{CM}}(X_1^D(N)) = \frac{\phi(N)}{3} .\]
\item If $(D,N)$ is not Type I and is Type II, then 
\[ d_{\textnormal{CM}}(X_1^D(N)) = \frac{\phi(N)}{2}. \]
\item If $(D,N)$ is not Type I or Type II, then
\[ d_{\textnormal{CM}}(X_1^D(N)) = \frac{\phi(N)}{2} \cdot d_{\textnormal{CM}}(X_0^D(N)). \]
\end{enumerate}
\end{proposition}
\begin{proof}
The natural map $X_1^D(N)_{/\mathbb{Q}} \rightarrow X_0^D(N)_{/\mathbb{Q}}$ has non-trivial ramification exactly when $(D,N)$ is either Type I or Type II. In these cases, we have $d_{\text{CM}}(X_0^D(N)) = 2$, which is as small as possible as the $D>1$ assumption implies these curves have no rational points. The statements then follow immediately from the residue degrees with respect to this map provided by Theorem \ref{inertness_thm_general}. 
\end{proof}

For a curve $X_{/\mathbb{Q}}$, let $\delta(X)$ denote the least positive integer $d$ such that $X$ has infinitely points of degree $d$. We call a point $x \in X$ \textbf{sporadic} if 
\[ \text{deg}(x) := [\mathbb{Q}(x) : \mathbb{Q}] < \delta(X). \]
That is, $x$ is a sporadic point if there are only finitely points $y \in X$ with $\text{deg}(y) \leq \text{deg}(x)$. Sporadic points on modular curves have been objects of interest in several recent works, including \cite{Naj16, BELOV19, BN21, CGPS22, Smi23, BGRW24}. 

In the remainder of this section, we apply our least degree computations towards the question of whether the curves $X_0^D(N)_{/\mathbb{Q}}$ and $X_1^D(N)_{/\mathbb{Q}}$ have sporadic CM points.

\subsection{An explicit upper bound on $d_{\text{CM}}(X_0^D(N))$}

In analogy to the Heegner hypothesis of the modular curve case, we make the following definition:

\begin{definition}
Let $D$ be an indefinite quaternion discriminant and $N$ a positive integer relatively prime to $D$. We will say that an imaginary quadratic discriminant $\Delta$ \textbf{satisfies the $(D,N)$ Heegner hypothesis} if 
\begin{enumerate}
    \item for all primes $\ell \mid D$, we have $\legendre{\Delta}{\ell} = -1$, and
    \item for all primes $\ell \mid N$, we have $\legendre{\Delta}{\ell} = 1$,
\end{enumerate}
\end{definition}

\noindent If $\Delta$ satisfies the $(D,N)$ Heegner hypothesis, this implies the existence of a $\Delta$-CM point on $X_0^D(N)_{/\mathbb{Q}}$ which is rational over $K(\ff)$, the ring class field of conductor $\ff$ where $\Delta = \ff^2\Delta_K$. This point therefore has degree at most $2\cdot h(\mathfrak{o}(\ff) )$.  

We provide an upper bound on the least degree of a CM point on $X_0^D(N)_{/\mathbb{Q}}$ as follows: let $L$ be the least positive integer such that
\begin{itemize}
    \item $\legendre{L}{p} = -1$ for all odd primes $p \mid D$,
    \item $\legendre{L}{p} = 1$ for all odd primes $p \mid N$, and
    \item we have \[L \equiv \begin{cases} 5 \pmod{8} \text{ if } 2 \mid D \\
    1 \pmod{8} \text{ otherwise}. \end{cases}\]
\end{itemize}
Then $0 < L < 8DN$, and so $d_0 = L - 16DN$ is an imaginary quadratic discriminant satisfying the $(D,N)$ Heegner hypothesis with $-16DN < d_0 < -8DN$. It follows that there exists a fundamental discriminant $\Delta_K$ of an imaginary quadratic field $K$ satisfying the $(D,N)$ Heegner hypothesis with $|\Delta_K| < 16DN$; take $K$ such that $d_0$ corresponds to an order in $K$ and hence $d_0 = \ff^2\Delta_K$ for some positive integer $\ff$. 

For an imaginary quadratic field $K$ of discriminant $\Delta_K < -4$, we have 
\[ h_K = h(\mathfrak{o}(\Delta_K)) \leq \frac{e}{2\pi} \sqrt{|d|} \log{|d|} \]
(see, e.g., \cite[Appendix]{CCS13}), such that the above provides
\begin{equation} \label{degree_bound} d_{\text{CM}}(X_0^D(N)) \leq 2\cdot  h_K \leq \frac{4e}{\pi} \sqrt{DN} \log{(16DN)}. \end{equation}

\subsection{Shimura curves with infinitely many points of degree 2}\label{inf_deg_2_section}

If $\delta(X_0^D(N)) = 2$, then as $X_0^D(N)_{/\mathbb{Q}}$ has no real points it certainly does not have a sporadic point. We mention here all pairs $(D,N)$ for which we know $\delta(X_0^D(N)) = 2$ based on the existing literature. 

All genus $0$ and $1$ cases necessarily have $\delta(X_0^D(N)) = 2$, as we have no degree $1$ points. Voight \cite{Voi09} lists all $(D,N)$ for which $X_0^D(N)_{/\mathbb{Q}}$ has genus zero:
\[ \{(6,1), (10,1), (22,1)  \}, \]
and genus one:
\[ \{(6,5), (6,7), (6,13), (10,3), (10,7), (14,1), (15,1), (21,1), (33,1), (34,1), (46,1)   \} .\]
By a result of Abramovich--Harris \cite{AH91},  a nice curve $X$ defined over $\mathbb{Q}$ of genus at least $2$ with $\delta(X) = 2$ is either hyperelliptic over $\mathbb{Q}$, or is bielliptic and emits a degree $2$ map to an elliptic curve over $\mathbb{Q}$ with positive rank. The pairs $(D,N)$ for which $X_0^D(N)_{/\mathbb{Q}}$ is hyperelliptic of genus at least $2$ were determined by Ogg\footnote{Actually, for the pairs $(10,19)$ and $(14,5)$, the referenced work of Ogg says that the corresponding curves are hyperelliptic over $\mathbb{R}$. Ogg does not say whether that is the case over $\mathbb{Q}$, but work of Guo-Yang \cite{GY17} answers negatively for the former pair and positively for the latter.} \cite{Ogg83}:
\begin{align*} 
\{ &(6,11), (6,19), (6,29), (6,31), (6,37), (10,11),(10,23), (14,5), (15,2), \\
& (22,3), (22,5), (26,1), (35,1), (38,1), (39,1), (39,2) , (51,1), (55,1), (58,1), (62,1), \\
& (69,1), (74,1), (86,1), (87,1), (94,1), (95,1), (111,1), (119,1), (134,1), \\
& (146,1), (159,1), (194,1), (206,1) \}. 
\end{align*}
As for the bielliptic case, Rotger \cite{Rot02} has determined all discriminants $D$ such that $X^D(1) = X_0^D(1)$ is bielliptic, and further determines those for which $X^D(1)_{/\mathbb{Q}}$ is bielliptc over $\mathbb{Q}$ and maps to a positive rank elliptic curve. All such discriminants $D$ with $g(X_0^D(1)) \geq 2$ and with $X_0^D(1)_{/\mathbb{Q}}$ not hyperelliptic are as follows:
\begin{align*}
 D \in \{57,65,77,82,106,118 ,122,129,143,166,210,215,314, 330,390,510,546\}. 
\end{align*}

\subsection{Sporadic CM points}

In order to declare the existence of a sporadic CM point on a Shimura curve $X_0^D(N)_{/\mathbb{Q}}$, a main tool for us will be the following result of Frey \cite[Prop. 2]{Frey94} on the least degree $\delta(X)$ over which a nice curve $X_{/F}$ has infinitely many closed points:

\begin{theorem}[Frey 1994] \label{Frey}
For a nice curve $X$ defined over a number field $F$, we have
\[ \frac{\gamma_F(X)}{2} \leq \delta(X) \leq \gamma_F(X), \]
where $\gamma_F(X)$ denotes the $F$-gonality of $X$, i.e., is the least degree of a non-constant $F$-rational map to the projective line. 
\end{theorem}

It follows from Theorem \ref{Frey} that if 
\begin{equation} \label{sporadic} d_{\text{CM}}(X_0^D(N)) < \frac{\gamma_\mathbb{Q}(X_0^D(N))}{2}, 
\end{equation}
then there exists a sporadic CM point on $X_0^D(N)_{/\mathbb{Q}}$. To complement this, a result of Abramovich provides a lower bound on the gonality of a Shimura curve. Our cases of interest in applying this result are $X_0^D(N) = X_{\Gamma^D_0(N)}$ and $X_1^D(N) = X_{\Gamma^D_1(N)}$ (or, equivalently, $X_0^D(N) = X_{\mathcal{O}_N^1}$, where $\mathcal{O}_N$ is an Eichler order of level $N$ in $B$, for the former curve). 

\begin{theorem}[Abramovich 1996] \label{Abr}
Let $X_\Gamma$ be the Shimura curve corresponding to $\Gamma \leq \mathcal{O}^1$ a subgroup of the units of norm $1$ in an order $\mathcal{O}$ of $B$. Then
\[ \frac{975}{8192}(g(X_\Gamma)-1) \leq \gamma_\mathbb{C}(X_\Gamma) \leq \gamma_\mathbb{Q}(X_\Gamma). \]
\end{theorem}
\begin{proof}
This is a version of \cite[Thm. 1.1]{Abr96}, where the constant has been improved using the best known progress due to Kim--Sarnak on Selberg's eigenvalue conjecture \cite[p. 176]{KS03}. 
\end{proof}

The following result will allow us to transfer information about the existence of sporadic points on $X_0^D(N)_{/\mathbb{Q}}$ to those on $X_1^D(N)_{/\mathbb{Q}}$:

\begin{proposition}\label{sporadic_X_1_prop}
Let $\pi : X_1^D(N)_{/\mathbb{Q}} \rightarrow X_0^D(N)_{/\mathbb{Q}}$ denote the natural modular map. Suppose that $P_0 \in X^D_0(N)_{/\mathbb{Q}}$ satisfies
\[ \textnormal{deg}(P_0) \leq \frac{975}{16384} \left(g(X^D_0(N)) - 1\right). \]
Then any $P \in X^D_1(N)_{/\mathbb{Q}}$ with $\pi(P) = P_0$ is sporadic. 
\end{proposition}
\begin{proof}
For such a point $P \in X_1^D(1)_{/\mathbb{Q}}$, using the notation and results of Proposition \ref{genus_formula} we have
\begin{align*} 
\text{deg}(P) &\leq \text{deg}(P_0) \cdot \deg(\pi) \\
&= \text{deg}(P_0) \cdot \frac{\phi(N)}{2} \\
&\leq \dfrac{975}{16384} \left( \frac{\phi(D)\psi(N)}{12} - \dfrac{\epsilon_1(D,N)}{4} - \dfrac{\epsilon_3(D,N)}{3} \right) \cdot \frac{\phi(N)}{2}  \\
&\leq \dfrac{975}{16384} \left( \frac{\phi(N)\phi(D)\psi(N)}{24}  \right) \\
&= \frac{975}{16384}(g(X^D_1(N) - 1). 
\end{align*}
It then follows from Theorem \ref{Abr} that $P$ is sporadic. 
\end{proof}

We now obtain a lower bound on the genus of $X_0^D(N)$ that will be amenable to our arguments:

\begin{lemma}
For $D> 1$ an indefinite quaternion discriminant over $\mathbb{Q}$ and $N \in \mathbb{Z}^+$ relatively prime to $D$, we have 

\begin{align*} g(X_0^D(N)) - 1 &>\dfrac{DN}{12} \left( \dfrac{1}{e^\gamma \log\log{D} + \frac{3}{\log\log(D)}} \right) - \frac{7\sqrt{DN}}{3} \\
&\geq \dfrac{DN}{12} \left( \dfrac{1}{e^\gamma \log\log{(DN)} + \frac{3}{\log\log(6)}} \right) - \frac{7\sqrt{DN}}{3}. 
\end{align*}
\end{lemma}

\begin{proof}
We make use of the trivial bound $\psi(N) \geq N$, and the lower bound
\[ \phi(D) > \dfrac{D}{e^\gamma \log\log{D} + \frac{3}{\log\log{D}}}. \]
For $M \in \mathbb{Z}^+$, let $\omega(M)$ and $d(M)$ denote, respectively, the number of distinct prime divisors of $M$ and the number of divisors of $M$. We then have
\[ \epsilon_1(D,N), \epsilon_3(D,N) \leq 2^{\omega(DN)} \leq d(DN) \leq d(D) \cdot d(N) \leq 4\sqrt{DN}. \]
Using these bounds along with the fact that $D \geq 6$ and $N \geq 1$, we arrive at the stated inequalities from the genus fromula given in Proposition \ref{genus_formula}. 
\end{proof}

\noindent The combination of this lemma with (\ref{degree_bound}) and (\ref{sporadic}) guarantees a sporadic CM point on $X_0^D(N)_{/\mathbb{Q}}$ if
\[ \frac{4e}{\pi}\sqrt{DN}\log{(16DN)} \leq \frac{325DN}{65536}\left(\dfrac{1}{e^\gamma \log\log{(DN)} + \frac{3}{\log\log(6)}} \right) - \frac{2275\sqrt{DN}}{16384}. \]
This inequality holds for all pairs $(D,N)$ with $DN \geq 4.27512 \cdot 10^{10}$.

Ranging through pairs $(D,N)$ with $DN$ below this bound, we attempt to determine the fundamental imaginary quadratic discriminant $\Delta_K$ of smallest absolute value satisfying the $(D,N)$-Heegner hypothesis. If found, we check whether we have a $\Delta_K$-CM point of degree at most half $\gamma_\Q(X_0^D(N))$ via the inequality
\begin{equation}\label{least_Heegner_check}
h_K < \frac{325\phi(D)\psi(N)}{65536}- \frac{2275\sqrt{DN}}{16384}. 
\end{equation}
We confirm that (\ref{least_Heegner_check}) holds, and thus a sporadic CM point on $X_0^D(N)_{/\mathbb{Q}}$ is ensured, for all pairs $(D,N)$ with $DN > 14982$ aside from the $20$ pairs comprising the following set $\mathcal{F}_1$:
\begin{align*}
\mathcal{F}_1 = \{&(101959, 210), (111397, 210), (141427, 210), (154583, 210), (164749, 210), \\
& (165053, 330), (174629, 330), (190619, 210), (192907, 210), (194051, 210), \\
& (199801, 330), (208351, 210), (218569, 210), (233519, 210), (240097, 210), \\
& (272459, 210), (287419, 210), (296153, 210), (304513, 210), (307241, 210) \} . 
\end{align*}
For each pair $(D,N) \in \mathcal{F}_1$, it is not that the inequality (\ref{least_Heegner_check}) does not hold. Rather, there is no imaginary quadratic discriminant of class number at most $100$ satisfying the $(D,N)$-Heegner hypothesis, such that we fail to perform the check using only such discriminants. For each of these pairs, we compute $d_{\text{CM}}(X_0^D(N))$ exactly and find that for each the inequality 
\[ d_{\text{CM}}(X_0^D(N)) < \frac{325\phi(D)\psi(N)}{32768}- \frac{2275\sqrt{DN}}{8192}\]
holds. By the preceding remarks, this confirms that the curve $X_0^D(N)_{/\mathbb{Q}}$ has a sporadic CM point for all $(D,N) \in \mathcal{F}_1$.  

There are exactly $4392$ pairs $(D,N)$, each with $DN \leq 14982,$ for which the inequality (\ref{least_Heegner_check}) does not hold. These are listed in the file \texttt{bads{\_}list.m} in \cite{Rep}. For each of these, we perform an exact computation of $d_{\text{CM}}(X_0^D(N))$. By the above, a sporadic CM point on $X_0^D(N)_{/\mathbb{Q}}$ is guaranteed if
\begin{equation} \label{sporadic_dcm_comp}
d_{\text{CM}}(X_0^D(N)) < \frac{975}{16384} \left( \frac{\phi(D)\psi(N)}{12} - \frac{e_1(D,N)}{4} - \frac{e_3(D,N)}{3} \right). 
\end{equation}

\begin{lemma}\label{sp_dcm_comp_list}
There are exactly $574$ pairs $(D,N)$ consisting of a quaternion discriminant $D > 1$ over $\mathbb{Q}$ and a positive integer $N$ coprime to $D$ such that the inequality \textnormal{(\ref{sporadic_dcm_comp})} does not hold. For all such pairs we have $d_{\textnormal{CM}}(X_0^D(1)) \in \{2,4,6\}$, and the largest value of $D$ occuring among such pairs is $D = 1590$. 
\end{lemma}

\begin{proof}
This follows from direct computation. The $574$ referenced pairs are listed in the file \texttt{fail{\_}dcm{\_}check.m} in \cite{Rep}. 
\end{proof}

\begin{lemma}\label{additional_knowns_lemma}
Set 
\begin{align*} 
\mathcal{D} := \{ &85   , 91   , 93   , 115   , 123   , 133   , 141   , 142   , 145   , 155   , 158   , 161   , 177   , 178   , 183   , 185   , 187   , 201   , 202   , \\
& 203   , 205   , 209   , 213   , 214   , 217   , 218   , 219   , 221   , 226   , 235   , 237   , 247   , 249   , 253   , 254   , 259   , 262   , 265   , \\
&267   , 274   , 278   , 287   , 291   , 295   , 298   , 299   , 301   , 302   , 303   , 305   , 309   , 319   , 321   , 323   , 326   , 327   , 329   , \\
&334   , 335   , 339   , 341   , 346   , 355   , 358   , 362   , 365   , 371   , 377   , 381   , 382   , 386   , 391   , 393   , 394   , 395   , 398   , \\
&403   , 407   , 411   , 413   , 415   , 417   , 422   , 427   , 437   , 445   , 446   , 447   , 451   , 453   , 454   , 458   , 462   , 466   , 469   , \\
&471   , 478   , 482   , 485   , 489   , 501   , 502   , 505   , 514   , 519   , 526   , 537   , 538   , 542   , 543   , 554   , 562   , 566   , 570   , \\
&573   , 579   , 586   , 591   , 597   , 614   , 622   , 626   , 634   , 662   , 674   , 690   , 694   , 698   , 706   , 714   , 718   , 734   , 746   , \\
&758   , 766   , 770   , 778   , 794   , 798   , 802   , 838   , 858   , 870   , 910   , 930   , 966   , 1110   , 1122   , 1190   , 1218   , 1230   , \\
&1254   , 1290   , 1302   , 1326   , 1410   , 1518   , 1590\}. 
\end{align*}
and 
\begin{align*}
\mathcal{E} := \{ &(85  ,2  ) , (85  ,3  ) , (85  ,4  ) , (91  ,2  ) , (91  ,3  ) , (93  ,2  ) , (93  ,4  ) , (93  ,5  ) , (115  ,2  ) , (115  ,3  ) , (123  ,2  ) , \\
&(133  ,2  ) , (141  ,2  ) , (142  ,3  ) , (145  ,2  ) , (155  ,2  ) , (158  ,3  ) , (161  ,2  ) , (177  ,2  ) , (178  ,3  ) , (183  ,2  ) ,\\
& (201  ,2  ) , (202  ,3  )\}.
\end{align*} 
For each of the 181 pairs $(D,N)$ with either $D \in \mathcal{D}$ and $N = 1$ or with $(D,N) \in \mathcal{E}$, the curve $X_0^D(N)_{/\mathbb{Q}}$ has a sporadic CM point. 
\end{lemma}
\begin{proof}
For each such pair $(D,N)$, we know from \S \ref{inf_deg_2_section} that $X_0^{D}(1)_{/\mathbb{Q}}$ does not have infinitely many degree $2$ points and hence $X_0^{D}(N)_{/\mathbb{Q}}$ does not have infinitely many degree $2$ points. At the same time, we compute that this curve has a CM point of degree $2$, which is therefore necessarily sporadic. 
\end{proof}
We are now prepared to end with the main result of this section:

\begin{theorem}\label{sp_CM_pts_thm}
\begin{enumerate}
\item For each of the $64$ pairs $(D,N)$ in Table \ref{table:no_sporadic_pts}, the Shimura curve $X_0^D(N)_{/\mathbb{Q}}$ has no sporadic points. For each of these pairs, we have $d_{\textnormal{CM}}(X_0^D(N)) = 2$. 
\item For each of the $64$ pairs $(D,N)$ in Table \ref{table:no_sporadic_pts} except for possibly the $10$ in the following set:
\[ \{ (6,5), (6, 7 ), ( 6, 13 ), ( 6, 19 ), (6,29), ( 6, 31 ), ( 6, 37 ), ( 10, 7 ), (14,5), (22,5)\}, \]
the Shimura curve $X_1^D(N)_{/\mathbb{Q}}$ has no sporadic CM points.
\item There are at most $329$ pairs $(D,N)$, consisting of an indefinite quaternion discriminant $D>1$ over $\mathbb{Q}$ and a positive integer $N$ coprime to $D$, which do not appear among the $64$ listed in Table \ref{table:no_sporadic_pts} and for which the Shimura curve $X_0^D(N)_{/\Q}$ does not have a sporadic CM point. These are listed in Table \ref{table:unknowns_table}. 
\item Let $(D,N)$ be a pair consisting of an indefinite quaternion discriminant $D>1$ over $\mathbb{Q}$ and a positive integer $N$ coprime to $D$. If $(D,N)$ is not listed in Table \ref{table:no_sporadic_pts} or Table \ref{table:unknowns_table} and is not equal to $(91,5)$, then the Shimura curve $X_1^D(N)_{/\mathbb{Q}}$ has a sporadic CM point. 
\end{enumerate}
\end{theorem}

\begin{proof}
\begin{enumerate}
\item These Shimura curves $X_0^D(N)_{/\mathbb{Q}}$ are exactly those for which we know that $\delta(X_0^D(N)) = 2$ via \S \ref{inf_deg_2_section}. That each such curve has a CM point of degree $2$ follows from direct computation. 
\item For each pair in this table, we have
\[ \delta(X_1^D(N)) \leq 2 \cdot \text{deg}(X_1^D(N) \rightarrow X_0^D(N)) = \text{max}\{2, \phi(N) \}. \]
For each pair in this table other than the $10$ listed pairs, we compute that
\[ \text{max}\{2,\phi(N)\} \leq d_{\text{CM}}(X_1^D(N)). \]
\item This is an immediate consequence of the preceding discussion, including Lemmas \ref{sp_dcm_comp_list} and \ref{additional_knowns_lemma}. 
\item By Proposition \ref{sporadic_X_1_prop}, we have that $X_1^D(N)_{/\mathbb{Q}}$ has a sporadic CM point for all pairs $(D,N)$ aside from possibly the $574$ referred to in Lemma \ref{sp_dcm_comp_list}. Of the $181$ pairs listed in Lemma \ref{additional_knowns_lemma}, we compute that each pair except for $(D,N) = (91,5)$ satisfies
\[ d_{\text{CM}}(X_1^D(N)) = 2 < \delta(X_0^D(1)) \leq \delta(X_1^D(N)), \]
and hence we have a sporadic CM point on $X_1^D(N)_{/\mathbb{Q}}$ for all such pairs. The result then follows from part (2). 
\end{enumerate}
\end{proof}

\begin{remark}
For all of the $329$ pairs $(D,N)$ listed in Table \ref{table:unknowns_table}, we have $d_{\text{CM}}(X_0^D(N)) \in \{2,4,6\}$. For all but $56$ of these pairs, we have $d_{\textnormal{CM}}(X_0^D(N)) = 2$. For such pairs, it follows that the curve $X_0^D(N)_{/\mathbb{Q}}$ has a sporadic (CM) point if and only if it is not bielliptic with a degree $2$ map to an elliptic curve over $\mathbb{Q}$ of positive rank. An extension of the results of \cite{Rot02} mentioned in \S \ref{inf_deg_2_section} to general level $N$ would then allow us to determine whether $X_0^D(N)_{/\mathbb{Q}}$ has a sporadic CM point for all but at most $56$ pairs $(D,N)$ with $D>1$. Such an extension will appear in work of the author and Oana Padurariu \cite{PS24}. 
\end{remark}

\begin{table}[h!]
\centering
\renewcommand{\arraystretch}{1.5}
\begin{tabular}{ |c|c|c|c|c|c| } 
 \hline

$ ( 6, 1) $ & $ ( 6, 5 ) $ & $ ( 6, 7 ) $ & $ ( 6, 11 ) $ & $ ( 6, 13 ) $ & $ ( 6, 19 ) $ \\ \hline

$ (6,29) $ & $ ( 6, 31) $ & $ ( 6, 37 ) $ & $ ( 10, 1 ) $ & $ ( 10, 3 ) $ & $ ( 10, 7 ) $  \\ \hline

$ ( 10, 11 ) $ & $ (10,23) $ & $(14,1)$ & $(14,5)$ & $ ( 15, 1) $ & $(15,2)$  \\ \hline

$ ( 21, 1 ) $ & $ ( 22, 1 ) $ & $(22,3)$ & $ ( 22, 5) $ & $ ( 26, 1 ) $ & $ ( 33, 1 ) $ \\ \hline

$ ( 34, 1 ) $ & $ ( 35, 1 ) $ & $ ( 38, 1 ) $ & $ (39,1) $ &  $ ( 39, 2) $ & $ ( 46, 1 ) $ \\ \hline
 
$ ( 51, 1 ) $ & $ ( 55, 1 ) $ & $ ( 57, 1 ) $ & $ ( 58, 1 ) $ & $ (62,1) $ & $ ( 65, 1) $  \\ \hline

$ ( 69, 1 ) $ & $ ( 74, 1 ) $ & $ ( 77, 1 ) $ & $ ( 82, 1 ) $ & $ ( 86, 1 ) $ & $ (87,1) $ \\ \hline
 
$ ( 94, 1) $ & $ ( 95, 1 ) $ & $ ( 106, 1 ) $ & $ ( 111, 1 ) $ & $ ( 118, 1 ) $ & $(119,1)$  \\ \hline
 
$ ( 122, 1 ) $ & $ (129,1) $ & $ ( 134, 1) $ & $ ( 143, 1 ) $ & $ (146,1)$ & $(159,1)$  \\ \hline
 
$( 166, 1 ) $ & $(194,1)$  & $(206,1)$ & $ ( 210, 1 ) $ & $ ( 215, 1 ) $ & $ ( 314, 1 ) $  \\  \hline
 
$ (330,1) $ & $ ( 390, 1) $ & $ ( 510, 1 ) $ & $(546,1)$ & &  \\  \hline

\end{tabular}
\caption{$64$ pairs $(D,N)$ with $\text{gcd}(D,N) = 1$ for which $\delta(X_0^D(N)) = 2$, and hence $X_0^D(N)_{/\mathbb{Q}}$ has no sporadic points \vspace{0.5em}}
\label{table:no_sporadic_pts}
\end{table}

\renewcommand{\arraystretch}{1.5}
\begin{longtable}{ |c|c|c|c|c|c|c|c|c| }
 \hline
 
$ ( 6 , 17 ) $ & 
$ ( 6 , 23 ) $ & 
$ ( 6 , 25 ) $ & 
$ ( 6 , 35 ) $ & 
$ ( 6 , 41 ) $ & 
$ ( 6 , 43 ) $ & 
$ ( 6 , 47 ) $ & 
$ ( 6 , 49 ) $ & 
$ ( 6 , 53 ) $ \\ \hline
$ ( 6 , 55 ) $ & 
$ ( 6 , 59 ) $ & 
$ ( 6 , 61 ) $ & 
$ ( 6 , 65 ) $ & 
$ ( 6 , 67 ) $ & 
$ ( 6 , 71 ) $ & 
$ ( 6 , 73 ) $ & 
$ ( 6 , 77 ) $ & 
$ ( 6 , 79 ) $ \\ \hline
$ ( 6 , 83 ) $ & 
$ ( 6 , 85 ) $ & 
$ ( 6 , 89 ) $ & 
$ ( 6 , 91 ) $ & 
$ ( 6 , 95 ) $ & 
$ ( 6 , 97 ) $ & 
$ ( 6 , 101 ) $ & 
$ ( 6 , 103 ) $ & 
$ ( 6 , 107 ) $ \\ \hline
$ ( 6 , 109 ) $ & 
$ ( 6 , 113 ) $ & 
$ ( 6 , 115 ) $ & 
$ ( 6 , 119 ) $ & 
$ ( 6 , 121 ) $ & 
$ ( 6 , 125 ) $ & 
$ ( 6 , 127 ) $ & 
$ ( 6 , 131 ) $ & 
$ ( 6 , 133 ) $ \\ \hline
$ ( 6 , 137 ) $ & 
$ ( 6 , 139 ) $ & 
$ ( 6 , 143 ) $ & 
$ ( 6 , 145 ) $ & 
$ ( 6 , 149 ) $ & 
$ ( 6 , 151 ) $ & 
$ ( 6 , 155 ) $ & 
$ ( 6 , 157 ) $ & 
$ ( 6 , 161 ) $ \\ \hline
$ ( 6 , 163 ) $ & 
$ ( 6 , 167 ) $ & 
$ ( 6 , 169 ) $ & 
$ ( 6 , 173 ) $ & 
$ ( 6 , 179 ) $ & 
$ ( 6 , 181 ) $ & 
$ ( 6 , 191 ) $ & 
$ ( 6 , 193 ) $ & 
$ ( 6 , 197 ) $ \\ \hline
$ ( 6 , 199 ) $ & 
$ ( 6 , 203 ) $ & 
$ ( 6 , 287 ) $ & 
$ ( 6 , 295 ) $ & 
$ ( 6 , 319 ) $ & 
$ ( 10 , 9 ) $ & 
$ ( 10 , 13 ) $ & 
$ ( 10 , 17 ) $ & 
$ ( 10 , 19 ) $ \\ \hline
$ ( 10 , 21 ) $ & 
$ ( 10 , 27 ) $ & 
$ ( 10 , 29 ) $ & 
$ ( 10 , 31 ) $ & 
$ ( 10 , 33 ) $ & 
$ ( 10 , 37 ) $ & 
$ ( 10 , 39 ) $ & 
$ ( 10 , 41 ) $ & 
$ ( 10 , 43 ) $ \\ \hline
$ ( 10 , 47 ) $ & 
$ ( 10 , 49 ) $ & 
$ ( 10 , 51 ) $ & 
$ ( 10 , 53 ) $ & 
$ ( 10 , 57 ) $ & 
$ ( 10 , 59 ) $ & 
$ ( 10 , 61 ) $ & 
$ ( 10 , 63 ) $ & 
$ ( 10 , 67 ) $ \\ \hline
$ ( 10 , 69 ) $ & 
$ ( 10 , 71 ) $ & 
$ ( 10 , 73 ) $ & 
$ ( 10 , 77 ) $ & 
$ ( 10 , 79 ) $ & 
$ ( 10 , 83 ) $ & 
$ ( 10 , 87 ) $ & 
$ ( 10 , 89 ) $ & 
$ ( 10 , 91 ) $ \\ \hline
$ ( 10 , 97 ) $ & 
$ ( 10 , 103 ) $ & 
$ ( 10 , 119 ) $ & 
$ ( 10 , 141 ) $ & 
$ ( 10 , 161 ) $ & 
$ ( 10 , 191 ) $ & 
$ ( 14 , 3 ) $ & 
$ ( 14 , 9 ) $ & 
$ ( 14 , 11 ) $ \\ \hline
$ ( 14 , 13 ) $ & 
$ ( 14 , 15 ) $ & 
$ ( 14 , 17 ) $ & 
$ ( 14 , 19 ) $ & 
$ ( 14 , 23 ) $ & 
$ ( 14 , 25 ) $ & 
$ ( 14 , 27 ) $ & 
$ ( 14 , 29 ) $ & 
$ ( 14 , 31 ) $ \\ \hline
$ ( 14 , 33 ) $ & 
$ ( 14 , 37 ) $ & 
$ ( 14 , 39 ) $ & 
$ ( 14 , 41 ) $ & 
$ ( 14 , 43 ) $ & 
$ ( 14 , 47 ) $ & 
$ ( 14 , 53 ) $ & 
$ ( 14 , 59 ) $ & 
$ ( 14 , 61 ) $ \\ \hline
$ ( 14 , 87 ) $ & 
$ ( 14 , 95 ) $ & 
$ ( 15 , 4 ) $ & 
$ ( 15 , 7 ) $ & 
$ ( 15 , 8 ) $ & 
$ ( 15 , 11 ) $ & 
$ ( 15 , 13 ) $ & 
$ ( 15 , 14 ) $ & 
$ ( 15 , 16 ) $ \\ \hline
$ ( 15 , 17 ) $ & 
$ ( 15 , 19 ) $ & 
$ ( 15 , 22 ) $ & 
$ ( 15 , 23 ) $ & 
$ ( 15 , 26 ) $ & 
$ ( 15 , 28 ) $ & 
$ ( 15 , 29 ) $ & 
$ ( 15 , 31 ) $ & 
$ ( 15 , 32 ) $ \\ \hline
$ ( 15 , 34 ) $ & 
$ ( 15 , 37 ) $ & 
$ ( 15 , 41 ) $ & 
$ ( 15 , 43 ) $ & 
$ ( 15 , 47 ) $ & 
$ ( 21 , 2 ) $ & 
$ ( 21 , 4 ) $ & 
$ ( 21 , 5 ) $ & 
$ ( 21 , 8 ) $ \\ \hline
$ ( 21 , 10 ) $ & 
$ ( 21 , 11 ) $ & 
$ ( 21 , 13 ) $ & 
$ ( 21 , 16 ) $ & 
$ ( 21 , 17 ) $ & 
$ ( 21 , 19 ) $ & 
$ ( 21 , 23 ) $ & 
$ ( 21 , 25 ) $ & 
$ ( 21 , 29 ) $ \\ \hline
$ ( 21 , 31 ) $ & 
$ ( 21 , 38 ) $ & 
$ ( 22 , 7 ) $ & 
$ ( 22 , 9 ) $ & 
$ ( 22 , 13 ) $ & 
$ ( 22 , 15 ) $ & 
$ ( 22 , 17 ) $ & 
$ ( 22 , 19 ) $ & 
$ ( 22 , 21 ) $ \\ \hline
$ ( 22 , 23 ) $ & 
$ ( 22 , 25 ) $ & 
$ ( 22 , 27 ) $ & 
$ ( 22 , 29 ) $ & 
$ ( 22 , 31 ) $ & 
$ ( 22 , 35 ) $ & 
$ ( 22 , 37 ) $ & 
$ ( 22 , 51 ) $ & 
$ ( 26 , 3 ) $ \\ \hline
$ ( 26 , 5 ) $ & 
$ ( 26 , 7 ) $ & 
$ ( 26 , 9 ) $ & 
$ ( 26 , 11 ) $ & 
$ ( 26 , 15 ) $ & 
$ ( 26 , 17 ) $ & 
$ ( 26 , 19 ) $ & 
$ ( 26 , 21 ) $ & 
$ ( 26 , 23 ) $ \\ \hline
$ ( 26 , 25 ) $ & 
$ ( 26 , 29 ) $ & 
$ ( 26 , 31 ) $ & 
$ ( 33 , 2 ) $ & 
$ ( 33 , 4 ) $ & 
$ ( 33 , 5 ) $ & 
$ ( 33 , 7 ) $ & 
$ ( 33 , 8 ) $ & 
$ ( 33 , 10 ) $ \\ \hline
$ ( 33 , 13 ) $ & 
$ ( 33 , 16 ) $ & 
$ ( 33 , 17 ) $ & 
$ ( 33 , 19 ) $ & 
$ ( 34 , 3 ) $ & 
$ ( 34 , 5 ) $ & 
$ ( 34 , 7 ) $ & 
$ ( 34 , 9 ) $ & 
$ ( 34 , 11 ) $ \\ \hline
$ ( 34 , 13 ) $ & 
$ ( 34 , 15 ) $ & 
$ ( 34 , 19 ) $ & 
$ ( 34 , 23 ) $ & 
$ ( 34 , 29 ) $ & 
$ ( 34 , 35 ) $ & 
$ ( 35 , 2 ) $ & 
$ ( 35 , 3 ) $ & 
$ ( 35 , 4 ) $ \\ \hline
$ ( 35 , 6 ) $ & 
$ ( 35 , 8 ) $ & 
$ ( 35 , 9 ) $ & 
$ ( 35 , 11 ) $ & 
$ ( 35 , 12 ) $ & 
$ ( 35 , 13 ) $ & 
$ ( 38 , 3 ) $ & 
$ ( 38 , 5 ) $ & 
$ ( 38 , 7 ) $ \\ \hline
$ ( 38 , 9 ) $ & 
$ ( 38 , 11 ) $ & 
$ ( 38 , 13 ) $ & 
$ ( 38 , 17 ) $ & 
$ ( 38 , 21 ) $ & 
$ ( 39 , 4 ) $ & 
$ ( 39 , 5 ) $ & 
$ ( 39 , 7 ) $ & 
$ ( 39 , 8 ) $ \\ \hline
$ ( 39 , 10 ) $ & 
$ ( 39 , 11 ) $ & 
$ ( 39 , 31 ) $ & 
$ ( 46 , 3 ) $ & 
$ ( 46 , 5 ) $ & 
$ ( 46 , 7 ) $ & 
$ ( 46 , 9 ) $ & 
$ ( 46 , 11 ) $ & 
$ ( 46 , 13 ) $ \\ \hline
$ ( 46 , 15 ) $ & 
$ ( 46 , 17 ) $ & 
$ ( 51 , 2 ) $ & 
$ ( 51 , 4 ) $ & 
$ ( 51 , 5 ) $ & 
$ ( 51 , 7 ) $ & 
$ ( 51 , 8 ) $ & 
$ ( 51 , 10 ) $ & 
$ ( 51 , 11 ) $ \\ \hline
$ ( 51 , 20 ) $ & 
$ ( 55 , 2 ) $ & 
$ ( 55 , 3 ) $ & 
$ ( 55 , 4 ) $ & 
$ ( 55 , 7 ) $ & 
$ ( 55 , 8 ) $ & 
$ ( 57 , 2 ) $ & 
$ ( 57 , 4 ) $ & 
$ ( 57 , 5 ) $ \\ \hline
$ ( 57 , 7 ) $ & 
$ ( 58 , 3 ) $ & 
$ ( 58 , 5 ) $ & 
$ ( 58 , 7 ) $ & 
$ ( 58 , 9 ) $ & 
$ ( 58 , 11 ) $ & 
$ ( 58 , 13 ) $ & 
$ ( 62 , 3 ) $ & 
$ ( 62 , 5 ) $ \\ \hline
$ ( 62 , 7 ) $ & 
$ ( 62 , 9 ) $ & 
$ ( 62 , 11 ) $ & 
$ ( 62 , 15 ) $ & 
$ ( 65 , 2 ) $ & 
$ ( 65 , 3 ) $ & 
$ ( 65 , 4 ) $ & 
$ ( 65 , 7 ) $ & 
$ ( 69 , 2 ) $ \\ \hline
$ ( 69 , 4 ) $ & 
$ ( 69 , 5 ) $ & 
$ ( 69 , 7 ) $ & 
$ ( 69 , 11 ) $ & 
$ ( 74 , 3 ) $ & 
$ ( 74 , 5 ) $ & 
$ ( 74 , 7 ) $ & 
$ ( 77 , 2 ) $ & 
$ ( 77 , 3 ) $ \\ \hline
$ ( 77 , 4 ) $ & 
$ ( 77 , 5 ) $ & 
$ ( 77 , 6 ) $ & 
$ ( 82 , 3 ) $ & 
$ ( 82 , 5 ) $ & 
$ ( 82 , 7 ) $ & 
$ ( 86 , 3 ) $ & 
$ ( 86 , 5 ) $ & 
$ ( 86 , 7 ) $ \\ \hline
$ ( 87 , 2 ) $ & 
$ ( 87 , 4 ) $ & 
$ ( 87 , 5 ) $ & 
$ ( 87 , 8 ) $ & 
$ ( 94 , 3 ) $ & 
$ ( 94 , 5 ) $ & 
$ ( 94 , 7 ) $ & 
$ ( 95 , 2 ) $ & 
$ ( 95 , 3 ) $ \\ \hline
$ ( 106 , 3 ) $ & 
$ ( 106 , 5 ) $ & 
$ ( 106 , 7 ) $ & 
$ ( 111 , 2 ) $ & 
$ ( 111 , 4 ) $ & 
$ ( 118 , 3 ) $ & 
$ ( 118 , 5 ) $ & 
$ ( 119 , 2 ) $ & 
$ ( 119 , 3 ) $ \\ \hline
$ ( 119 , 6 ) $ & 
$ ( 122 , 3 ) $ & 
$ ( 122 , 5 ) $ & 
$ ( 122 , 7 ) $ & 
$ ( 129 , 2 ) $ & 
$ ( 129 , 7 ) $ & 
$ ( 134 , 3 ) $ & 
$ ( 134 , 5 ) $ & 
$ ( 134 , 9 ) $ \\ \hline
$ ( 143 , 2 ) $ & 
$ ( 143 , 4 ) $ & 
$ ( 146 , 3 ) $ & 
$ ( 146 , 7 ) $ & 
$ ( 159 , 2 ) $ & 
$ ( 166 , 3 ) $ & 
$ ( 183 , 5 ) $ & 
$ ( 194 , 3 ) $ & 
$ ( 215 , 2 ) $ \\ \hline
$ ( 215 , 3 ) $ & 
$ ( 326 , 3 ) $ & 
$ ( 327 , 2 ) $ & 
$ ( 335 , 2 ) $ & 
$ ( 390 , 7 ) $ & & & & \\ \hline 
\caption{All $329$ pairs $(D,N)$ with $D>1$ for which we remain unsure whether $X_0^D(N)_{/\mathbb{Q}}$ has a sporadic CM point}
\label{table:unknowns_table}
\end{longtable}



\begin{thebibliography}{fszw90}
\bibliographystyle{alpha}

\bibitem[Abr96]{Abr96} D. Abramovich, \textit{A linear lower bound on the gonality of modular curves}. Internat. Math. Res. Notices, no. 20 (1996), 1005-1011. 

\bibitem[AH91]{AH91} D. Abramovich and J. Harris, \textit{Abelian varieties and curves in $W_d(C)$}. Compositio Math. 78 (1991), 227-238. 

\bibitem[AB04]{AB04} M. Alsina and P. Bayer, \textit{Quaternion orders, quadratic forms and Shimura curves}. CRM Monograph Series 22, American Mathematical Society, Providence (2004). 

\bibitem[BT07]{BT07} P. Bayer and A. Travesa, \textit{{Uniformizing functions for certain {S}himura curves, in the case {$D=6$}}}. Acta Arith. 126 (2007), 315--339.

\bibitem[BC20]{BC20} A. Bourdon and P. L. Clark, \textit{Torsion points and Galois representations on CM elliptic curves}. Pacific J. Math. 305 (2020), 43-88.  

\bibitem[BELOV19]{BELOV19} A. Bourdon, O. Ejder, Y. Liu, F. Odumodu and B. Viray, \textit{On the level of modular curves that give rise to isolated {$j$}-invariants}. Adv. Math. 357 (2019). 

\bibitem[BGRW24]{BGRW24} A. Bourdon, D. R. Gill, J. Rouse and L. D. Watson, \textit{Odd degree isolated points on {$X_1(N)$} with rational {$j$}-invariant}. Res. Number Theory 10 (2024). 

\bibitem[BN21]{BN21} A. Bourdon and F. Najman, \textit{Sporadic points of odd degree on {$X_1(N)$} coming from {$\Q$}-curves}. Preprint: arXiv:2107.10909v2 (2021). 

\bibitem[Magma]{Magma} W. Bosma, J. Cannon and C. Playoust, \emph{The Magma algebra system. I. The user language}.  J. Symbolic Comput., 24 (1997), 235–265.

\bibitem[Buz97]{Buz97} K. Buzzard, \emph{Integral models of certain Shimura curves}. Duke Math. J. 87 (1997), 591-612. 

\bibitem[Cl03]{Cl03} P. L. Clark, \textit{Rational points on Atkin-Lehner quotients of Shimura curves}. Harvard PhD. thesis (2003). 

\bibitem[Cl09]{Cl09} P. L. Clark, \textit{On the Hasse principle for Shimura curves}. Israel J. Math. 171 (2009), 349-365. 

\bibitem[Cl22]{Cl22} P. L. Clark \textit{CM elliptic curves: volcanoes, reality, and applications}. Preprint: arXiv:2212.13316 (2022). 

\bibitem[CCS13]{CCS13} P. L. Clark, B. Cook, J. Stankewicz, \textit{Torsion points on elliptic curves with complex multiplication (with an appendix by Alex Rice)}. Int. J. Number Theory 9 (2013), 447-479.  

\bibitem[CGPS22]{CGPS22} P. L. Clark, T. Genao, P. Pollack, F. Saia, \textit{The least degree of a CM point on a modular curve}, J. Lond. Math. Soc. (2) 105 (2022), 825-883.

\bibitem[CS22]{CS22} P. L. Clark and F. Saia, \textit{CM elliptic curves: volcanoes, reality, and applications, part II}. Preprint, arXiv:2212.13327 (2023). 

\bibitem[CSt18]{CSt18} P. L. Clark and J. Stankewicz, \textit{Hasse principle violations for Atkin-Lehner twists of Shimura curves}. Proc. Amer. Math. Soc. 146 (2018), 2839-2851.

\bibitem[Cox13]{Cox13} D. A. Cox, \emph{Primes of the form $x^2+ny^2$. Fermat, class field theory and complex multiplication}. Second Edition. John Wiley \& Sons, New York, (2013). 

\bibitem[DS05]{DS05} F. Diamond and J. Shurman, \textit{A first course in modular forms}. Graduate texts in mathematics, 228. Springer, New York, (2005). 

\bibitem[Fou01]{Fou01} M. Fouquet, \textit{Anneau d’endomorphismes et cardinalit\'{e} des courbes elliptiques: aspects algorithmiques}. \'{E}cole Polytechnique PhD. thesis, (2002). 

\bibitem[FM02]{FM02} M. Fouquet and F. Morain, \textit{Isogeny volcanoes and the SEA algorithm}. Algorithmic number theory (Sydney, 2002), 276-291, Lecture Notes in Comput. Sci., 2369, Springer, Berlin (2002).

\bibitem[Frey94]{Frey94} G. Frey, \textit{Curves with infinitely many points of fixed degree}. Israel J. Math. 85, (1994), 79-83.

\bibitem[GR06]{GR06} J. Gonz\'{a}lez and V. Rotger, \emph{Non-elliptic Shimura curves of genus one}. J. Math. Soc. Japan 58 (2006), 927-948. 

\bibitem[GY17]{GY17} J. Guo and Y. Yang, \emph{Equations of hyperelliptic Shimura curves}. Compos. Math. 153 (2017), 1-40.

\bibitem[Jor81]{Jor81} B. W. Jordan, \emph{On the Diophantine arithmetic of
Shimura curves}. Harvard PhD. thesis, (1981).

\bibitem[Kan11]{Kan11} E. Kani, \emph{Products of CM elliptic curves}. Collect. Math. 62 (2011), 297-339. 

\bibitem[Kat75]{Kat75} T. Katsura, \emph{On the structure of singular abelian varieties.} Proc. Japan Acad. 51 (1975), 224--228. 

\bibitem[KS03]{KS03} H. H. Kim, \emph{Functoriality for the exterior square of {${\rm GL}_4$} and the symmetric fourth of {${\rm GL}_2$}}, \emph{With appendix 1 by Dinakar Ramakrishnan and appendix 2 by Kim and Peter Sarnak}. J. Amer. Math. Soc. 16 (2003), 139-183. 

\bibitem[Koh96]{Koh96} D. Kohel, \emph{Endomorphism rings of elliptic curves over finite fields}. University of California at Berkeley PhD. thesis (1996). 

\bibitem[La75]{La75} H. Lange, \emph{Produkte elliptischer Kurven}. Nachr. Akad. Wiss. Göttingen Math.-Phys. Kl. II, (1975), 95-108. 

\bibitem[Mil72]{Mil72} J. Milne, \emph{Abelian varieties defined over their fields of moduli. I}. Bull. London Math. Soc. 4 (1972), 370-372. 

\bibitem[Naj16]{Naj16} F. Najman, \textit{Torsion of rational elliptic curves over cubic fields and sporadic points on {$X_1(n)$}}. Math. Res. Lett. 23 (2016), 245--272. 

\bibitem[Ogg83]{Ogg83} A. P. Ogg, \textit{Real points on Shimura curves}. Arithmetic and geometry, Vol. I, 277–307. Progr. Math., 35, Birkh\"{a}user Boston, Boston, MA (1983). 


\bibitem[PS23]{PS23} O. Padurariu and C. Schembri, \textit{Rational points on Atkin-Lehner quotients of geometrically hyperelliptic Shimura curves}. Expo. Math. 43, (2023), 492--513. 

\bibitem[PS24]{PS24} O. Padurariu and F. Saia, \textit{Bielliptic Shimura curves $X_0^D(N)$ with nontrivial level}. Preprint, arXiv:2401.08829v2 (2024). 

\bibitem[Rot02]{Rot02} V. Rotger, \emph{On the group of automorphisms of Shimura curves and applications}. Compositio Math. 132, (2002), 229-241. 

\bibitem[Rot04]{Rot04} V. Rotger, \emph{Shimura curves embedded in {I}gusa's threefold.} Modular curves and abelian varieties, Progr. Math. 224 (2004), 263--276. 

\bibitem[RSY05]{RSY05} V. Rotger, A. Skorobogatov and A. Yafaev, \emph{Failure of the Hasse principle on Atkin-Lehner quotients of Shimura curves over $\mathbb{Q}$}. Mosc. Math. J. 5 (2005), 463-476, 495. 

\bibitem[Rep]{Rep} F. Saia, \emph{CM-Points-Shimura-Curves} Github Repository

\url{https://github.com/fsaia/CM-Points-Shimura-Curves} (2023). 

\bibitem[Sc92]{Sc92} C. Schoen, \emph{Produkte Abelscher Variet\"{a}ten und Moduln \"{u}ber Ordnungen}. J. Reine Angew. Math. 429 (1992), 115-123. 

\bibitem[Sh66]{Sh66} G. Shimura, \emph{Moduli and fibre systems of abelian varieties}. Ann. of Math. (2) 83 (1966), 294--338. 

\bibitem[Sh67]{Sh67} G. Shimura, \emph{Construction of class fields and zeta functions of algebraic curves}. Ann. of Math. (2) 85 (1967), 58-159.

\bibitem[Sh72]{Sh72} G. Shimura, \emph{On the field of rationality for an abelian variety}. Nagoya Math. J. 45 (1972), 167-178. 

\bibitem[Sh75]{Sh75} G. Shimura, \emph{On the real points of an arithmetic quotient of a bounded symmetric domain}. Math. Ann. 215 (1975), 135-164. 

\bibitem[SM74]{SM74} T. Shioda and N. Mitani, \emph{Singular abelian surfaces and binary quadratic forms}. Classification of algebraic varieties and compact complex manifolds, pp. 259–287. Lecture Notes in Math., Vol. 412, Springer, Berlin (1974).

\bibitem[SS16]{SS16} S. Siksek and A. Skorobogatov, \emph{On a Shimura curve that is a counterexample to the Hasse principle}. Bull. London Math. Soc. 35 (2003), 409-414. 

\bibitem[Smi23]{Smi23} H. Smith, \emph{Ramification in division fields and sporadic points on modular curves}. Res. Number Theory 9 (2023). 

\bibitem[St01]{St01} P. Stevenhagen, \emph{Hilbert’s 12th problem, complex multiplication and Shimura reciprocity}. Class field theory -- its centenary and prospect, pp. 161-176. Adv. Stud. Pure Math. 30, Math. Soc. Japan, Tokyo (2001). 

\bibitem[Sut13]{Sut13} D. Sutherland, \emph{Isogeny volcanoes}. ANTS X--Proceedings of the Tenth Algorithmic Number Theory Symposium, 507-530, Open Book Ser., 1, Math. Sci. Publ., Berkeley, CA (2013). 


\bibitem[Uf10]{Uf10} D. Ufer, \emph{Shimura-Kurven, Endomorphismen und
$q$-Parameter.} Universit\"at Ulm PhD. thesis (2010). 

\bibitem[Vig80]{Vig80} M. Vign\` eras. \textit{Arithm\` etique des alg\' ebres de quaternions}. Lecture Notes in Mathematics 800, Springer, Berlin (1980).

\bibitem[Voi09]{Voi09} J. Voight, \textit{Shimura curves of genus at most two}. Math. Comp. 78 (2009), 1155-1172. 

\bibitem[Voi21]{Voi21} J. Voight, \textit{Quaternion Algebras}. Graduate Texts in Mathematics 288, Springer, Cham (2021). 

\end{thebibliography}
\end{document}